\documentclass[11pt,a4paper]{article}

\usepackage[T1]{fontenc}
\usepackage[utf8]{inputenc}
\usepackage{lmodern}
\usepackage{microtype}
\usepackage{amsmath,amssymb,amsthm,mathtools}
\allowdisplaybreaks
\usepackage{aliascnt}
\usepackage{enumitem}
\usepackage{geometry}
\usepackage{hyperref}
\usepackage[nameinlink,capitalize]{cleveref}
\usepackage{xfrac}

\hypersetup{
  hidelinks,
  pdftitle={Global Strong Solutions for Maxwell--Stefan Diffusion with Additive Friction Coefficients},
  pdfauthor={Dieter Bothe},
  pdfsubject={Global strong solvability and structural analysis of normalized Maxwell--Stefan diffusion with additive friction coefficients},
  pdfkeywords={Maxwell--Stefan diffusion, additive friction, global strong solutions, root coordinates, multi-EPD truncations, De Giorgi regularity, entropy symmetrization}
}
\setlist[itemize]{leftmargin=2em}
\setlist[enumerate]{leftmargin=2em}

\newtheorem{theorem}{Theorem}[section]

\newaliascnt{proposition}{theorem}
\newtheorem{proposition}[proposition]{Proposition}
\aliascntresetthe{proposition}

\newaliascnt{lemma}{theorem}
\newtheorem{lemma}[lemma]{Lemma}
\aliascntresetthe{lemma}

\newaliascnt{corollary}{theorem}
\newtheorem{corollary}[corollary]{Corollary}
\aliascntresetthe{corollary}

\newaliascnt{definition}{theorem}
\newtheorem{definition}[definition]{Definition}
\aliascntresetthe{definition}

\newaliascnt{assumption}{theorem}

\aliascntresetthe{assumption}

\theoremstyle{remark}
\newaliascnt{remark}{theorem}
\newtheorem{remark}[remark]{Remark}
\aliascntresetthe{remark}

\newaliascnt{example}{theorem}

\aliascntresetthe{example}

\crefname{lemma}{Lemma}{Lemmas}
\Crefname{lemma}{Lemma}{Lemmas}
\crefname{proposition}{Proposition}{Propositions}
\Crefname{proposition}{Proposition}{Propositions}
\crefname{theorem}{Theorem}{Theorems}
\Crefname{theorem}{Theorem}{Theorems}
\crefname{corollary}{Corollary}{Corollaries}
\Crefname{corollary}{Corollary}{Corollaries}
\crefname{definition}{Definition}{Definitions}
\Crefname{definition}{Definition}{Definitions}
\crefname{assumption}{Assumption}{Assumptions}
\Crefname{assumption}{Assumption}{Assumptions}
\crefname{remark}{Remark}{Remarks}
\Crefname{remark}{Remark}{Remarks}
\crefname{example}{Example}{Examples}
\Crefname{example}{Example}{Examples}

\newcommand{\R}{\mathbb R}

\newcommand{\dd}{\,\mathrm d}
\newcommand{\Div}{\operatorname{div}}
\newcommand{\diag}{\operatorname{diag}}
\newcommand{\co}{\operatorname{co}}
\newcommand{\one}{\mathbf 1}

\newcommand{\cD}{\mathcal D}

\newcommand{\cT}{\mathcal T}
\newcommand{\cS}{\mathcal S}
\newcommand{\cZ}{\mathcal Z}

\newcommand{\cK}{\mathcal K}

\newcommand{\sym}{\operatorname{sym}}

\newcommand{\hz}{{h_z}}

\title{Global Strong Solutions for Maxwell--Stefan Diffusion with Additive Friction Coefficients}
\author{Dieter Bothe\thanks{Mathematical Modeling and Analysis, Department of Mathematics, Technische Universit\"at Darmstadt, Peter-Gr\"unberg-Stra\ss e 10, 64287 Darmstadt, Germany. E-mail: \texttt{bothe@mma.tu-darmstadt.de}.}}

\begin{document}
\maketitle

\begin{abstract}
We study Maxwell--Stefan diffusion with additive friction coefficients $f_{ij}=g_i+g_j$,
$0<g_1<\cdots<g_N$. In mass fractions, the system isolates the constrained pair-friction block;
in mole fractions, it is the classical ideal
isothermal/isobaric Maxwell--Stefan system at constant total molar concentration. 
Additivity makes the constrained pair-friction dissipation
species-diagonal; conversely, species-diagonality on one interior barycentric
constraint space forces a pair-sum representation. At operator level, the
positive constrained relaxation operator is a scalar shift of a compression of
$G=\diag(g_1,\ldots,g_N)$. Its scalar resolvent yields both an explicit
constrained inverse and interlacing spectral roots, which form global
real-analytic coordinates on the open simplex and whose differentials are left
eigen-covectors. In root coordinates the principal part is
diagonal, no self-square gradient term occurs, and scalar comparison yields invariant rectangles
and separation from the simplex boundary. For regularity we introduce entropy-stabilized
one-sided multi-EPD truncations: Euler--Poisson--Darboux entropies cancel mixed quadratic
production, while for $N\ge4$ a truncation-weighted mixing-entropy correction supplies
transverse coercivity. Caccioppoli and logarithmic estimates, shrinking, and critical mass yield
H\"older continuity up to the Neumann boundary. {The mixing entropy also symmetrizes the moment system;}
frozen conormal estimates give spatial Lipschitz bounds. Together with time H\"older control and
short-interval maximal regularity, this yields global strong solvability on bounded
$C^{2+\alpha}$ domains {with $0<\alpha<1$}, for each $N\ge2$, $d\ge2$, and $p>d+2$, for all uniformly positive
initial concentrations in the natural trace class satisfying Neumann compatibility. Solutions
become classical for positive times and converge exponentially to equilibrium in relative
entropy, $L^2$, and $C^1(\overline\Omega)$.
\end{abstract}



\section{Introduction and main results}

\subsection{Maxwell--Stefan diffusion and the additive setting}

Multicomponent diffusion is a classical setting in which scalar Fickian diffusion
\cite{Fick1855} is inadequate.  {Relative motion of distinct constituents
produces interspecies friction, and the resulting cross-diffusion may produce
uphill diffusion, osmotic diffusion, and diffusion barriers.}  The Maxwell--Stefan
approach, originating in Maxwell's kinetic theory of gases and Stefan's continuum
model for gas mixtures \cite{Maxwell1867,Stefan1871}, balances thermodynamic
driving forces against pairwise interspecies friction.  It is standard in
physical chemistry and chemical engineering; see
\cite{ErnGiovangigli1994,TaylorKrishna1993,Giovangigli1999,KrishnaWesselingh1997}.

Conservation of total mass makes the force--flux matrix singular, so the
Maxwell--Stefan relation can be inverted only after restriction to a
codimension-one constraint space.  After elimination of one component, the
resulting diffusion system is quasilinear and generally nonsymmetric.  Its
principal part is normally elliptic, but the component equations for the
fractions do not satisfy scalar maximum principles, and the standard entropy
structure by itself does not provide the pointwise regularity needed for a
global strong-solution continuation argument.  {The analysis below identifies an
additive coefficient class for which additional algebraic structure yields a
global strong-solution theory without a small-data assumption.}

{
Consider a smooth scalar state function $z=z(y)$.  From the balance law
\[
        \partial_t y+\Div J=0{,}
\]
the chain and product rules give
\begin{equation}\label{eq:intro-transformed-balance}
 \partial_t z+\Div\bigl(J^{\sf T}\nabla_y z\bigr)
 =\sum_{i=1}^N\nabla_x(\partial_{y_i}z)\cdot J_i.
\end{equation}
After insertion of the flux law, the right-hand side is quadratic in first
spatial derivatives.  Hence the principal second-order term is determined by
the transformed flux $J^{\sf T}\nabla_y z$.  A sufficient condition for a scalar
divergence-form principal part is that this flux be proportional to
$\nabla_x z$.  Since $\nabla_x y=B(y)J$, where $B(y)$ denotes the
Maxwell--Stefan matrix defined below in \eqref{eq:B-general},
\[
 \nabla_x z=(\nabla_x y)^{\sf T}\nabla_y z
              =J^{\sf T}B(y)^{\sf T}\nabla_yz.
\]
It is therefore sufficient that
\begin{equation}\label{eq:intro-integrable-left-eigenvector}
        B(y)^{\sf T}\nabla_y z=-\lambda(y)\nabla_y z,
        \qquad \lambda(y)>0,
\end{equation}
because then $J^{\sf T}\nabla_y z=-\lambda^{-1}\nabla_xz$.  {Thus the principal
second-order term in this single $z$-equation is a scalar divergence operator involving only
$\nabla_x z$.  This does not decouple the nonlinear system: lower-order gradient couplings and
coefficient dependence on the other state variables remain.  A coordinate system with diagonal
principal second-order part requires $N-1$ scalar functions $z_j(y)$ satisfying the
corresponding identities.}  In particular, the selected left eigen-covector
fields of $B(y)$ must be exact, i.e. differentials of scalar state functions.
Pointwise diagonalizability of $B(y)$ does not imply this exactness property.
The additive class considered below admits such integrable spectral
covectors; their construction is given in \Cref{sec:root-coordinates}.
}

We consider the additive class
\begin{equation}\label{eq:intro-additive-class-new}
        f_{ij}=g_i+g_j,\qquad g_i>0,\qquad i\ne j,
\end{equation}
and, for the global root-coordinate theory, assume after relabeling that
$0<g_1<\cdots<g_N$.  The additive law reduces the $N(N-1)/2$ pair parameters
to $N$ species parameters when $N\ge4$.  It is a constitutive subclass, not a
universal consequence of molecular kinetic theory.  {In the molar formulation,
the parameters $g_i$ admit a species-specific momentum-exchange interpretation:
the pair coefficient $f_{ij}$ is the sum of the two corresponding rates.}  A
complementary first-order approximation interpretation is given below.

The normalized system studied here isolates the constrained Maxwell--Stefan
diffusion block.  We use normalized mass fractions $y_i$ and, in
\eqref{eq:MS-law}, the simplified driving force $\nabla y$.  This keeps both
the simplex constraint and the barycentric mass-flux constraint free of
molar-mass weights, but $\nabla y$ should not be confused with the general
thermodynamic driving force.  For a classical isothermal ideal mixture with
molar masses $M_i>0$, let $\rho$ and $c$ denote the total mass density and
total molar concentration, respectively; thus $\rho_i=\rho y_i$.  Converting the mass fractions $y_i$ to mole
fractions gives
\[
 x_i(y)=\frac{\rho}{c}\frac{y_i}{M_i},\qquad
 \frac{\rho}{c}=\left(\sum_{k=1}^N\frac{y_k}{M_k}\right)^{-1}.
\]
With the mass-specific chemical potentials
$\mu_i=\mu_i^\circ(T,p)+(RT/M_i)\log x_i$, the standard projected force
$\rho_i(RT)^{-1}(\nabla\mu_i-\sum_k y_k\nabla\mu_k)$ reduces, by the
Gibbs--Duhem relation and after division by the total mass density $\rho$,
to
\[
 \begin{gathered}
 \Gamma_M(y)\nabla y,\qquad
 \Gamma_M(y):=D_M-x(y)\otimes {m_M}^{\sf T},\\
 D_M:=\diag(M_1^{-1},\ldots,M_N^{-1}),\qquad
 {m_M}:=(M_i^{-1})_{i=1}^N,
 \end{gathered}
\]
since $(c/\rho)\nabla x=\Gamma_M(y)\nabla y$.  Writing
$E:=\{v\in\R^N:\sum_i v_i=0\}$, one has
$\Gamma_M(y)\R^N\subset E$ and
$\ker\Gamma_M(y)=\operatorname{span}\{y\}$, so
$\Gamma_M(y)|_E$ is invertible; for equal molar masses it is simply
$M^{-1}I_E$.  Thus the same singular pair-friction structure and constrained
inversion remain, while the thermodynamic factor generally changes the
effective diffusion matrix
\cite{BotheDreyer2015,HerbergMeyriesPruessWilke2017,BotheDruet2023}.

There is also an exact realization of the normalized system in the classical
molar formulation.  For an ideal isothermal and isobaric mixture with constant
total molar concentration, let $\mathcal J_i$ denote the molar diffusion fluxes
relative to the molar-average velocity and {assume that the molar-average velocity vanishes}.  With mole fractions $x_i$ and normalized fluxes
$j_i:=\mathcal J_i/c$, {the species balances and Maxwell--Stefan}
law become
\[
 \partial_t x_i+\Div j_i=0,\qquad
 \sum_i j_i=0,\qquad
 \nabla x_i=\sum_{k\ne i}f_{ik}(x_i j_k-x_kj_i),
\]
with {$f_{ik}=(\mathcal D_{ik}^{\mathrm{MS}})^{-1}$}; see \cite{Bothe2011}.  Thus the present
PDE is a normalized model for general thermodynamic Maxwell--Stefan systems
and, with $x$ in place of $y$, exactly the ideal mole-fraction Maxwell--Stefan
system in the molar-average frame.

\subsection{\texorpdfstring{{Main result and proof structure}}{Main result and proof structure}}

The main global theorem, \Cref{thm:global-strong-additive}, gives maximal-$L^p$
strong solvability on bounded $C^{2+\alpha}$ domains, $0<\alpha<1$, for each
$N\ge2$, $d\ge2$, and $p>d+2$, without a smallness condition on the initial
oscillation or on the range of the spectral variables.  Equivalently, the
concentration data are those in the natural trace space that are uniformly
positive on $\overline\Omega$, satisfy $\sum_i y_i^0=1$, and satisfy the
homogeneous Neumann compatibility; see
\Cref{cor:global-positive-y-data,rem:initial-trace-equivalence}.  The solution
remains in a compact subset of the open composition simplex, is classical on
every positive-time strip, and converges exponentially to the homogeneous
equilibrium.

The proof uses complementary state coordinates.  First, the
additive law makes the pair-friction dissipation species-diagonal on the
barycentric constraint and yields an explicit constrained inverse; conversely,
constrained species-diagonality characterizes the real pair-sum form.  The
interlacing roots $z_1,\ldots,z_{N-1}$ of
$m_y(z)=\sum_i y_i/(g_i-z)$ are global spectral coordinates.  Their differentials
are left eigen-covectors of the Maxwell--Stefan matrix, the root principal part is
diagonal, and no $j$-th equation contains a self-square
$|\nabla z_j|^2$.  Scalar comparison therefore gives invariant root rectangles
and compact separation from the simplex boundary.

Second, one-sided multi-EPD truncation entropies cancel the mixed root-gradient
production in a selected level set.  For $N\ge4$, a truncation-weighted mixing
entropy supplies the transverse coercivity needed for the localized
Caccioppoli and logarithmic estimates.  Shrinking and critical mass then yield
space--time H\"older continuity of every root up to the homogeneous Neumann
boundary, without a prior scalar Schauder estimate.

Finally, the affine moments $U_\ell=\sum_i g_i^\ell y_i$ close in divergence
form, and the Hessian of the same mixing entropy symmetrizes their diffusion
matrix.  Root H\"older continuity permits this metric to be frozen locally;
weak-solution conormal estimates give a spatial Lipschitz bound, while the time
H\"older modulus controls short terminal intervals in maximal regularity.  A
finite freezing argument propagates the natural trace norm and excludes a
finite maximal existence time.  Thus the mixing entropy has two distinct
roles: it supplies transverse coercivity in root coordinates and the system
symmetrizer in moment coordinates.  No corresponding global strong result is
asserted here for general non-additive coefficients when $N\ge4$.

\subsection{Relation to existing theory}

Early rigorous strong-solution results include the work of Giovangigli and
Massot, who proved local existence and uniqueness for the Cauchy problem for a
detailed multicomponent reactive-flow model in full vibrational nonequilibrium
by an entropic symmetrization and a hyperbolic--parabolic normal form
\cite{GiovangigliMassot1998Local}; their companion paper established asymptotic
stability of equilibrium states \cite{GiovangigliMassot1998Stability}.
Giovangigli's monograph also contains global classical solvability and
convergence near an isolated positive equilibrium for reversible mass-action
kinetics \cite[Theorem~9.7.4]{Giovangigli1999}.  For nonideal supercritical
multicomponent reactive fluids, Giovangigli and Matuszewski derived local
symmetrized and hyperbolic--parabolic normal forms and proved global existence
and asymptotic stability near a thermodynamically stable equilibrium
\cite{GiovangigliMatuszewski2013}.  These works concern complete reactive-flow
models on the whole space, rather than the bounded-domain isothermal diffusion
system considered here.

In \cite{Bothe2011} we proved normal ellipticity for the isothermal, isobaric
Maxwell--Stefan diffusion system with vanishing molar-average velocity and
explained how quasilinear parabolic theory yields local strong well-posedness.
Herberg, Meyries, Pr\"uss, and Wilke subsequently gave a complete maximal-$L^p$
analysis of a mass-based reactive Maxwell--Stefan system on a bounded domain
\cite{HerbergMeyriesPruessWilke2017}.  Under constant total density, vanishing
barycentric velocity, and constant temperature, they proved local
well-posedness, a local semiflow, instantaneous regularization, and strict
positivity of every initially nontrivial component.  For reversible
mass-action kinetics they identified the positive-equilibrium manifold and
proved normal stability.  Their projected chemical-potential driving forces
sum to zero, and the stability analysis uses the ideal-mixture Helmholtz free
energy relative to a positive chemical equilibrium.  The model is a
bounded-domain reactive diffusion system without Navier--Stokes coupling.

Marion and Temam proved global weak existence and positivity for fully nonlinear
multispecies reaction--\allowbreak diffusion systems whose diffusion coefficients
are obtained from the Stefan--Maxwell relations \cite{MarionTemam2015}.  Other
strong-solution results concern special, perturbative, or coupled-flow
settings.  Boudin, Grec, and Salvarani proved global smooth solvability and
convergence in a coefficient-restricted system with two equal binary coefficients
\cite{BoudinGrecSalvarani2012}, while Bondesan and Briant established a
quantitative perturbative Cauchy theory for a flux-incompressible system
\cite{BondesanBriant2022}.  Piasecki, Shibata, and Zatorska proved local and,
near a constant state, global strong well-posedness for a compressible
two-component mixture in an $L^p$--$L^q$ framework
\cite{PiaseckiShibataZatorskaStrong2019}, as well as local existence and maximal
$L^p$--$L^q$ regularity for an isothermal compressible multicomponent flow
coupled to cross-diffusion \cite{PiaseckiShibataZatorska2019}.  In later joint
work with Druet, we obtained local strong well-posedness and near-equilibrium
strong-solution results for compressible and incompressible multicomponent flow
models with Fick--Onsager or Maxwell--Stefan diffusion
\cite{BotheDruetCompressible2021,BotheDruetIncompressible2021}.  Druet proved
short-time existence for the full nonisothermal compressible system in a mixed
parabolic--hyperbolic maximal-regularity class and, under additional growth
assumptions, excluded breakdown caused solely by blow-up or vanishing of the
temperature or a partial mass density \cite{DruetFull2022}.  These coupled-flow
results do not apply directly to the quiescent isothermal normalized system
considered below.

For arbitrary positive constant binary diffusivities, J\"ungel and Stelzer
established global bounded weak solutions by entropy variables and the
boundedness-by-entropy method \cite{JungelStelzer2013}.  For reactive
Maxwell--Stefan systems, Daus, J\"ungel, and Tang proved global bounded weak
solutions and exponential convergence to equilibrium under detailed-balance or
complex-balance assumptions \cite{DausJungelTang2020}.  Mucha, Pokorn\'y, and
Zatorska proved global weak existence on the three-dimensional torus for a
heat-conducting compressible reactive mixture in which the
Navier--Stokes--Fourier equations are coupled to the full Maxwell--Stefan
relations, under particular constitutive assumptions
\cite{MuchaPokornyZatorska2015}.  Other extensions of the weak theory include
pressure-driven diffusion, compressible and incompressible flow, and
heat-conducting mixtures; see, among others,
\cite{DruetJungel2020,Druet2021,ChenJungel2015,HelmerJungel2021,
GeorgiadisJungel2024}.  Relative-entropy methods yield weak--strong uniqueness
for Maxwell--Stefan systems \cite{HuoJungelTzavaras2022}; Geltner,
Heitzinger, and J\"ungel relaxed the positivity and regularity assumptions on
the strong solution within a general cross-diffusion framework that includes
Maxwell--Stefan diffusion \cite{GeltnerHeitzingerJungel2026}.  Braukhoff,
Raithel, and Zamponi proved partial H\"older regularity for weak solutions of a
broad class of entropy-structured cross-diffusion systems that includes
Maxwell--Stefan diffusion \cite{BraukhoffRaithelZamponi2022}.  Their entropy
modification is used in a Campanato partial-regularity framework; the
level-dependent one-sided multi-EPD construction here serves a different
purpose, namely scalar localization of individual spectral components in the
presence of quadratic cross-root drift.  {For classical uses of
EPD equations in entropy constructions for hyperbolic conservation laws and
in hydrodynamic-type integrable systems, see
\cite{ChenLeFloch1997,KodamaKonopelchenkoSchief2015}; here the EPD identities
instead cancel mixed gradient production in a parabolic level-set argument, that is,
an argument based on one-sided super- and sublevel truncations of an individual root.}

Further work addresses renormalized uniqueness, entropy equality,
homogenization, degenerate thermal conductivity, and the whole-space problem
\cite{GeorgiadisKimTzavaras2024,BerselliGeorgiadisTzavaras2025,
NocitaSalvarani2025,GeorgiadisDegenerate2026,GeorgiadisWholeSpace2026}.
Kinetic derivations of reactive Maxwell--Stefan systems were obtained for an
isothermal simple-reacting-sphere model and for non-isothermal polyatomic gas
mixtures in \cite{AnwasiaGoncalvesSoares2019,
AnwasiaBisiSalvaraniSoares2020}.  For nonreactive mixtures, Boudin, Grec, and
Pavan derived the Maxwell--Stefan diffusion limit for general collision cross
sections \cite{BoudinGrecPavan2017}.  In a non-isothermal setting on the
three-dimensional torus, Chen, Jiang, and Luo obtained a perturbative
global classical theory for sufficiently small deviations from a positive
constant state as part of a rigorous multispecies Boltzmann diffusion-limit
analysis \cite{ChenJiangLuo2025}.  Further results on kinetic limits and diffusion
asymptotics are given in \cite{BoudinGrecSalvarani2015,
HutridurgaSalvarani2017,HutridurgaSalvarani2018,BondesanBriant2021,
JungelPollinoTaguchi2026}.

\subsection{Additive dissipation, root coordinates, and moment coordinates}

{
For additive coefficients, the constrained pair-friction quadratic form
admits a species-diagonal representation; \Cref{prop:diagonal-barycentric-dissipation}
shows that this property characterizes additive pair coefficients.  The same
algebra yields the compressed-diagonal form, its diagonal-plus-rank-one
resolvent, and the spectral representations of the constrained Maxwell--Stefan
operator and its inverse.  The relation to the
core-diagonal closure of Bothe and Druet \cite{BotheDruet2023} is formulated at
the level of diffusion entropy production in
\Cref{subsec:core-diagonal-relation}: additivity gives a species-diagonal
representation of the constrained resistance quadratic form, whereas the
core-diagonal closure prescribes species-diagonal coefficients on the mobility
side before the conservation projection.  The two coefficient classes are not
algebraically equivalent because diagonality is not preserved by constrained
projection and inversion.
}

Let $G=\diag(g_1,\ldots,g_N)$, $s=g\cdot y$, and let
$\Pi_yv:=v-y(\one\cdot v)$ be the projection onto
$E=\{v:\one\cdot v=0\}$ along $\operatorname{span}\{y\}$.  Then
\begin{equation}\label{eq:intro-compressed-E}
        -B(y)|_E=\Pi_yG|_E+sI_E.
\end{equation}
Thus, for additive friction, the positive constrained relaxation operator
$-B(y)|_E$ is a scalar shift of the compression $\Pi_yG|_E$ of the diagonal
species operator $G$.  For $y$ in the simplex interior,
$Y=\diag(y_1,\ldots,y_N):E_y\to E$ is an isomorphism, and this compression
is similar to the weighted-orthogonal compression $P_yG|_{E_y}$, where
$P_yv=v-\one(y\cdot v)$; indeed $\Pi_yY=YP_y$ and $YG=GY$.  The associated
scalar resolvent function
\begin{equation}\label{eq:intro-cauchy-new}
        m_y(z)=\one^{\sf T}(G-zI)^{-1}y
              =\sum_{i=1}^N\frac{y_i}{g_i-z}
\end{equation}
provides the three representations used below: evaluation at $z=-s$ gives the
constrained rank-one inverse, its zeros are the spectral root coordinates, and
its Laurent coefficients at infinity are the affine moments.  Equivalently,
the roots are the eigenvalues of the compressed operator, whereas the moments
determine the coefficients of its characteristic polynomial.

For $y$ in the open simplex, $m_y$ has exactly one zero
$z_j$ in every gap $(g_j,g_{j+1})$, and the roots satisfy $g_1<z_1<g_2<\cdots<z_{N-1}<g_N$.  The map $y\mapsto z$ is a real-analytic diffeomorphism from the open simplex
onto the full interlacing box
$\prod_{j=1}^{N-1}(g_j,g_{j+1})$, with explicit inverse
\[
 y_i(z)=
\prod_{j=1}^{N-1}(g_i-z_j) \left/
      \prod_{m\ne i}^{\mbox{$\;$}}(g_i-g_m).\right.
\]
The corresponding relaxation eigenvalues are $\lambda_j=s+z_j$.  The
differentials $D_yz_j$ are left eigen-covectors of the Maxwell--Stefan
operator, and the root coordinate frame diagonalizes the principal diffusion
operator.

Writing $e_1:=\sum_{i=1}^N g_i$ and $a_j:=\lambda_j^{-1}$, the resulting
root equations have the form
\begin{align*}
 \partial_tz_j-\Div(a_j(z)\nabla z_j)
 +\sum_{k\ne j}C_{jk}(z)\nabla z_j\cdot\nabla z_k&=0,\\[0.5ex]
 C_{jk}=\frac{a_j+a_k}{z_j-z_k}=-C_{kj},\qquad
 \lambda_j&=e_1-\sum_{\ell\ne j}z_\ell.
\end{align*}
Thus $a_j$ is independent of the selected root $z_j$.  In particular, the
$j$-th equation contains no term proportional to $|\nabla z_j|^2$; every
quadratic correction contains $\nabla z_j$ together with another root
gradient.  Because every quadratic correction in the $j$-th equation contains
$\nabla z_j$, the equation admits the componentwise comparison argument used in
the invariant-root rectangle theorem.  Once the initial root values lie in a compact interlacing
rectangle, the whole subsequent trajectory remains in that rectangle.  Via
the explicit inverse root map this yields compact separation of every
composition component from zero and excludes approach to the simplex boundary
{on the strong-existence interval.}

The second coordinate system is affine rather than spectral: set
$U_\ell=\sum_i g_i^\ell y_i$, $\ell=1,\ldots,N-1$.  The moment equations close in divergence form.  Their diffusion matrix is, in
the interior, exactly similar to the inverse weighted relaxation operator and
therefore has eigenvalues $1/\lambda_1,\ldots,1/\lambda_{N-1}$.  It extends
analytically to the closed moment simplex and remains normally elliptic there.
In the interior the Hessian of the mixing entropy symmetrizes it; see
\Cref{prop:moment-PDE-spectrum,prop:closed-simplex-normal-ellipticity,thm:entropy-symmetrization}.  Accordingly, the roots are used for scalar comparison, while the affine moments
provide divergence structure and the entropy metric used for the system bootstrap.

\subsection{The entropy-stabilized one-sided multi-EPD truncation method}

For a state function $\Psi(z)$, testing the root equations by $D\Psi(z)$
produces the quadratic entropy-production form
$\mathfrak Q_\Psi(z;\nabla z)$ defined below.  Its mixed $(j,k)$ contribution contains
$\Psi_{jk}+(\Psi_j-\Psi_k)/(z_j-z_k)$, so the pairwise multi-EPD identities
$(z_k-z_j)\Psi_{jk}=\Psi_j-\Psi_k$ are exactly the cancellation equations for
mixed quadratic production.
{For an individual root $z_i$ and a level $k\in(g_i,g_{i+1})$,
\Cref{sec:multi-EPD-dossier} introduces an explicit level-dependent potential
$\Phi_{i,k}$ and its upper and lower one-sided branches $\Phi_{i,k}^{\pm}$.  The exact
integral formula is deferred to that section; the properties relevant here are that the branches
are $C^{1,1}$ across $z_i=k$, are comparable to $(z_i-k)_\pm^2$, satisfy the multi-EPD
system on their active side, and have vanishing first derivatives at the truncation interface.
Moreover, $(\Phi_{i,k})_{jj}=0$ for $j\ne i$ and
$(\Phi_{i,k})_{ii}>0$.  Hence the pure production controls only the selected
truncation gradient,
$\mathfrak Q_{\Phi_{i,k}^\pm}\simeq|\nabla(z_i-k)_\pm|^2$, with all mixed
production terms canceled exactly.

The pure truncation does not control the transverse-root flux terms in the
localized entropy identity, which have the sizes $r|\nabla r|$ and
$r^2|\nabla z_j|$, where $r=(z_i-k)_\pm$.  With two root variables, such an integration by parts introduces no
third-root derivative.  With three or more root variables, it differentiates
coefficients depending on additional roots and hence produces gradients that
are not controlled by the pure truncation.

{Let $\hz(z):=\sum_{\ell=1}^N y_\ell(z)\log y_\ell(z)$ denote the physical
mixing entropy expressed in root variables.  It satisfies the same pairwise multi-EPD identities and has strictly positive
diagonal production coefficients, so its entropy production controls all root gradients.  Choose
$C_h$ so that $\widehat h:=\hz+C_h>0$ on the compact root range under consideration.  With
$\chi=r^2/2$, set $\Psi_{i,k}^\pm=\Phi_{i,k}^\pm+\varepsilon\chi\widehat h$.}
{We refer to this truncation-weighted mixing-entropy correction as the stabilization.
It supplies the required transverse-gradient control without a further integration by parts in
{a transverse root}.}
The factor $\widehat h$ is an exact multi-EPD entropy, while the product
{$\chi\widehat h$} has a controlled product defect.  Its production decomposes into
a positive piece {$\chi\mathfrak Q_{\hz}$}, a term in the selected direction, and
explicit cross terms of lower truncation order.  Choosing $\varepsilon>0$
small enough allows the latter terms to be absorbed by the coercivity of the
pure truncation and yields
\[
 \mathfrak Q_{\Psi_{i,k}^\pm}
 \ge c|\nabla r|^2
   +c\varepsilon r^2\sum_{j\ne i}|\nabla z_j|^2.
\]
{This coercivity yields a one-sided scalar Caccioppoli inequality for each root
without differentiating an uncontrolled {transverse-root} coefficient}; see
\Cref{prop:root-Caccioppoli-dossier}.

The weighted transverse term is also matched to the logarithmic propagation
step.  The scaled logarithmic-square function is chosen so that the ratio of
its squared first derivative to its second derivative is bounded by a constant
times $r^2$.  {Young's inequality bounds each cross-root drift contribution by a term
absorbed into the selected-gradient dissipation and a constant multiple of
$r^2|\nabla z_j|^2$; the latter is controlled by the stabilized transverse
dissipation.}
These bounds give the logarithmic estimate of
\Cref{lem:root-logarithmic-estimate}.  The subsequent shrinking and critical
mass alternatives are scalar, and yield the boundary H\"older estimate of
\Cref{thm:root-DeGiorgi-holder}.  {The multi-EPD construction therefore yields a level-dependent scalar
localization whose dissipation controls both the localized flux terms and the
weighted cross-gradient term in the logarithmic estimate.}
}

\subsection{From root H\"older control to global continuation}

{The De Giorgi argument yields a positive H\"older exponent for the root
variables.}  Once $z$ is H\"older continuous on positive-time strips, every smooth state
coordinate, in particular the affine moments $U$, inherits the same H\"older continuity.  On a
sufficiently small cylinder the state therefore stays in a fixed neighborhood
of a reference value $U_*$.  Freezing the mixing-entropy Hessian at $U_*$
turns the state-dependent entropy metric into a constant strong symmetrizer for
the moment system on that cylinder.  Weak-solution conormal gradient estimates
then yield a uniform spatial Lipschitz bound for $U$, and hence for the root
and entropy variables as well.

{The spatial Lipschitz estimate and the De Giorgi time H\"older estimate provide
the bounds required for continuation.}  In entropy variables the equation can
be written as a normally elliptic nondivergence system with a quadratic
first-gradient right-hand side.  The Lipschitz bound controls that forcing,
while the time H\"older modulus makes the family of linearized operators
H\"older continuous in time.  Freezing the operator successively on finitely
many short terminal intervals gives a uniform maximal-regularity estimate and
propagates the natural maximal-$L^p$ trace norm, defined in Section~10, uniformly up to any finite endpoint of the strong-existence interval.
The local theory, whose restart time is uniform for data with the resulting compact
state range and trace-norm bound, can then be restarted, contradicting finite maximal time.  The continuation argument uses the root H\"older estimate, the spatial
Lipschitz bound, and the short-interval maximal-regularity estimate.  The
{higher conormal gains, expressed with derived positive H\"older exponents, and in
particular the second-order $C^{2+\beta}$ system regularity, are used subsequently for
positive-time smoothing.}

After global strong solvability has been established, positive-time conormal
regularity gives the stated classical smoothing.  The ordinary mixing entropy
is a Lyapunov functional; the boundary-uniform Fisher-information bound and the
bounded-domain logarithmic Sobolev inequality imply exponential relaxation in
relative entropy and $L^2$.  Uniform positive-time smoothing then upgrades the
convergence to $C^1(\overline\Omega)$.

\subsection{Framework and organization}

The algebraic and coordinate results in Sections~2--9 are formulated for each
$N\ge2$; their PDE applications use bounded $C^2$ domains.  The local
maximal-$L^p$ theory (Section~10), the strong truncation argument (Section~11),
and boundary De Giorgi regularity (Section~12) assume $p>d+2$ and $C^2$
boundary.  The weak-conormal boundary gains in Section~13, and hence the global
theorem in Section~14, use a bounded $C^{2+\alpha}$ domain with
$0<\alpha<1$.  Here $\alpha$ is the fixed boundary-regularity exponent, while
$\beta$ (and later $\beta_1,\beta_2$) denotes a H\"older exponent produced by
the bootstrap and may be chosen smaller than $\alpha$.  The one-sided
piecewise-$C^2$ truncation identities are justified directly in the
positive-time maximal-$L^p$ strong class by Bochner/Sobolev chain rules, so the
De Giorgi and continuation arguments do not pass through an auxiliary
classical solution.  Initial homogeneous Neumann compatibility is built into
the trace space of the global theorem.

\section{The normalized Maxwell--Stefan system and the additive class}

\subsection{Constrained formulation}

Throughout, a domain means a nonempty open connected subset of $\R^d$.
Throughout the PDE analysis we assume $d\ge2$\footnote{{The restriction
$d\ge2$ is technical and avoids a separate one-dimensional variant of the boundary regularity
argument.  For $d=1$, the conormal estimate used below can be brought into the cited
two-dimensional framework by adjoining a dummy spatial variable, but this is an artificial
dimension-raising device; a direct one-dimensional proof should be more elementary.  We do not
pursue either variant here.}}.  Let
$\Omega\subset\R^d$ be a bounded domain of class $C^2$.  {We consider a
quiescent, isothermal and isobaric mixture of $N\ge2$ species.}  The unknowns are normalized
species fractions; throughout the paper we use the mass-fraction convention
\[
        y=(y_1,\ldots,y_N)^{\sf T},
        \qquad y_i\ge0,
        \qquad \sum_{i=1}^N y_i=1.
\]
{We denote the closed and open concentration simplices by}
{
\[
 \overline{\cS}:=\left\{y\in[0,1]^N:\ \sum_{i=1}^N y_i=1\right\},
 \qquad
 \cS^\circ:=\left\{y\in(0,1)^N:\ \sum_{i=1}^N y_i=1\right\}.
\]}
The corresponding diffusional mass fluxes are $J_i\in\R^d$, collected in
$J=(J_1,\ldots,J_N)^{\sf T}\in\R^{N\times d}$.  We use the normalization \eqref{eq:B-general} throughout; all statements about the
coefficients $f_{ij}$ refer to this convention.  {The model below uses the simplified
driving force $\nabla y$.}
The balance laws are
\begin{equation}\label{eq:balance}
        \partial_t y_i+\Div J_i=0,
        \qquad i=1,\ldots,N,
\end{equation}
with no-flux boundary condition
\begin{equation}\label{eq:noflux}
        J_i\cdot\nu=0\qquad\text{on }\partial\Omega,
\end{equation}
and the flux constraint
\begin{equation}\label{eq:flux-constraint}
        \sum_{i=1}^N J_i=0.
\end{equation}

For $k=1,\ldots,d$, let $J^k=(J_1^k,\ldots,J_N^k)^{\sf T}\in\R^N$ be the
$k$-th column of $J$.  The Maxwell--Stefan law is written as
\begin{equation}\label{eq:MS-law}
        \partial_{x_k}y=B(y)J^k,
        \qquad k=1,\ldots,d,
\end{equation}
equivalently $\nabla_x y=B(y)J$, where
\begin{equation}\label{eq:B-general}
        B_{ij}(y)=y_i f_{ij}\quad (i\ne j),
        \qquad
        B_{ii}(y)=-\sum_{\ell\ne i} f_{i\ell}y_\ell.
\end{equation}
{For the symmetric pair coefficients considered below,
$\one^{\sf T}B(y)=0$ and $B(y)y=0$.  Thus $\one=(1,\ldots,1)^{\sf T}$ is the left conservation null vector, whereas
$y$ is a right null vector.  We set}
\begin{equation}\label{eq:E-space}
        E:=\{v\in\R^N:\ \one\cdot v=0\}.
\end{equation}
Since $\sum_i y_i=1$, every derivative $\partial_{x_k}y$ belongs to $E$, and by \eqref{eq:flux-constraint} every $J^k$ belongs to $E$.

\subsection{The additive coefficient cone}

\begin{definition}[Additive Maxwell--Stefan class]
A symmetric coefficient family $(f_{ij})_{i\ne j}$ belongs to the additive Maxwell--Stefan class if there exist numbers $g_1,\ldots,g_N>0$ such that
\begin{equation}\label{eq:fij-gigj}
        f_{ij}=g_i+g_j,
        \qquad i\ne j.
\end{equation}
{We call any such vector $g=(g_1,\ldots,g_N)$ an additive generator of the coefficient family.}
\end{definition}

The algebraic cone characterization and constrained-inverse results below do
not require the values $g_i$ to be pairwise distinct.  After relabeling we may, and do, assume for these sections that
$0<g_1\le\cdots\le g_N$.  Thus $g_1$ and $g_N$ are the minimum and maximum of the $g_i$ in the uniform
bounds below, while repetitions remain allowed.  Pairwise distinctness is imposed only
when the root-coordinate theory is introduced in
Section~\ref{sec:root-coordinates}; there the ordering is strengthened to be
strict.

\begin{proposition}[Characterization of the additive cone for $N\ge4$]\label{prop:additive-cone}
Let $N\ge4$ and let $(f_{ij})_{1\le i<j\le N}$ be a symmetric family of real
numbers.  The following are equivalent.
\begin{enumerate}[label=\textup{(\roman*)}]
\item There is a unique vector $g=(g_1,\ldots,g_N)\in(0,\infty)^N$ such that
      $f_{ij}=g_i+g_j$ for all $i\ne j$.
\item For every four distinct indices $i,j,k,\ell$ one has
      \begin{equation}\label{eq:four-point}
              f_{ij}+f_{k\ell}=f_{ik}+f_{j\ell}=f_{i\ell}+f_{jk},
      \end{equation}
      and, for each $i$, the number
      \begin{equation}\label{eq:gi-recovery}
              g_i:=\frac{f_{ij}+f_{ik}-f_{jk}}2
      \end{equation}
      is positive for one, and hence for every, pair of distinct indices
      $j,k\ne i$.
\end{enumerate}
Under condition~\textup{(ii)}, the value in \eqref{eq:gi-recovery} is
independent of the auxiliary pair $(j,k)$.
\end{proposition}

\begin{proof}
If $f_{ij}=g_i+g_j$, then \eqref{eq:four-point} is immediate and
\eqref{eq:gi-recovery} recovers $g_i$, which also proves uniqueness.

Conversely, assume~\textup{(ii)}.  Fix $i$.  If $j,k,\ell$ are distinct and
all differ from $i$, then the four-point identity gives
\[
 \frac{f_{ij}+f_{ik}-f_{jk}}2
 -\frac{f_{ij}+f_{i\ell}-f_{j\ell}}2
 =\frac{f_{ik}+f_{j\ell}-f_{i\ell}-f_{jk}}2=0.
\]
Thus replacing one member of the auxiliary pair does not change the recovered
value.  The graph whose vertices are the two-element subsets of
$\{1,\ldots,N\}\setminus\{i\}$ and whose edges join pairs sharing one index is
connected; hence \eqref{eq:gi-recovery} is independent of the pair $(j,k)$.
It is positive by hypothesis.

For distinct $i,j,k$, recover $g_i$ with the pair $(j,k)$ and $g_j$ with the
pair $(i,k)$.  Then
\[
 g_i+g_j
 =\frac{f_{ij}+f_{ik}-f_{jk}}2
  +\frac{f_{ij}+f_{jk}-f_{ik}}2
 =f_{ij}.
\]
Therefore the family is additive.
\end{proof}

\begin{remark}[Geometry of the cone for $N\ge4$]
For $N\ge4$, the additive class has only $N$ free parameters, whereas a
general symmetric Maxwell--Stefan family has $N(N-1)/2$ entries.  Thus the
additive coefficients form an $N$-dimensional subset of the
$N(N-1)/2$-dimensional space of symmetric coefficient families.
\end{remark}

\begin{remark}[The ternary case]\label{rem:ternary-exception}
For $N=3$, every symmetric friction triple has a unique real pair-sum
representation.  The representing values $g_i$ are positive if and only if the
strict triangle inequalities hold for $(f_{12},f_{13},f_{23})$, namely
\[
 f_{12}<f_{13}+f_{23},\qquad
 f_{13}<f_{12}+f_{23},\qquad
 f_{23}<f_{12}+f_{13}.
\]
\end{remark}

\begin{remark}[The binary case]
For $N=2$, every coefficient $f_{12}>0$ admits additive representations, but
the generator is not unique: any $g_1,g_2>0$ with $g_1+g_2=f_{12}$ is
admissible.  {This nonuniqueness has no effect on the physical binary equation.}
{The binary PDE is linear in the present normalization:} $J_2=-J_1$ and the first row of
\eqref{eq:MS-law} gives $\nabla y_1=-f_{12}J_1$, so that
\[
        \partial_t y_1-\frac1{f_{12}}\Delta y_1=0,
        \qquad \partial_\nu y_1=0,
        \qquad y_2=1-y_1.
\]
Standard Neumann heat-equation theory therefore gives global strong/classical
well-posedness for every admissible binary datum in the regularity classes used
below.  Thus the open arbitrary-large-data global strong well-posedness questions
discussed below concern systems with $N\ge3$.  For $N\ge3$ the additive
generator, when it exists, is unique.
\end{remark}

\subsection{Dissipation under the barycentric momentum constraint}
\label{subsec:species-generated-friction}

{We use the mass-based normalization in which the total density is one.  If
$u_i\in\R^d$ denotes the diffusion velocity of species $i$, then}
\begin{equation}\label{eq:J-Yu-physical}
        J_i=y_i u_i,
        \qquad
        \sum_{i=1}^N y_i u_i=0.
\end{equation}
The second relation is the barycentric constraint for diffusion velocities
measured relative to the common barycentric velocity.  {All identities in this subsection use the normalization
\eqref{eq:B-general} and the simplex weights $y_i$.  We first take
$y\in\cS^\circ$, so that the weighted form below is an inner product.}

Let
\begin{equation}\label{eq:Ey-physical}
        E_y^{(d)}:=\left\{u\in(\R^d)^N:\ \sum_{i=1}^Ny_i u_i=0\right\},
        \qquad E_y:=E_y^{(1)}.
\end{equation}
We equip $(\R^d)^N$ with the weighted inner product
\begin{equation}\label{eq:y-weighted-inner-product}
        \langle u,v\rangle_y:=\sum_{i=1}^Ny_i u_i\cdot v_i.
\end{equation}
{With $Y:=\diag(y_1,\ldots,y_N)$, the relation $J_i=y_i u_i$ reads
$J=Yu$ componentwise.  Since $y\in\cS^\circ$, multiplication by $Y$ is an isomorphism
$Y:E_y\to E$ in each spatial component.  Thus $E_y$ is the barycentric velocity constraint
space, whereas $E$ is the flux/tangent constraint space.  In one velocity component we also
define}
{
\begin{equation}\label{eq:Py-def-early}
 P_yv:=v-\one(y\cdot v).
\end{equation}}
{Then $P_y$ is the orthogonal projection onto $E_y$ with respect to
$\langle\cdot,\cdot\rangle_y$.}
{
In continuum thermodynamics, the primary dissipative quantity associated with
diffusion is its contribution to the entropy production.  In a class-II
multi-velocity formulation, each constituent satisfies a partial momentum
balance, and interspecies momentum transfer enters this entropy-production
term through a force--velocity pairing.  Under the isothermal specialization
used here and the diffusive approximation in which the diffusion velocities
are small compared with the speed of sound, the partial momentum balances
reduce to Maxwell--Stefan force balances.  The diffusion entropy production
therefore admits a mechanical representation as the power dissipated by
pairwise interspecies friction; this reduction is developed in
\cite{BotheDreyer2015}.  The thermodynamic force--flux formulation and its
relation to Maxwell--Stefan and Fick--Onsager closures are discussed in
\cite{BotheDruet2023}.

{In the present normalized ideal setting}, this mechanical
representation is the Rayleigh dissipation
\begin{equation}\label{eq:rayleigh-general}
        \cD_y^{\mathrm{fr}}(u)
        :=\frac12\sum_{\substack{i,j=1\\ i\ne j}}^N
        f_{ij}y_i y_j|u_i-u_j|^2,
\end{equation}
and the Maxwell--Stefan force balance gives, with $J_i=y_i u_i$,
\begin{equation}\label{eq:rayleigh-entropy-production}
 -\sum_{i=1}^N J_i\cdot\nabla\log y_i
 =\cD_y^{\mathrm{fr}}(u).
\end{equation}
Thus \eqref{eq:rayleigh-general} is the mechanical friction-power
representation of the diffusion entropy-production density in the present
normalization.  The corresponding entropy inequality for the isothermal
Maxwell--Stefan system, including its representation in terms of individual
diffusion velocities, is also given in \cite{Bothe2011}.
}

\begin{remark}[Rao quadratic entropy and additive diagonalization]
For each fixed diffusion-velocity configuration $u$, set
$d_{ij}^{u}:=\frac12 f_{ij}|u_i-u_j|^2$ for $i\ne j$ and $d_{ii}^{u}:=0$.
Then \eqref{eq:rayleigh-general} is exactly Rao's quadratic diversity (or
quadratic entropy) \cite{Rao1982} of the composition $y$ for the instantaneous
dissimilarity matrix $d^{u}$, namely
$\cD_y^{\mathrm{fr}}(u)=\sum_{i,j}y_i y_jd_{ij}^{u}$.  In the additive class,
\Cref{prop:diagonal-barycentric-dissipation} gives the equivalent barycentric
representation $\cD_y^{\mathrm{fr}}(u)=\sum_i y_i(g_i+s)|u_i|^2$; thus
additivity diagonalizes this Rao quadratic form into a species-weighted second
moment.  The qualification is dynamical rather than algebraic: $d^{u}$ depends
on the instantaneous rate variables, so $\cD_y^{\mathrm{fr}}$ is not a state
entropy for the Maxwell--Stefan evolution; its thermodynamic role remains the
production density of the logarithmic mixing entropy in
\eqref{eq:rayleigh-entropy-production}.
\end{remark}

Because the diagonal terms vanish, whenever an all-index sum is used below we
{adopt the convention $f_{ii}:=0$}; the additive law
$f_{ij}=g_i+g_j$ is imposed only for $i\ne j$.  This convention does not
change the dissipation.

In the continuum-thermodynamic derivation, symmetry and non-negativity of the
$f_{ij}$ follow from binary mechanical interactions, conservation of total
momentum and the entropy inequality; the dependence of $f_{ij}$ on the pair of
species remains constitutive \cite{BotheDreyer2015}.

The next proposition characterizes when the pair-friction quadratic form
$\cD_y^{\mathrm{fr}}$ becomes species-diagonal after restriction to the
barycentric velocity space $E_y^{(d)}$.

{
\begin{proposition}[Species-diagonal Rayleigh dissipation characterizes additivity]
\label{prop:diagonal-barycentric-dissipation}
Let $f_{ij}=f_{ji}$ for $i\ne j$.  The following statements are equivalent.
\begin{enumerate}[label=\textup{(\roman*)}]
\item There exists at least one $y\in\cS^\circ$ and real numbers
{$\kappa_i(y)$} such that
\begin{equation}\label{eq:diagonal-rayleigh-abstract}
        \cD_y^{\mathrm{fr}}(u)
        =\sum_{i=1}^N y_i {\kappa_i(y)}|u_i|^2
        \qquad\text{for every }u\in E_y^{(d)}.
\end{equation}
\item For every $y\in\cS^\circ$ the Rayleigh dissipation admits a
species-diagonal representation of the form
\eqref{eq:diagonal-rayleigh-abstract} on $E_y^{(d)}$.
\item There exist real species parameters $g_1,\ldots,g_N$ such that
\begin{equation}\label{eq:additivity-from-diagonal-dissipation}
        f_{ij}=g_i+g_j,\qquad i\ne j.
\end{equation}
\end{enumerate}
If these equivalent conditions hold and
$ s:=g\cdot y=\sum_i g_i y_i$, then one may take
\begin{equation}\label{eq:diagonal-rayleigh}
        {\kappa_i(y)}=g_i+s,\qquad
        \cD_y^{\mathrm{fr}}(u)
        =\sum_{i=1}^N y_i(g_i+s)|u_i|^2.
\end{equation}
For $N\ge3$ the generator $g$ is unique and is recovered from any three
distinct indices by
\begin{equation}\label{eq:gi-from-diagonal-dissipation}
        g_i=\frac{f_{ij}+f_{ik}-f_{jk}}{2}.
\end{equation}

Under \eqref{eq:additivity-from-diagonal-dissipation}, define
$G:=\diag(g_1,\ldots,g_N)$ and the linear friction operator
$\mathcal L_y$ by
\begin{equation}\label{eq:Ly-def}
        (\mathcal L_yu)_i
        :=\sum_{j=1}^N(g_i+g_j)y_j(u_i-u_j).
\end{equation}
Then, on $E_y^{(d)}$,
\begin{equation}\label{eq:Ly-diagonal-rankone}
        (\mathcal L_yu)_i=(g_i+s)u_i-\mathfrak m_y(u),
        \qquad
        \mathfrak m_y(u):=\sum_{j=1}^Ng_jy_j u_j,
\end{equation}
and
\begin{equation}\label{eq:D-L-inner}
        \cD_y^{\mathrm{fr}}(u)=\langle\mathcal L_yu,u\rangle_y.
\end{equation}
If $Y=\diag(y_1,\ldots,y_N)$, then the Maxwell--Stefan matrix from
\eqref{eq:B-general} satisfies
\begin{equation}\label{eq:BY-minus-YL}
        B(y)Y=-Y\mathcal L_y.
\end{equation}
\end{proposition}

\begin{proof}
It is enough to work with one velocity component.  For fixed
$y\in\cS^\circ$, write the pairwise form as
\[
        \cD_y^{\mathrm{fr}}(u)=u^{\sf T}L_f(y)u,
\]
where $L_f(y)=L_f(y)^{\sf T}$ has off-diagonal entries
\[
        (L_f(y))_{ij}=-f_{ij}y_i y_j,\qquad i\ne j,
\]
and diagonal entries chosen so that the displayed quadratic identity holds.
Assume (i) and set
\[
        {\Delta_y}:=\diag\bigl(y_1{\kappa_1(y)},\ldots,
        y_N{\kappa_N(y)}\bigr),
        \qquad K:=L_f(y)-{\Delta_y}.
\]
Then $u^{\sf T}Ku=0$ for every $u\in E_y=y^\perp$.  Polarization gives
$v^{\sf T}Kw=0$ for all $v,w\in E_y$.  Thus, with respect to the Euclidean
orthogonal decomposition
\[
        \R^N=E_y\oplus\operatorname{span}\{y\},
\]
the $E_y\times E_y$ block of $K$ vanishes.  Since $K$ is symmetric, there is
a vector $a\in\R^N$ such that
\begin{equation}\label{eq:K-constraint-rank-two}
        K=ya^{\sf T}+ay^{\sf T}.
\end{equation}
Equivalently,
\[
        u^{\sf T}Ku=2(y\cdot u)(a\cdot u),
        \qquad u\in\R^N.
\]
Thus the difference of the two quadratic forms factors through the normal
functional $u\mapsto y\cdot u$.  Comparing off-diagonal entries in
\eqref{eq:K-constraint-rank-two} yields
\[
        -f_{ij}y_i y_j=y_i a_j+a_i y_j,\qquad i\ne j.
\]
Since $y_i>0$, division by $y_i y_j$ and the definition
$g_i:=-a_i/y_i$ give $f_{ij}=g_i+g_j$.  Thus (i) implies (iii).
For $N\ge3$, \eqref{eq:gi-from-diagonal-dissipation} follows immediately
and proves uniqueness.

Assume conversely (iii).  By symmetry in $i$ and $j$, for
$u\in E_y^{(d)}$,
\[
\begin{aligned}
\cD_y^{\mathrm{fr}}(u)
&=\sum_{i,j=1}^N g_i y_i y_j|u_i-u_j|^2\\
&=\sum_{i=1}^Ng_i y_i
  \left(|u_i|^2-2u_i\cdot\sum_jy_j u_j+\sum_jy_j|u_j|^2\right).
\end{aligned}
\]
The barycentric constraint eliminates the middle term, and
$\sum_i g_i y_i=s$ gives \eqref{eq:diagonal-rayleigh}.  Hence (iii)
implies (ii), while (ii) trivially implies (i).

Finally, expanding \eqref{eq:Ly-def} and using $\sum_jy_j=1$ and
$\sum_jy_j u_j=0$ gives
\[
(\mathcal L_yu)_i
=(g_i+s)u_i-\sum_jg_jy_j u_j,
\]
which is \eqref{eq:Ly-diagonal-rankone}.  Taking the weighted inner product
with $u$ eliminates the common rank-one term and yields
\eqref{eq:D-L-inner}.  For arbitrary $u$,
\[
[B(y)Yu]_i
=y_i\sum_{j\ne i}f_{ij}y_j u_j-y_i u_i\sum_{j\ne i}f_{ij}y_j
=-y_i(\mathcal L_yu)_i,
\]
which proves \eqref{eq:BY-minus-YL}.
\end{proof}

{The representation obtained in the converse has $g\in\mathbb R^N$; the global
theory assumes in addition that $g_i>0$.  For $N=3$ this is equivalent to the strict triangle inequalities
in \Cref{rem:ternary-exception}; for $N\ge4$, additivity itself is restrictive and is characterized
by \Cref{prop:additive-cone}.  The algebraic identities
\eqref{eq:diagonal-rayleigh}--\eqref{eq:BY-minus-YL} remain valid for
$y\in\overline{\cS}$, where the weighted form is only nonnegative semidefinite.}
}

{The additive law also admits the following mean-field mesoscopic
realization.  In the ideal molar formulation, let $\nu_i>0$ be a
species-specific momentum-exchange activation rate.  Conditional on an
activation associated with species $i$, the partner species is sampled
according to the local mole fractions, and the mean momentum transfer is taken
proportional to the relative velocity.  The oriented $i\to j$ and $j\to i$
event intensities are then proportional to $x_i x_j\nu_i$ and
$x_i x_j\nu_j$, respectively.  Summation of the two contributions gives,
after a common normalization, $f_{ij}=g_i+g_j$ with $g_i=C\nu_i$, $C>0$.
{This gives an exact constitutive realization of the additive law within the
specified mesoscopic closure.  It is not a derivation valid for general
molecular collision models: reduced masses and pair-specific cross sections
generally yield Maxwell--Stefan coefficients depending on the species pair;
see, for example, \cite{BoudinGrecPavan2017}.}

In the molar-average frame, \eqref{eq:diagonal-rayleigh} with
$y$ replaced by $x$ shows that the effective constrained relaxation coefficient
of species $i$ is $g_i+s_x$, where $s_x:=\sum_jg_jx_j$.
{Here $s_x=\sum_jx_jg_j$ is the composition-weighted arithmetic mean of the
generator values, so $g_i+s_x$ is the sum of a species-specific contribution and an arithmetic
mixture average.}  Hence the closure yields the
diagonal-plus-rank-one constrained operator used in the subsequent spectral
analysis.

An independent approximation argument leads to the same additive form.  Suppose
$f_{ij}=F(\vartheta_i,\vartheta_j)$ with $F(a,b)=F(b,a)$, where $F$ is smooth and symmetric and the $\vartheta_i$ are scalar
constituent parameters close to a common reference value $\vartheta_*$.  With
$\delta:=\max_i|\vartheta_i-\vartheta_*|$, symmetry gives
\[
 F(\vartheta_i,\vartheta_j)
 =F(\vartheta_*,\vartheta_*)
 +\partial_1F(\vartheta_*,\vartheta_*)
  \bigl[(\vartheta_i-\vartheta_*)+(\vartheta_j-\vartheta_*)\bigr]
 +O(\delta^2).
\]
Hence, with
\[
 g_i:=\frac12F(\vartheta_*,\vartheta_*)
      +\partial_1F(\vartheta_*,\vartheta_*)
       (\vartheta_i-\vartheta_*),
\]
one has $f_{ij}=g_i+g_j+O(\delta^2)$; if
$F(\vartheta_*,\vartheta_*)>0$, then $g_i>0$ for sufficiently small
$\delta$.  Thus $g_i+g_j$ is the first-order Taylor approximation of a smooth
symmetric pair law near a common reference state.

Within the mesoscopic closure above, the coefficient is exactly additive.
Independently, the additive form is the first-order Taylor approximation of a
smooth symmetric pair law near a common reference state.  For $N\ge4$, in the normalization \eqref{eq:B-general}, it is a
restrictive constitutive assumption: a coefficient family is additive exactly
when the
four-point identities \eqref{eq:four-point} hold and the parameters recovered
from \eqref{eq:gi-recovery} are positive.
}

\subsection{Relation to the core-diagonal closure}
\label{subsec:core-diagonal-relation}

{
The diffusion contribution to entropy production provides a common
thermodynamic formulation for the additive Maxwell--Stefan class and the
core-diagonal closure of Bothe and Druet \cite{BotheDruet2023}.  In the
isothermal setting, diffusion velocities and thermodynamic driving forces are
conjugate variables in this entropy-production term.  On the constrained
spaces, a linear diffusion closure can equivalently be represented by a
resistance operator mapping diffusion velocities to driving forces or by {its inverse, the
mobility operator, mapping projected driving forces to diffusion velocities or fluxes}.  Species-diagonal representations, however, are not in
general preserved under the conservation projection and constrained
inversion.

For additive Maxwell--Stefan coefficients,
\Cref{prop:diagonal-barycentric-dissipation} gives the resistance-side identity
\[
 \cD_y^{\mathrm{fr}}(u)
 =\sum_i y_i(g_i+g\cdot y)|u_i|^2,
 \qquad y\cdot u=0.
\]
Thus, on the barycentric velocity space, the pair-friction quadratic form
coincides with the restriction of a species-diagonal quadratic form.  This is
the resistance-side diagonal representation associated with the additive
class.

In the core-diagonal closure of \cite[Section~8]{BotheDruet2023}, species-wise
diagonal mobility coefficients are prescribed before the projection required
by total-mass conservation.  The resulting physical flux operator is generally
non-diagonal.  Thus the additive Maxwell--Stefan class gives a species-diagonal
representation on the resistance side of the thermodynamic force--flux
pairing, whereas the core-diagonal closure prescribes species-diagonal
constitutive coefficients on the mobility side.

The two diagonal representations do not define the same coefficient subclass.
On the constrained spaces, resistance and mobility are inverse operators, but
projection and inversion do not preserve diagonality.  In the molar
normalization of \cite[Section~8]{BotheDruet2023}, a core-diagonal mobility with
coefficients $d_i>0$ yields
\begin{equation}\label{eq:core-diagonal-Darken-short}
 \mathcal D_{ij}^{\mathrm{MS}}
 =d_i d_j\sum_{k=1}^N\frac{x_k}{d_k},
 \qquad i\ne j.
\end{equation}
{Define the composition-weighted harmonic mixture diffusivity by
$D_{\mathrm{mix}}(x,d):=(\sum_{k=1}^N x_k/d_k)^{-1}$.  This is the mole-fraction-weighted harmonic mean of the species diffusivities,
and \eqref{eq:core-diagonal-Darken-short} takes the multicomponent Darken form}
{
\[
 \mathcal D_{ij}^{\mathrm{MS}}=\frac{d_i d_j}{D_{\mathrm{mix}}(x,d)},\qquad i\ne j.
\]}
{The common factor is therefore structurally part of the closure: it is the
harmonic-mixture normalization generated by the conservation projection and couples each pair to
the full composition.  To compare the two closures on the same resistance side, introduce the
corresponding Maxwell--Stefan friction coefficients}
{
\[
 f_{ij}^{\mathrm{CD}}
 :=\bigl(\mathcal D_{ij}^{\mathrm{MS}}\bigr)^{-1}
 =\frac{D_{\mathrm{mix}}(x,d)}{d_i d_j}
 =b_i b_j,
 \qquad
 b_i:=\frac{\sqrt{D_{\mathrm{mix}}(x,d)}}{d_i}.
\]}
{Thus, at a fixed thermodynamic state, the core-diagonal resistance pair array is
multiplicatively separated and satisfies, for four distinct indices,}
{
\[
 f_{ij}^{\mathrm{CD}}f_{k\ell}^{\mathrm{CD}}
 =f_{ik}^{\mathrm{CD}}f_{j\ell}^{\mathrm{CD}}
 =f_{i\ell}^{\mathrm{CD}}f_{jk}^{\mathrm{CD}}.
\]}
{The additive resistance coefficients are instead $f_{ij}=g_i+g_j$, and}
\[
 (g_i+g_j)(g_k+g_\ell)-(g_i+g_k)(g_j+g_\ell)
 =(g_i-g_\ell)(g_k-g_j).
\]
{Hence, for $N\ge4$, if $g_i\ne g_j$ whenever $i\ne j$, the additive family does
not belong to the core-diagonal pair class.}  {The arithmetic/harmonic analogy is structural rather
than an identification under inversion: the arithmetic mean $s_x=\sum_kx_kg_k$ enters the
specieswise constrained resistance coefficient $g_i+s_x$, whereas $D_{\mathrm{mix}}$ normalizes
the pair Maxwell--Stefan diffusivity in the core-diagonal closure.}

For the additive Maxwell--Stefan class one additionally has the
compressed-diagonal identity
\[
 \mathcal L_y|_{E_y}=P_yG|_{E_y}+sI,
 \qquad G=\diag(g_1,\ldots,g_N),
 \qquad s=g\cdot y.
\]
The next two sections exploit the same generator matrix $G$ from complementary
sides: first through the compressed-diagonal form and its resolvent, and then
through the associated spectral coordinates in \Cref{sec:root-coordinates}.
}

\section{Explicit inversion of the Maxwell--Stefan law}

The constrained diagonality obtained above has the following operator form.  Let
\begin{equation}\label{eq:s-def}
\begin{gathered}
        G:=\diag(g_1,\ldots,g_N),\qquad
        s:=g^{\sf T}y=\sum_{i=1}^N g_i y_i,\\
        D_s:=(G+sI)^{-1}
        =\diag\!\left(\frac1{g_1+s},\ldots,\frac1{g_N+s}\right),
\end{gathered}
\end{equation}
and let
\[
        \Pi_y:=I-y\one^{\sf T},\qquad E=\ker\one^{\sf T}.
\]
Thus $\Pi_y$ is the projection onto $E$ along $\operatorname{span}\{y\}$.
By the ordering fixed in Section~2, $g_1\le s\le g_N$ on the closed
simplex; in particular $s>0$.

\begin{lemma}[Compressed-diagonal form on the constraint space]
\label{lem:row-formula}
For every $v\in E$,
\begin{equation}\label{eq:B-compressed-E}
        -B(y)|_E=\Pi_yG|_E+sI_E.
\end{equation}
Equivalently,
\begin{equation}\label{eq:B-row-additive}
        B(y)v=-(G+sI)v+y\,(g^{\sf T}v),
\end{equation}
and hence
\[
        [B(y)v]_i=y_i(g\cdot v)-(g_i+s)v_i.
\]
Thus all non-diagonal coupling on $E$ is carried by the one-dimensional
correction required by the conservation constraint.
\end{lemma}

\begin{proof}
Using $\sum_jv_j=0$ and $\sum_jy_j=1$ in the definition of $B(y)$ gives
\[
 [B(y)v]_i
 =y_i\sum_{j\ne i}(g_i+g_j)v_j
  -v_i\sum_{j\ne i}(g_i+g_j)y_j
 =y_i(g\cdot v)-(g_i+s)v_i.
\]
Since $v\in E$, this is $B(y)v=-\Pi_y(G+sI)v$, proving
\eqref{eq:B-compressed-E}.
\end{proof}

Put
\begin{equation}\label{eq:C-compression-def}
        \mathcal C_y:=\Pi_yG|_E:E\to E.
\end{equation}
Here and below, the compression of an ambient operator to a constraint space means
its composition with the corresponding projection, restricted to that space.
The compression has an explicit real resolvent formula.  For
$\zeta\in\mathbb R\setminus\{g_1,\ldots,g_N\}$ set
\begin{equation}\label{eq:cauchy-transform}
        R_\zeta:=(G-\zeta I)^{-1},\qquad
        m_y(\zeta):=\one^{\sf T}R_\zeta y
        =\sum_{i=1}^N\frac{y_i}{g_i-\zeta}.
\end{equation}

\begin{lemma}[Compressed resolvent]\label{lem:compressed-resolvent}
If $m_y(\zeta)\ne0$, then for every $h\in E$,
\begin{equation}\label{eq:compressed-resolvent}
 (\mathcal C_y-\zeta I_E)^{-1}h
 =R_\zeta h
  -\frac{\one^{\sf T}R_\zeta h}{m_y(\zeta)}\,R_\zeta y.
\end{equation}
Moreover,
\begin{equation}\label{eq:compressed-secular-equivalence}
 m_y(\zeta)=0
 \quad\Longleftrightarrow\quad
 \zeta\in\sigma(\mathcal C_y),
 \qquad \zeta\notin\{g_1,\ldots,g_N\},
\end{equation}
and in that case $R_\zeta y\in E$ is a nonzero eigenvector.
\end{lemma}

\begin{proof}
For $v\in E$ the equation
$(\mathcal C_y-\zeta I_E)v=h$ is equivalent to
\[
 (G-\zeta I)v=h+\alpha y,
 \qquad
 \alpha:=\one^{\sf T}Gv
          =\one^{\sf T}(G-\zeta I)v.
\]
Hence $v=R_\zeta h+\alpha R_\zeta y$.  The single constraint
$\one^{\sf T}v=0$ gives
\[
 \alpha=-\frac{\one^{\sf T}R_\zeta h}{m_y(\zeta)},
\]
which proves \eqref{eq:compressed-resolvent}.  If $m_y(\zeta)=0$, then
$R_\zeta y\in E$ and
$(\mathcal C_y-\zeta I_E)R_\zeta y=0$.  Conversely, a nonzero vector in the
kernel must have the form $\alpha R_\zeta y$ with $\alpha\ne0$, so its
membership in $E$ forces $m_y(\zeta)=0$.
\end{proof}

The Maxwell--Stefan inverse is now the resolvent at the single point
$\zeta=-s$.  Define
\begin{equation}\label{eq:R0-def}
        R_0(y):=m_y(-s)=\one^{\sf T}D_sy
        =\sum_{i=1}^N\frac{y_i}{g_i+s}>0.
\end{equation}

\begin{proposition}[Explicit inverse on the constraint space $E$]
\label{prop:flux-inverse}
Let $y\in\overline{\cS}$ and $h\in E$.  The equation
\begin{equation}\label{eq:Bv-h}
        B(y)v=h,\qquad v\in E,
\end{equation}
has a unique solution.  As an operator $E\to E$,
\begin{equation}\label{eq:inverse-B-rank-one}
 \bigl(B(y)|_E\bigr)^{-1}
 =\left[-D_s+\frac1{R_0(y)}(D_sy)\bigl(\one^{\sf T}D_s\bigr)\right]\Big|_E.
\end{equation}
Equivalently,
\begin{equation}\label{eq:inverse-B}
        v_i
        =-\frac{h_i}{g_i+s}
        +\frac{y_i}{R_0(y)(g_i+s)}
          \sum_{r=1}^N\frac{h_r}{g_r+s},
        \qquad i=1,\ldots,N.
\end{equation}
Thus one first applies the ambient diagonal inverse and then makes a
rank-one correction to restore the constraint $\one^{\sf T}v=0$.
\end{proposition}

\begin{proof}
By \eqref{eq:B-compressed-E},
$B(y)|_E=-(\mathcal C_y+sI_E)$.  Since
$R_{-s}=D_s$ and $m_y(-s)=R_0(y)>0$, formula
\eqref{eq:compressed-resolvent} at $\zeta=-s$ gives
\eqref{eq:inverse-B-rank-one}; \eqref{eq:inverse-B} is its component form.
Uniqueness follows from the same resolvent formula.
\end{proof}

\begin{corollary}[Flux formula and regularity up to the closed simplex]
\label{cor:inverse-rank-one}
For each spatial direction $k$, the additive Maxwell--Stefan law
$\partial_{x_k}y=B(y)J^k$ gives
\begin{equation}\label{eq:J-formula}
        J_i^k
        =-\frac{\partial_{x_k}y_i}{g_i+s}
        +\frac{y_i}{R_0(y)(g_i+s)}
          \sum_{r=1}^N\frac{\partial_{x_k}y_r}{g_r+s}.
\end{equation}
Moreover, $B(y)|_E$ is invertible for every $y\in\overline{\cS}$, and
$y\mapsto(B(y)|_E)^{-1}$ is real-analytic on a neighbourhood of
$\overline{\cS}$ in the affine hyperplane $\one^{\sf T}y=1$.
\end{corollary}

\begin{proof}
Both $\partial_{x_k}y$ and $J^k$ belong to $E$, so
\eqref{eq:J-formula} follows from \eqref{eq:inverse-B}.  On the closed
simplex,
\[
        g_i+s\ge2g_1,
        \qquad
        R_0(y)=\one^{\sf T}D_sy\ge\frac1{2g_N}.
\]
Hence all denominators in \eqref{eq:inverse-B-rank-one} are uniformly
separated from zero.  The formula is rational in $y$ and therefore gives the
asserted analytic extension.
\end{proof}

The same scalar function $m_y$ now controls both inversion and spectrum:
$\bigl(B(y)|_E\bigr)^{-1}$ is obtained from the compressed resolvent at $\zeta=-s$, while
the zeros of $m_y$ are precisely the spectral parameters of the compression.
This is the common origin of the rank-one inverse and the root coordinates.

\section{Root coordinates}
\label{sec:root-coordinates}

From this section onward, except where repeated generators are discussed
explicitly, we assume after relabeling that
\begin{equation}\label{eq:g-ordering}
        0<g_1<g_2<\cdots<g_N.
\end{equation}
Repeated values lead to multiple eigenvalues and do not provide the full root
coordinate chart used in the remainder of the paper; see
\Cref{rem:repeated-generators} below.

Away from the fixed poles $g_i$, the compressed-resolvent formula
\eqref{eq:compressed-resolvent} shows that the compression is singular precisely
at the zeros of the scalar function $m_y$.  We therefore use these zeros as
state variables.

\subsection{Definition and interlacing}

For $y\in\cS^\circ$, recall from \eqref{eq:cauchy-transform} that
\[
        m_y(z)=\one^{\sf T}(G-zI)^{-1}y
        =\sum_{i=1}^N\frac{y_i}{g_i-z}.
\]

\begin{proposition}[Interlacing roots]\label{prop:interlacing}
For every $y\in\cS^\circ$, the equation
\begin{equation}\label{eq:root-equation}
        m_y(z)=0
\end{equation}
 has exactly one solution $z_j=z_j(y)$ in each interval $(g_j,g_{j+1})$, $j=1,\ldots,N-1$.  Hence
\begin{equation}\label{eq:interlacing}
        g_1<z_1<g_2<z_2<\cdots<g_{N-1}<z_{N-1}<g_N.
\end{equation}
\end{proposition}

\begin{proof}
On each interval $(g_j,g_{j+1})$, the function $m_y$ is smooth and strictly increasing, since
\[
        m_y'(z)=\sum_{i=1}^N\frac{y_i}{(g_i-z)^2}>0.
\]
Moreover,
\[
        \lim_{z\downarrow g_j}m_y(z)=-\infty,
        \qquad
        \lim_{z\uparrow g_{j+1}}m_y(z)=+\infty.
\]
Therefore there is exactly one zero in each interval.
\end{proof}

\begin{remark}[Repeated generators]\label{rem:repeated-generators}
The identities in \Cref{prop:diagonal-barycentric-dissipation,prop:flux-inverse}
do not require the $g_i$ to be distinct.  Suppose the distinct values are
$\gamma_1<\cdots<\gamma_M$, let
$I_a=\{i:g_i=\gamma_a\}$, and put $Y_a=\sum_{i\in I_a}y_i$.  The space $E_y$
then decomposes orthogonally into the within-group spaces
\[
 W_a=\left\{v:\operatorname{supp}v\subset I_a,
             \ \sum_{i\in I_a}y_iv_i=0\right\}
\]
and the space of vectors constant on each group.  On $W_a$, the compressed
operator $P_yG|_{E_y}$ equals $\gamma_a I$; on the group-constant space its
remaining $M-1$ eigenvalues are the zeros of
\[
        \sum_{a=1}^M\frac{Y_a}{\gamma_a-z}=0.
\]
{The roots of this reduced rational equation}, together with the moments $\sum_i g_i^\ell y_i$, determine only the
group masses $Y_a$.  If $M<N$, additional variables describing the
composition within each group are therefore necessary.  The remainder of the
paper assumes $M=N$.
\end{remark}

\subsection{The roots as eigenvalues of the compressed operator}
\label{subsec:roots-relaxation-spectrum}

Recall from \eqref{eq:Py-def-early} that
$P_yv=v-\one(y\cdot v)$ is the weighted orthogonal projection onto $E_y$.
Define
\begin{equation}\label{eq:Ay-def}
        \mathcal A_y:=P_yG\big|_{E_y}:E_y\to E_y.
\end{equation}
It is self-adjoint in $\langle\cdot,\cdot\rangle_y$, and
\begin{equation}\label{eq:Ly-Ay-shift}
        \mathcal L_y\big|_{E_y}=\mathcal A_y+sI_{E_y}.
\end{equation}
Moreover, multiplication by $Y=\diag(y_1,\ldots,y_N)$ intertwines the
flux-space and weighted-orthogonal compressions:
\begin{equation}\label{eq:C-A-similarity}
        \mathcal C_yY=Y\mathcal A_y,
        \qquad Y:E_y\to E.
\end{equation}
Thus the flux-space compression $\mathcal C_y$ and the weighted-orthogonal
compression $\mathcal A_y$ are two realizations of the same operator.

\begin{proposition}[Secular equation and eigenvectors]
\label{prop:secular-relaxation-modes}
For $z\notin\{g_1,\ldots,g_N\}$ set
\begin{equation}\label{eq:qz-def}
        q(z):=Y^{-1}(G-zI)^{-1}y
        =\left(\frac1{g_1-z},\ldots,
                    \frac1{g_N-z}\right)^{\sf T}.
\end{equation}
Then
\begin{equation}\label{eq:qz-Ey-iff}
        q(z)\in E_y
        \quad\Longleftrightarrow\quad
        m_y(z)=0.
\end{equation}
For every root $z_j$ one has
\begin{equation}\label{eq:Ay-qj}
        \mathcal A_yq(z_j)=z_jq(z_j),
        \qquad
        \mathcal L_yq(z_j)=\lambda_jq(z_j),
        \qquad
        \lambda_j=s+z_j.
\end{equation}
Consequently,
\begin{equation}\label{eq:spectrum-A-L}
        \sigma(\mathcal A_y)=\{z_1,\ldots,z_{N-1}\},
        \qquad
        \sigma(\mathcal L_y|_{E_y})
        =\{\lambda_1,\ldots,\lambda_{N-1}\}.
\end{equation}
Moreover,
\begin{equation}\label{eq:compression-characteristic-roots}
        \det_{E_y}(zI_{E_y}-\mathcal A_y)
        =\det_E(zI_E-\mathcal C_y)
        =\prod_{j=1}^{N-1}(z-z_j).
\end{equation}
\end{proposition}

\begin{proof}
Since $Yq(z)=(G-zI)^{-1}y$, the equivalence
\eqref{eq:qz-Ey-iff} is exactly the identity
$y\cdot q(z)=m_y(z)$.  If $m_y(z_j)=0$, then
\Cref{lem:compressed-resolvent} gives
\[
 \mathcal C_y\bigl((G-z_jI)^{-1}y\bigr)
 =z_j(G-z_jI)^{-1}y.
\]
The intertwining identity \eqref{eq:C-A-similarity} yields
$\mathcal A_yq(z_j)=z_jq(z_j)$, and
\eqref{eq:Ly-Ay-shift} gives the relaxation eigenvalue $\lambda_j=s+z_j$.
The $N-1$ interlacing roots are distinct and therefore exhaust the spectrum.
The two operators are similar by \eqref{eq:C-A-similarity}; their
characteristic polynomials are monic of degree $N-1$ and have precisely these
roots, which proves \eqref{eq:compression-characteristic-roots}.
\end{proof}

\subsection{Polynomial representation and inverse map}

Set
\begin{equation}\label{eq:P-polynomial}
        P(z):=\det(zI-G)=\prod_{i=1}^N(z-g_i)
\end{equation}
and
\begin{equation}\label{eq:Qy-def}
        Q_y(z):=\sum_{i=1}^N y_i\prod_{m\ne i}(z-g_m).
\end{equation}
Then
\begin{equation}\label{eq:m-Q-over-P}
        m_y(z)=-\frac{Q_y(z)}{P(z)}.
\end{equation}
Since $\sum_i y_i=1$, $Q_y$ is monic of degree $N-1$; by interlacing and
\eqref{eq:m-Q-over-P},
\begin{equation}\label{eq:Q-roots}
        Q_y(z)=\prod_{j=1}^{N-1}(z-z_j).
\end{equation}
Combining this with \eqref{eq:compression-characteristic-roots} gives
\begin{equation}\label{eq:Q-characteristic}
        Q_y(z)=\det_E(zI_E-\mathcal C_y)
              =\det_{E_y}(zI_{E_y}-\mathcal A_y).
\end{equation}
Thus $Q_y$ is exactly the characteristic polynomial of the compression.

The sum of the roots follows directly from the trace.
Indeed, since $\Pi_yG$ has range in $E$,
\[
 \sum_{j=1}^{N-1}z_j
 =\operatorname{tr}_E\mathcal C_y
 =\operatorname{tr}(\Pi_yG)
 =\operatorname{tr}G-\operatorname{tr}(y\one^{\sf T}G)
 =e_1-s,
\]
where $e_1:=\sum_i g_i$.  Thus
\begin{equation}\label{eq:general-root-trace-identity}
        \sum_{j=1}^{N-1} z_j=e_1-s.
\end{equation}
Consequently,
\begin{equation}\label{eq:general-root-gradient-trace}
        \lambda_j=e_1-\sum_{\substack{k=1\\k\ne j}}^{N-1}z_k,
        \qquad
        \nabla s=-\sum_{k=1}^{N-1}\nabla z_k,
        \qquad
        \nabla\lambda_j
        =-\sum_{\substack{k=1\\k\ne j}}^{N-1}\nabla z_k
\end{equation}
along every smooth composition field.  In particular, in root coordinates
$\lambda_j$ and $\lambda_j^{-1}$ are independent of the selected coordinate
$z_j$.

\begin{theorem}[Global root coordinates]\label{thm:root-diffeo}
The map
\[
        y\in\cS^\circ\longmapsto z(y)=(z_1(y),\ldots,z_{N-1}(y))
\]
is a real-analytic diffeomorphism from the open simplex $\cS^\circ$ onto the interlacing box
\begin{equation}\label{eq:interlacing-box}
        \cZ:=\prod_{j=1}^{N-1}(g_j,g_{j+1}).
\end{equation}
Its inverse is
\begin{equation}\label{eq:y-inverse-roots}
        y_i(z)=
        \frac{\displaystyle\prod_{j=1}^{N-1}(g_i-z_j)}
             {\displaystyle\prod_{m\ne i}(g_i-g_m)},
        \qquad i=1,\ldots,N.
\end{equation}
\end{theorem}

\begin{proof}
Starting from $y\in\cS^\circ$, evaluation of \eqref{eq:Q-roots} at $z=g_i$ gives
\[
        y_i\prod_{m\ne i}(g_i-g_m)
        =\prod_{j=1}^{N-1}(g_i-z_j),
\]
which is \eqref{eq:y-inverse-roots}.

Conversely, let $z\in\cZ$ and put
\[
        P_z(\xi):=\prod_{j=1}^{N-1}(\xi-z_j),
        \qquad
        y_i:=\frac{P_z(g_i)}{\prod_{m\ne i}(g_i-g_m)}.
\]
Interlacing shows that the numerator and denominator have the same sign, hence
$y_i>0$.  Lagrange interpolation gives
\[
        P_z(\xi)
        =\sum_{i=1}^N y_i\prod_{m\ne i}(\xi-g_m).
\]
Comparison of the leading coefficients yields $\sum_i y_i=1$.  Thus
$y\in\cS^\circ$, and the last identity shows that its scalar resolvent
function $m_y$ has exactly the prescribed zeros.  This proves bijectivity.  The roots depend real analytically
on $y$ by the implicit function theorem because they are simple, while the inverse
is analytic by \eqref{eq:y-inverse-roots}.
\end{proof}

\begin{remark}[Stieltjes-transform interpretation]
The rational function $m_y$ is the Stieltjes transform, up to the sign convention
in the denominator, of the atomic measure $\sum_i y_i\delta_{g_i}$.  Its zeros are
the eigenvalues of $\mathcal A_y=P_yG|_{E_y}$ and provide the coordinates in
\Cref{thm:root-diffeo}.
\end{remark}

\subsection{Differential identities and the spectral root frame}

For vectors in $\mathbb R^N$ we occasionally write $(a\mid b):=a\cdot b$.
For a scalar state function, $D_yz$ denotes its state differential and
$\nabla_yz$ the corresponding Euclidean gradient.  Since
$m_{cy}=c\,m_y$ for $c>0$, each root has the canonical homogeneous extension
$z_j(cy)=z_j(y)$ to the positive cone.  Its differential restricted to
$E$ is therefore the intrinsic differential on the simplex.

For $j=1,\ldots,N-1$, define
\begin{equation}\label{eq:gamma-def}
        \gamma_j(y):=\biggl(\sum_{i=1}^N
        \frac{y_i}{(g_i-z_j)^2}\biggr)^{-1}>0.
\end{equation}
Differentiating $m_y(z_j(y))=0$ gives
\begin{equation}\label{eq:grad-z}
        \partial_{y_i}z_j=-\frac{\gamma_j}{g_i-z_j},
\end{equation}
or, with $q_j:=q(z_j)$,
\begin{equation}\label{eq:gradient-is-mode}
        \nabla_yz_j=-\gamma_jq_j.
\end{equation}
Thus the left spectral covector field is exact.

The inverse root map supplies the dual right spectral frame.  Set
\begin{equation}\label{eq:root-coordinate-frame}
        e_j:=\partial_{z_j}y(z)
        =-(G-z_jI)^{-1}y
        =-Yq_j
        =\frac1{\gamma_j}Y\nabla_yz_j\in E.
\end{equation}
Since $y\leftrightarrow z$ is a diffeomorphism,
\begin{equation}\label{eq:root-frame-duality}
        D_yz_k[e_j]=\delta_{jk}.
\end{equation}
Hence $\{e_j\}$ is the coordinate basis of $E$ and $\{D_yz_j\}$ its dual
basis.

\begin{lemma}[Weighted orthogonality]\label{lem:Y-orthogonality}
For $j,k\in\{1,\ldots,N-1\}$,
\begin{equation}\label{eq:Y-orthogonality}
        (Y\nabla_y z_j\mid \nabla_y z_k)
        =\delta_{jk}\gamma_j.
\end{equation}
Equivalently,
\begin{equation}\label{eq:q-weighted-orthogonality}
        q_j^{\sf T}Yq_k=\frac{\delta_{jk}}{\gamma_j}.
\end{equation}
\end{lemma}

\begin{proof}
By \eqref{eq:root-coordinate-frame}--\eqref{eq:root-frame-duality},
\[
 \delta_{jk}=D_yz_k[e_j]
 =\gamma_k q_k^{\sf T}Yq_j.
\]
This is \eqref{eq:q-weighted-orthogonality};
\eqref{eq:Y-orthogonality} then follows from
\eqref{eq:gradient-is-mode}.
\end{proof}

\begin{proposition}[Spectral frame and modal resolution]
\label{prop:modal-inverse}
Define the rank-one coordinate projectors on $E$ by
\begin{equation}\label{eq:root-coordinate-projectors}
        \Pi_j:=e_j\otimes D_yz_j,
        \qquad
        \Pi_jh=e_j\,D_yz_j[h]
        =\gamma_j(q_j\cdot h)Yq_j.
\end{equation}
Then
\begin{equation}\label{eq:modal-projectors}
        I_E=\sum_{j=1}^{N-1}\Pi_j,
        \qquad
        \Pi_j\Pi_k=\delta_{jk}\Pi_j,
\end{equation}
and
\begin{align}
 B(y)|_E
   &=-\sum_{j=1}^{N-1}\lambda_j\Pi_j
     =-\sum_{j=1}^{N-1}\lambda_j\gamma_j
       (Yq_j)q_j^{\sf T},
       \label{eq:B-modal-resolution}\\
 \bigl(B(y)|_E\bigr)^{-1}
   &=-\sum_{j=1}^{N-1}\frac1{\lambda_j}\Pi_j
     =-\sum_{j=1}^{N-1}\frac{\gamma_j}{\lambda_j}
       (Yq_j)q_j^{\sf T}.
       \label{eq:B-inverse-modal}
\end{align}
Thus the diagonal-plus-rank-one formula
\eqref{eq:inverse-B-rank-one} and the modal formula
\eqref{eq:B-inverse-modal} are two representations of the same inverse.
\end{proposition}

\begin{proof}
The projector identities are immediate from the duality
\eqref{eq:root-frame-duality}.  By \eqref{eq:BY-minus-YL} and
\eqref{eq:Ay-qj},
\[
        B(y)e_j=-\lambda_je_j.
\]
Applying $B(y)$ and its inverse to the coordinate decomposition
$h=\sum_j e_jD_yz_j[h]$ gives the two modal resolutions.
\end{proof}

\begin{lemma}[Spectral identity]\label{lem:spectral-identity}
For every $j=1,\ldots,N-1$,
\begin{equation}\label{eq:spectral-identity}
        B(y)^{\sf T}\nabla_y z_j=-\lambda_j(y)\nabla_y z_j,
        \qquad
        \lambda_j(y):=s+z_j.
\end{equation}
In particular, $\lambda_j>0$, and
\begin{equation}\label{eq:lambda-diff}
        \lambda_j-\lambda_k=z_j-z_k.
\end{equation}
\end{lemma}

\begin{proof}
The modal identity $B(y)e_k=-\lambda_ke_k$ and the duality
\eqref{eq:root-frame-duality} give, intrinsically on $E$,
\[
        D_yz_j\circ B(y)|_E=-\lambda_jD_yz_j.
\]
Thus $B(y)^{\sf T}\nabla_yz_j+\lambda_j\nabla_yz_j$ annihilates $E$ and is a
multiple of $\one$.  Its scalar product with $y$ vanishes because
$B(y)y=0$ and the homogeneous extension satisfies
$y\cdot\nabla_yz_j=0$.  Since $\one\cdot y=1$, that multiple is zero, proving
\eqref{eq:spectral-identity}.
\end{proof}

The same coordinate frame also yields explicit identities for the nonlinear
second derivatives.

\begin{lemma}[Affine root calculus and Hessian identity]
\label{lem:hessian-identity}
The inverse root map satisfies
\begin{equation}\label{eq:affine-root-connection}
        \partial_{z_jz_j}y=0,
        \qquad
        (z_k-z_j)\partial_{z_jz_k}y=e_j-e_k
        \quad(j\ne k).
\end{equation}
For $V,W\in E$, set
$c_r:=D_yz_r[V]$ and $d_r:=D_yz_r[W]$.  Then
\begin{equation}\label{eq:hessian-bilinear}
 D_y^2z_j[V,W]
 =\sum_{r\ne j}\frac{c_jd_r+c_rd_j}{z_j-z_r}.
\end{equation}
In particular,
\begin{equation}\label{eq:hessian-identity}
        D_y^2 z_j[V,V]
        =2c_j\sum_{r\ne j}\frac{c_r}{z_j-z_r}.
\end{equation}
\end{lemma}

\begin{proof}
Formula \eqref{eq:y-inverse-roots} is affine in each individual root and gives
$\partial_{z_j}y=e_j=-(G-z_jI)^{-1}y$.  A second derivative, together with
the resolvent identity
$R_{z_k}-R_{z_j}=(z_k-z_j)R_{z_j}R_{z_k}$, gives
\eqref{eq:affine-root-connection}.  Differentiating the identity
$z(y(z))=z$ twice gives
\[
 D_y^2z_j[e_a,e_b]
 =-D_yz_j[\partial_{z_az_b}y].
\]
Using \eqref{eq:affine-root-connection} and
$D_yz_j[e_a]=\delta_{ja}$ shows that the only nonzero entries are
\[
 D_y^2z_j[e_j,e_r]=D_y^2z_j[e_r,e_j]
 =\frac1{z_j-z_r},\qquad r\ne j.
\]
Expanding $V=\sum_rc_re_r$ and $W=\sum_rd_re_r$ proves
\eqref{eq:hessian-bilinear} and hence \eqref{eq:hessian-identity}.
\end{proof}

\begin{corollary}[Quadratic correction]\label{lem:quadratic-correction}
For every $V\in E$, with $c_r=D_yz_r[V]$,
\begin{equation}\label{eq:quadratic-correction}
        -D_y^2 z_j[B(y)V,V]
        =c_j\sum_{r\ne j}
          \frac{\lambda_j+\lambda_r}{\lambda_j-\lambda_r}\,c_r.
\end{equation}
\end{corollary}

\begin{proof}
Since $B(y)e_r=-\lambda_re_r$, insert
$V=\sum_rc_re_r$ and $B(y)V=-\sum_r\lambda_rc_re_r$ into
\eqref{eq:hessian-bilinear}, and use
$\lambda_j-\lambda_r=z_j-z_r$.
\end{proof}

For later use, the preceding identities admit a compact affine interpretation.
If $\Psi=\Psi(z)$ and $\widetilde\Psi(y):=\Psi(z(y))$, define the
concentration-affine Hessian of $\Psi$ in the root frame by
\begin{equation}\label{eq:affine-root-Hessian-def}
 \mathcal H^{\mathrm{aff}}_{jk}(\Psi)
 :=D_y^2\widetilde\Psi[e_j,e_k].
\end{equation}
Only tangent directions $e_j,e_k\in E$ occur here, so this Hessian is
intrinsic to the affine hyperplane $\one^{\sf T}y=1$ and does not depend on
how $\widetilde\Psi$ is extended off that hyperplane.  Then \eqref{eq:hessian-bilinear} and the chain rule give
\begin{equation}\label{eq:affine-root-Hessian-components}
 \mathcal H^{\mathrm{aff}}_{jk}(\Psi)
 =\begin{cases}
   \Psi_{jj}, & j=k,\\[1mm]
   \displaystyle
   \Psi_{jk}+\frac{\Psi_j-\Psi_k}{z_j-z_k}, & j\ne k.
  \end{cases}
\end{equation}
Consequently the multi-EPD relations
\begin{equation}\label{eq:EPD-as-affine-diagonal}
        (z_k-z_j)\Psi_{jk}=\Psi_j-\Psi_k,
        \qquad j\ne k,
\end{equation}
are precisely the condition that the concentration-affine Hessian
$\mathcal H^{\mathrm{aff}}(\Psi)$ be diagonal in the spectral root frame.

\begin{remark}[Relation to earlier spectral representations]
Characteristic-root and eigenvalue representations have appeared earlier in
steady multicomponent Maxwell--Stefan mass-transfer analysis; see, for example,
\cite{Taylor1981Exact}.  Here \eqref{eq:gradient-is-mode},
\eqref{eq:spectral-identity}, and \Cref{thm:root-diffeo} show that the
corresponding left spectral covectors are exact and furnish a global nonlinear
coordinate chart for the time-dependent diffusion system.  For $N=3$, every
friction triple has a unique real pair-sum representation.  If the three
resulting values $g_i$ are pairwise different, the same secular construction
applies without the positivity assumption $g_i>0$; repeated values are covered
by the grouped description of \Cref{rem:repeated-generators}.
\end{remark}

\section{The root-variable parabolic system}

We now derive the PDE satisfied by the root variables.

Let $y$ be a classical solution of \eqref{eq:balance}--\eqref{eq:MS-law}.
From \eqref{eq:spectral-identity}, for each spatial direction
$\alpha=1,\ldots,d$,
\begin{equation}\label{eq:z-flux-identity}
        (\nabla_y z_j\mid J^\alpha)
        =-\frac{\partial_{x_\alpha}z_j}{\lambda_j}.
\end{equation}
{For the normal flux vector write
$J^\nu:=(J_1\cdot\nu,\ldots,J_N\cdot\nu)^{\sf T}$.}
The boundary condition \eqref{eq:noflux} implies
\begin{equation}\label{eq:z-neumann}
        \partial_\nu z_j=0,
        \qquad j=1,\ldots,N-1.
\end{equation}
Indeed,
\[
\partial_\nu z_j
=(\nabla_y z_j\mid \partial_\nu y)
=(\nabla_y z_j\mid B(y)J^\nu)
=(B(y)^{\sf T}\nabla_y z_j\mid J^\nu)=0,
\]
because $J_i\cdot\nu=0$ for every species.

\subsection{Equations in root variables}

\begin{theorem}[Root-variable system]\label{thm:root-system}
{Let $y$ be a classical solution of \eqref{eq:balance}--\eqref{eq:MS-law} in the additive class.}  Then the root variables satisfy, for $j=1,\ldots,N-1$,
\begin{equation}\label{eq:root-system}
        \partial_t z_j
        -\Div\left(\frac{\nabla z_j}{\lambda_j}\right)
        +\sum_{k\ne j}
          \frac{\lambda_j+\lambda_k}
               {\lambda_j\lambda_k(\lambda_j-\lambda_k)}
          \nabla z_j\cdot\nabla z_k=0,
\end{equation}
with the homogeneous Neumann boundary condition \eqref{eq:z-neumann}.
\end{theorem}

For later use we set
\begin{equation}\label{eq:root-diffusion-aj}
 \mathfrak a_j(z):=\lambda_j(z)^{-1},\qquad
 C_{jk}(z):=\frac{\lambda_j+\lambda_k}
 {\lambda_j\lambda_k(\lambda_j-\lambda_k)}
 =\frac{\mathfrak a_j(z)+\mathfrak a_k(z)}{z_j-z_k}
 =-C_{kj}(z).
\end{equation}
Thus $\mathfrak a_j$ is the scalar diffusion coefficient in the $j$-th root equation and $C_{jk}$ is skew-symmetric.

The only nonlinear chain-rule term is the Hessian correction generated by
$D_y^2z_j$.  In the spectral root frame this correction is already given by
\Cref{lem:quadratic-correction}, so no further component calculation is
needed.

\begin{proof}[Proof of \Cref{thm:root-system}]
Using \eqref{eq:balance},
\begin{align*}
\partial_t z_j
&=(\nabla_y z_j\mid \partial_t y)
=-(\nabla_y z_j\mid \Div J)\\
&=-\Div(J^{\sf T}\nabla_y z_j)
  +(\partial_{x_\alpha}\nabla_y z_j\mid J^\alpha),
\end{align*}
where summation over the spatial index $\alpha$ is understood.  By \eqref{eq:z-flux-identity},
\[
        (\nabla_y z_j\mid J^\alpha)
        =-\frac{\partial_{x_\alpha}z_j}{\lambda_j}.
\]
Therefore
\[
        \partial_t z_j-
        \Div\left(\frac{\nabla z_j}{\lambda_j}\right)
        -(\partial_{x_\alpha}\nabla_y z_j\mid J^\alpha)=0.
\]
Since $\partial_{x_\alpha}y=B(y)J^\alpha$,
\[
        (\partial_{x_\alpha}\nabla_y z_j\mid J^\alpha)
        =D_y^2z_j[B(y)J^\alpha,J^\alpha].
\]
Now apply \Cref{lem:quadratic-correction} and use
\[
        c_r^\alpha=(\nabla_y z_r\mid J^\alpha)
        =-\frac{\partial_{x_\alpha}z_r}{\lambda_r}.
\]
This gives exactly \eqref{eq:root-system}.  The boundary condition was shown in \eqref{eq:z-neumann}; the final identity in \eqref{eq:root-diffusion-aj} follows from \eqref{eq:lambda-diff}.

\end{proof}

\begin{remark}[Absence of self-gradient square terms]
The absence of $|\nabla z_j|^2$ follows from two independent identities.  First,
$\partial_{z_jz_j}y=0$ in \eqref{eq:affine-root-connection}, equivalently
$D_y^2z_j[e_j,e_j]=0$, so the nonlinear coordinate Hessian has no self-mode
contribution.  Second, the trace identity
\eqref{eq:general-root-gradient-trace} gives $\partial_{z_j}\lambda_j=0$, so
expanding $-\Div(\lambda_j^{-1}\nabla z_j)$ also produces no self-square.
Thus the $j$-th equation has a scalar principal part depending only on the
transverse roots, and every quadratic correction is mixed.  This is the
structural input for the componentwise comparison argument in
\Cref{thm:root-rectangle}.
\end{remark}

\begin{corollary}[Root-sum conservation law and first-moment flux]
\label{cor:root-sum-first-moment}
For every classical additive Maxwell--Stefan solution,
\begin{equation}\label{eq:general-root-sum-equation}
 \partial_t\left(\sum_{j=1}^{N-1}z_j\right)
 -\Div\left(\sum_{j=1}^{N-1}\frac{\nabla z_j}{\lambda_j}\right)=0,
 \qquad
 \partial_\nu\left(\sum_{j=1}^{N-1}z_j\right)=0.
\end{equation}
Equivalently, since $s=e_1-\sum_jz_j$ {(and $s=U_1$ in the moment notation
introduced in \Cref{eq:U-moments})}, if
$F_1:=\sum_i g_iJ_i=g\cdot J$ denotes the first moment flux, then
\begin{equation}\label{eq:first-moment-flux-root-form}
 \partial_ts+\Div F_1=0,
 \qquad
 F_1=\sum_{j=1}^{N-1}\frac{\nabla z_j}{\lambda_j}.
\end{equation}
Thus the spatial integral of the root sum is conserved, and the explicit
quadratic corrections in the individual root equations cancel pairwise.
\end{corollary}

\begin{proof}
Sum \eqref{eq:root-system} over $j$.  Since
$C_{jk}=-C_{kj}$ while
$\nabla z_j\cdot\nabla z_k$ is symmetric in $j,k$, the double sum of the
explicit quadratic corrections vanishes.  This proves
\eqref{eq:general-root-sum-equation}, including its Neumann condition.
The trace identity \eqref{eq:general-root-trace-identity} gives the balance
law for $s$.  Its differential on $E$ is
\[
        g\cdot V=D_ys[V]=-\sum_rD_yz_r[V],\qquad V\in E.
\]
Taking $V=J^\alpha$ and using \eqref{eq:z-flux-identity} gives
$g\cdot J^\alpha=\sum_r\partial_{x_\alpha}z_r/\lambda_r$, which proves the pointwise
flux identity in \eqref{eq:first-moment-flux-root-form}.
\end{proof}

\begin{remark}[Mixed-gradient terms from variation of the spectral frame]
\label{rem:modal-meaning-root-PDE}
Write $J^\alpha=Yu^\alpha$ with $u^\alpha\in E_y$.  By
\eqref{eq:gradient-is-mode},
\begin{equation}\label{eq:modal-flux-amplitude}
        c_j^\alpha
        =(\nabla_yz_j\mid J^\alpha)
        =-\gamma_j\langle q(z_j),u^\alpha\rangle_y
        =-\frac{\partial_{x_\alpha}z_j}{\lambda_j}.
\end{equation}
{Thus $-c_j^\alpha$ is the coefficient of $q(z_j)$ in the orthogonal
eigenbasis expansion of $u^\alpha$.}  The explicit mixed-gradient terms outside
the diagonal divergence flux in \eqref{eq:root-system} arise when the
$y$-dependence of these eigenvectors is differentiated.  Their coefficients are
antisymmetric in the root indices $j,r$.
For $N=2$ there is only one root variable, and these terms are absent.
\end{remark}

\section{Invariant rectangles in root coordinates}

Using \eqref{eq:root-diffusion-aj}, the $j$-th equation in
\eqref{eq:root-system} can be written in the scalar divergence form
\begin{equation}\label{eq:zj-scalar-form}
 \partial_tz_j-
 \Div\bigl({\mathfrak a_j(z)}\nabla z_j\bigr)
 +W_j(z,\nabla z_{\ne j})\cdot\nabla z_j=0,
\end{equation}
where
\begin{equation}\label{eq:zj-drift}
 W_j(z,\nabla z_{\ne j})
 =\sum_{k\ne j}C_{jk}(z)\nabla z_k.
\end{equation}
Thus every lower-order term in the $j$-th equation contains the factor
$\nabla z_j$.  On a compact root range in $\cZ$, $\mathfrak a_j$ is
uniformly positive and the coefficients $C_{jk}$ are bounded.  This scalar
form yields the componentwise comparison argument below.

\begin{theorem}[Forward-invariant root rectangle]\label{thm:root-rectangle}
Let $\Omega\subset\mathbb R^d$ be a bounded domain of class $C^2$, and let
$z\in C([0,T);C^1(\overline\Omega;\mathbb R^{N-1}))$ take values in the
{interlacing box $\cZ$.  Assume that there exists a fixed $p>d+2$ such that, for every
$0<\tau<T'<T$,}
\[
 z\in W^{1,p}(\tau,T';L^p(\Omega))
     \cap L^p(\tau,T';W^{2,p}(\Omega)),
\]
that the root system \eqref{eq:root-system} holds distributionally on that
strip, and that $\partial_\nu z_j=0$ in the Sobolev trace sense.  Classical
solutions and the positive-time strong solutions constructed below are
included.  Assume that the initial root variables satisfy
\begin{equation}\label{eq:initial-root-rectangle}
        a_j\le z_j(0,x)\le b_j,
        \qquad x\in\overline\Omega,
        \qquad j=1,\ldots,N-1,
\end{equation}
where
\begin{equation}\label{eq:rectangle-inside}
        g_j<a_j\le b_j<g_{j+1}.
\end{equation}
Then
\begin{equation}\label{eq:rectangle-invariance}
        a_j\le z_j(t,x)\le b_j,
        \qquad (t,x)\in[0,T)\times\overline\Omega,
        \qquad j=1,\ldots,N-1.
\end{equation}
Moreover, with
\[
 \underline z_j(t):=\min_{\overline\Omega}z_j(t,\cdot),
 \qquad
 \overline z_j(t):=\max_{\overline\Omega}z_j(t,\cdot),
\]
one has, for $0\le t_1\le t_2<T$,
\begin{equation}\label{eq:root-extrema-monotonicity}
 g_j<\underline z_j(t_1)\le\underline z_j(t_2)
 \le\overline z_j(t_2)\le\overline z_j(t_1)<g_{j+1}.
\end{equation}
In particular, each coordinatewise oscillation
$\overline z_j(t)-\underline z_j(t)$ is nonincreasing.
Consequently, $y(t,x)$ remains in a compact subset of the open simplex.
\end{theorem}

\begin{proof}
Fix a sufficiently small $\delta>0$ such that the enlarged rectangle
\[
 R_\delta:=\prod_{j=1}^{N-1}[a_j-\delta,b_j+\delta]
 \Subset\prod_{j=1}^{N-1}(g_j,g_{j+1}).
\]
Let $T_\delta$ be the first exit time of $z$ from $R_\delta$, with the
convention $T_\delta=T$ if no exit occurs.  The initial range has distance at
least $\delta$ from $\partial R_\delta$; hence continuity at $t=0$ gives
$T_\delta>0$.

Suppose for contradiction that $T_\delta<T$.  By continuity, there is
$\tau\in(0,T_\delta)$ such that
\begin{equation}\label{eq:strict-inner-rectangle-at-tau}
 a_j-\frac\delta2\le z_j(\tau,x)\le b_j+\frac\delta2,
 \qquad x\in\overline\Omega,
 \qquad j=1,\ldots,N-1.
\end{equation}
Fix $T'\in(\tau,T_\delta)$.  On the compact strip
$[\tau,T']\times\overline\Omega$ {the root equation holds in the stated strong Sobolev class}, and its
range lies in $R_\delta$.  In \eqref{eq:zj-scalar-form}, the coefficient
{$\mathfrak a_j$} is uniformly positive on $R_\delta$, and the drift $W_j$
is bounded on the strip because $z\in C([\tau,T'];C^1)$ and all $C_{jk}$ are
bounded on $R_\delta$.

We record the comparison step directly in the strong class.  Put
$M_j=b_j+\delta/2$ and $r=(z_j-M_j)_+$.  Testing
\eqref{eq:zj-scalar-form} by $r$ (equivalently, using the standard Lipschitz
truncation approximation) and the homogeneous Neumann trace gives for almost
every time
\[
 \frac12\frac{\mathrm d}{\mathrm dt}\|r\|_2^2
 +a_*\|\nabla r\|_2^2
 \le \|W_j\|_\infty\int_\Omega r|\nabla r|
 \le \frac{a_*}{2}\|\nabla r\|_2^2+C\|r\|_2^2,
\]
where $a_*>0$ is the minimum of {$\mathfrak a_j$} on $R_\delta$.  Since
$r(\tau)=0$ by \eqref{eq:strict-inner-rectangle-at-tau}, Gronwall's lemma
gives $r\equiv0$.  Applying the same argument to
$(a_j-\delta/2-z_j)_+$ gives the lower bound.  Thus
\[
 a_j-\frac\delta2\le z_j(t,x)\le b_j+\frac\delta2,
 \qquad (t,x)\in[\tau,T']\times\overline\Omega.
\]
All integrations are legitimate for the stated maximal-regularity class:
$z_{j,t},D^2z_j\in L^p$, $\nabla z_j\in L^\infty$, and the truncation belongs
to the corresponding energy class.  Hence no classical comparison theorem
is being imported here.
Letting $T'\uparrow T_\delta$ and using continuity shows that the range at
$t=T_\delta$ still has distance at least $\delta/2$ from
$\partial R_\delta$.  This contradicts the definition of the first exit time.
Thus $T_\delta=T$.

The conclusion holds for every sufficiently small $\delta>0$.  Letting
$\delta\downarrow0$ gives \eqref{eq:rectangle-invariance}.

To prove \eqref{eq:root-extrema-monotonicity}, fix
$0\le t_1\le t_2<T$ and restart the preceding invariant-rectangle argument
at time $t_1$, using the actual extrema
$\underline z_j(t_1)$ and $\overline z_j(t_1)$ as the rectangle endpoints.
For $t_1>0$ the shifted solution satisfies the same strong root hypotheses; for
$t_1=0$ the preceding argument applies directly.  Continuity and strict
interlacing on the compact spatial domain place these extrema strictly inside
$(g_j,g_{j+1})$.  The restarted comparison therefore gives
\eqref{eq:root-extrema-monotonicity}.

Finally, the inverse formula \eqref{eq:y-inverse-roots} is continuous and
strictly positive on the compact rectangle $\prod_j[a_j,b_j]$.  Hence
\[
 \min_{1\le i\le N}\min_{z\in\prod_j[a_j,b_j]}y_i(z)>0,
\]
which proves the last assertion.
\end{proof}

\begin{corollary}[Uniform positivity along strong solutions]\label{cor:uniform-positivity-classical}
If $y(0,\cdot)$ is uniformly separated from the boundary of the simplex, then every solution covered by \Cref{thm:root-rectangle} admits $c>0$, depending on the initial upper and lower bounds for the root variables, such that
\[
        y_i(t,x)\ge c,
        \qquad i=1,\ldots,N,
\]
for all times for which the strong solution exists.
\end{corollary}

\begin{remark}[Boundary separation and continuation]
The corollary excludes approach to the boundary of the simplex on the maximal
strong-existence interval.  Thus finite-time continuation can fail only through
loss of the regularity bounds required by the local theory.  The regularity and
continuation estimates developed below exclude this alternative for the
additive class.
\end{remark}

\section{Affine moment coordinates and the closed moment system}

The root variables are adapted to comparison, but the global regularity
argument also needs affine coordinates that preserve divergence form.  Both
coordinate systems are encoded by the same scalar resolvent function.  Indeed,
for $|z|>g_N$,
\begin{equation}\label{eq:resolvent-moment-expansion}
 m_y(z)=\one^{\sf T}(G-zI)^{-1}y
       =-\sum_{\ell=0}^{\infty}\frac{U_\ell}{z^{\ell+1}},
 \qquad
 U_\ell:=\one^{\sf T}G^\ell y=\sum_{i=1}^Ng_i^\ell y_i,
 \qquad U_0=1.
\end{equation}
Together with $m_y=-Q_y/P$, this says that the roots are the zeros of the
numerator $Q_y$ of the scalar resolvent function, while the moments are its
Laurent coefficients at infinity.

\subsection{Moment map}

Define
\begin{equation}\label{eq:U-moments}
        U_\ell:=\sum_{i=1}^N g_i^\ell y_i,
        \qquad \ell=1,\ldots,N-1,
\end{equation}
and set $U_0:=1$.  Thus $U_1=s$.  Let
\begin{equation}\label{eq:moment-vertices}
        \Gamma_i:=(g_i,g_i^2,\ldots,g_i^{N-1})^{\sf T}\in\R^{N-1},
\end{equation}
and
\begin{equation}\label{eq:moment-simplex}
        \cT:=\co\{\Gamma_1,\ldots,\Gamma_N\}\subset\R^{N-1}.
\end{equation}
The vertices are affinely independent by the Vandermonde determinant.

\begin{proposition}[Affine moment bijection]\label{prop:moment-bijection}
The map $y\mapsto U=(U_1,\ldots,U_{N-1})$ is an affine bijection from the
concentration simplex onto $\cT$.  If
\begin{equation}\label{eq:Lagrange-polynomial}
        L_i(\xi):=\prod_{m\ne i}\frac{\xi-g_m}{g_i-g_m}
        =\sum_{\ell=0}^{N-1}\ell_{i\ell}\xi^\ell,
\end{equation}
then its inverse is
\begin{equation}\label{eq:y-inverse-moments}
        y_i(U)=\sum_{\ell=0}^{N-1}\ell_{i\ell}U_\ell,
        \qquad U_0=1.
\end{equation}
\end{proposition}

\begin{proof}
For every polynomial $p(\xi)=\sum_{\ell=0}^{N-1}p_\ell\xi^\ell$ of degree at
most $N-1$,
\[
        \sum_i y_ip(g_i)=\sum_{\ell=0}^{N-1}p_\ell U_\ell.
\]
Taking $p=L_i$ gives \eqref{eq:y-inverse-moments}.  The Vandermonde
determinant gives affine independence and hence bijectivity.
\end{proof}

We use the constant derivative of the affine moment map as standing notation:
\begin{equation}\label{eq:moment-linear-isomorphisms}
 \mathsf L:E\longrightarrow\mathbb R^{N-1},\qquad
 \mathsf Lv=\left(\sum_{i=1}^N g_i^\ell v_i\right)_{\ell=1}^{N-1},
 \qquad
 \mathsf R:=\mathsf L^{-1}.
\end{equation}
Thus $dU=\mathsf L\,dy$ and $dy=\mathsf R\,dU$.

\subsection{Characteristic polynomial and moments}

Let $e_r=e_r(g_1,\ldots,g_N)$ be the elementary symmetric polynomial of
degree $r$, with $e_0=1$.

\begin{proposition}[Characteristic coefficients and moment variables]
\label{prop:characteristic-moment-coordinates}
The characteristic polynomial of the compression has the representation
\begin{equation}\label{eq:Q-moment-coefficients}
        Q_y(z)
        =\det_E(zI_E-\mathcal C_y)
        =\det_{E_y}(zI_{E_y}-\mathcal A_y)
        =\sum_{r=0}^{N-1}(-1)^r c_r(U)z^{N-1-r},
\end{equation}
where
\begin{equation}\label{eq:cr-moment-formula}
        c_r(U)=\sum_{\ell=0}^r(-1)^\ell e_{r-\ell}U_\ell,
        \qquad r=0,\ldots,N-1.
\end{equation}
Hence every characteristic coefficient is affine in
$U_1,\ldots,U_{N-1}$.
\end{proposition}

\begin{proof}
By \eqref{eq:m-Q-over-P} and \eqref{eq:resolvent-moment-expansion},
\[
 \frac{Q_y(z)}{P(z)}
 =\sum_{\ell=0}^{\infty}\frac{U_\ell}{z^{\ell+1}}
 \qquad(|z|>g_N).
\]
Write $P(z)=\sum_{r=0}^N(-1)^re_rz^{N-r}$ and compare the coefficient of
$z^{N-1-r}$ in
$Q_y=P\sum_{\ell\ge0}U_\ell z^{-\ell-1}$.  This gives
\eqref{eq:cr-moment-formula}.
\end{proof}

Thus the roots and moments are complementary representations of the same
finite-dimensional spectral data: the former are the characteristic roots,
the latter affine coordinates for the characteristic coefficients.

\subsection{Closed moment system by operator conjugation}

Define the moment diffusion matrix directly from the constrained
Maxwell--Stefan inverse by
\begin{equation}\label{eq:AN-constrained-inverse-identity}
 A_N(U):=-\mathsf L\bigl(B(y(U))|_E\bigr)^{-1}\mathsf R.
\end{equation}
This formula is the affine-coordinate representation of the positive diffusion
operator $-\bigl(B(y)|_E\bigr)^{-1}$.

For the moment flux set
\begin{equation}\label{eq:F-ell-def}
        F:=\mathsf LJ,
        \qquad
        F_\ell=\sum_{i=1}^N g_i^\ell J_i.
\end{equation}

\begin{theorem}[Closed moment fluxes]\label{thm:moment-fluxes}
Let $y$ be a sufficiently smooth composition field with values in
$\overline{\cS}$, let $J$ satisfy $\one^{\sf T}J=0$ and
$\nabla y=B(y)J$, and let $U$ be its moment vector.  Then
\begin{equation}\label{eq:F-ell-formula}
        F=-A_N(U)\nabla U.
\end{equation}
If the species balances hold, then
\begin{equation}\label{eq:moment-system}
        \partial_tU-\Div\bigl(A_N(U)\nabla U\bigr)=0.
\end{equation}
Moreover, $A_N$ extends real-analytically to a neighbourhood of the closed
moment simplex $\cT$ and remains invertible there.
\end{theorem}

\begin{proof}
Since $\nabla y=\mathsf R\nabla U$ and
$J=(B(y)|_E)^{-1}\mathsf R\nabla U$, applying $\mathsf L$ gives
\eqref{eq:F-ell-formula} directly from
\eqref{eq:AN-constrained-inverse-identity}.  Applying $\mathsf L$ to the
species balances gives \eqref{eq:moment-system}.  The analytic extension and
invertibility follow from the affine inverse moment map and
\Cref{prop:flux-inverse}.
\end{proof}

\begin{proposition}[Operator similarity and spectrum of the moment diffusion]
\label{prop:moment-PDE-spectrum}
\label{prop:moment-relaxation-similarity}
Let $y\in\cS^\circ$ and let $U$ be its moment vector.  Multiplication by
$Y=\diag(y_1,\ldots,y_N)$ is an isomorphism $Y:E_y\to E$, and
\begin{equation}\label{eq:B-inverse-relaxation-similarity}
  \bigl(B(y)|_E\bigr)^{-1}
  =-Y\bigl(\mathcal L_y|_{E_y}\bigr)^{-1}Y^{-1}.
\end{equation}
Consequently, with
\[
        \mathfrak S_y:=\mathsf LY:E_y\to\mathbb R^{N-1},
\]
\begin{equation}\label{eq:AN-relaxation-inverse-similarity}
  A_N(U)
  =\mathfrak S_y\bigl(\mathcal L_y|_{E_y}\bigr)^{-1}\mathfrak S_y^{-1}.
\end{equation}
Hence
\begin{equation}\label{eq:spectrum-AN-reciprocal}
        \sigma(A_N(U))=
        \left\{\lambda_1^{-1},\ldots,\lambda_{N-1}^{-1}\right\}.
\end{equation}
If $\mathsf J_U(z):=D_zU(z)$, then
\begin{equation}\label{eq:AN-root-Jacobian-diagonalization}
        A_N(U(z))\mathsf J_U(z)
        =\mathsf J_U(z)\operatorname{diag}
        (\lambda_1^{-1},\ldots,\lambda_{N-1}^{-1}).
\end{equation}
Equivalently, for $p_j(U):=\nabla_Uz_j(U)$,
\begin{equation}\label{eq:left-eigenvectors-moment-A}
        p_j(U)^{\sf T}A_N(U)=\lambda_j^{-1}p_j(U)^{\sf T}.
\end{equation}
\end{proposition}

\begin{proof}
The identity $B(y)Y=-Y\mathcal L_y$ on $E_y$ gives
\eqref{eq:B-inverse-relaxation-similarity}; substitution into
\eqref{eq:AN-constrained-inverse-identity} gives
\eqref{eq:AN-relaxation-inverse-similarity}.  This proves the spectral
statement.  Alternatively, the $j$-th column of $D_zU$ is
$\mathsf Le_j$, and
\[
 A_N(\mathsf Le_j)
 =-\mathsf L(B|_E)^{-1}e_j
 =\lambda_j^{-1}\mathsf Le_j,
\]
which is \eqref{eq:AN-root-Jacobian-diagonalization}.  Dualizing gives
\eqref{eq:left-eigenvectors-moment-A}.
\end{proof}

The signs in the three representations are summarized by
\begin{equation}\label{eq:sign-consistency-ledger}
 B(y)e_j=-\lambda_je_j,
 \qquad
 B(y)^{\sf T}\nabla_yz_j=-\lambda_j\nabla_yz_j,
 \qquad
 A_N(U)(\mathsf Le_j)=\lambda_j^{-1}\mathsf Le_j.
\end{equation}
Thus the negative constrained Maxwell--Stefan spectrum and the positive moment
diffusion spectrum are reciprocal.

\begin{proposition}[Uniform normal ellipticity up to the closed moment simplex]
\label{prop:closed-simplex-normal-ellipticity}
With
\[
 f_{\min}:=\min_{i\ne j}f_{ij}=g_1+g_2,
 \qquad
 f_{\max}:=\max_{i\ne j}f_{ij}=g_{N-1}+g_N,
\]
one has
\begin{equation}\label{eq:closed-simplex-spectral-bounds}
 \sigma(A_N(U))\subset
 \left[\frac1{f_{\max}},\frac1{f_{\min}}\right],
 \qquad U\in\cT.
\end{equation}
In particular, $A_N$ is uniformly normally elliptic on the whole closed moment
simplex and
\[
 \sup_{U\in\cT}
 \bigl(\|A_N(U)\|+\|A_N(U)^{-1}\|\bigr)<\infty.
\]
\end{proposition}

\begin{proof}
For $u\in E_y$ the weighted variance identity gives
\[
 \frac12\sum_{i,j=1}^N y_iy_j|u_i-u_j|^2=\|u\|_y^2.
\]
Since $f_{\min}\le f_{ij}\le f_{\max}$,
\eqref{eq:rayleigh-general} and \eqref{eq:D-L-inner} imply
\[
 f_{\min}\|u\|_y^2
 \le\langle\mathcal L_yu,u\rangle_y
 \le f_{\max}\|u\|_y^2.
\]
The min--max principle and \eqref{eq:AN-relaxation-inverse-similarity} yield
\eqref{eq:closed-simplex-spectral-bounds} in $\cT^\circ$.  The analytic
extension from \Cref{thm:moment-fluxes} and continuity of characteristic roots
extend the same bound to $\cT$.

For completeness, fix a closed sector
$\overline\Sigma_{\pi-\vartheta}$ with $0<\vartheta<\pi/2$ and normalize
$|\lambda|+|\xi|^2=1$.  By \eqref{eq:closed-simplex-spectral-bounds},
$\lambda+|\xi|^2A_N(U)$ is invertible for every
$U\in\cT$ and every normalized
$(\lambda,\xi)\in\overline\Sigma_{\pi-\vartheta}\times\mathbb R^d$.
Compactness and continuity of inversion give a uniform resolvent bound;
parabolic scaling yields the usual uniform sectorial estimate.  The final
operator-norm bound follows from continuity of $A_N$ and $A_N^{-1}$ on the
compact simplex.
\end{proof}

\begin{remark}[Boundary behavior of ellipticity and entropy coercivity]
\label{rem:boundary-spectrum-versus-entropy}
Vanishing species do not cause spectral degeneration of the affine moment
diffusion matrix.  At the concentration boundary, however, the entropy Hessian
$H(U)$ introduced in \eqref{eq:H-formula} becomes unbounded in directions
associated with vanishing species, and the interior root coordinates degenerate
when roots reach the generator walls.  Thus normal ellipticity persists on the
closed simplex, whereas the entropy and Schauder estimates below use compact
subsets of its interior.
\end{remark}

The moment no-flux condition is $A_N(U)\partial_\nu U=0$.  Since $A_N$ is
invertible on $\cT$, it is equivalent to $\partial_\nu U=0$.

\begin{proposition}[Equivalence of the constrained and moment formulations]
\label{prop:constrained-moment-equivalence}
Let $U$ be sufficiently smooth with values in $\cT$, set
$y=y(U)\in\overline\cS$, and for every spatial direction $k$ let $J^k\in E$
be the unique vector satisfying
\begin{equation}\label{eq:reconstructed-flux}
        B(y)J^k=\partial_k y.
\end{equation}
Then
\begin{equation}\label{eq:reconstructed-moment-flux}
        \mathsf LJ^k=-A_N(U)\partial_kU.
\end{equation}
Consequently, in every strong class in which the displayed chain rules are
valid, the constrained Maxwell--Stefan system is equivalent to the moment
system \eqref{eq:moment-system}; the specieswise no-flux condition is
equivalent to $A_N(U)\partial_\nu U=0$, hence to $\partial_\nu U=0$.
\end{proposition}

\begin{proof}
Since $\partial_ky=\mathsf R\partial_kU$, formula
\eqref{eq:reconstructed-moment-flux} is exactly
\eqref{eq:AN-constrained-inverse-identity}.  Applying $\mathsf L$ to the
species balances gives the moment system.  Conversely, if $U$ solves the
moment system, reconstruct $J$ by \eqref{eq:reconstructed-flux}.  The residual
$r:=\partial_ty+\Div J$ belongs to $E$ and satisfies
\[
 \mathsf Lr
 =\partial_tU-\Div(A_N(U)\nabla U)=0.
\]
Since $\mathsf L:E\to\mathbb R^{N-1}$ is injective, $r=0$.  The same
injectivity identifies the boundary fluxes.
\end{proof}

\begin{corollary}[Strong-Sobolev root equations for moment solutions]
\label{cor:strong-sobolev-root-system}
Let $I\Subset(0,T)$, $p>d+2$, and suppose
\[
 U\in W^{1,p}(I;L^p(\Omega))\cap L^p(I;W^{2,p}(\Omega)),
\]
takes values in a compact subset of $\cT^\circ$, solves
\eqref{eq:moment-system}, and satisfies $\partial_\nu U=0$.  Then the
associated $y$ and $z$ belong to the same maximal-regularity class, and every
root satisfies \eqref{eq:root-system} distributionally together with
$\partial_\nu z_j=0$.
\end{corollary}

\begin{proof}
All coordinate maps are smooth with bounded derivatives on the compact state
range.  Parabolic Sobolev embedding gives $\nabla U\in L^\infty$, and the
Sobolev chain rule puts $y,z$ in the same maximal-regularity class.  Reconstruct
\[
        J^k=-\mathsf RA_N(U)\partial_kU.
\]
Then $B(y)J^k=\partial_ky$ by
\eqref{eq:AN-constrained-inverse-identity}, and
$J^k\in L^p(I;W^{1,p})$ because $A_N$ is smooth,
$\nabla U\in L^\infty$, and $D^2U\in L^p$.  Hence
$\partial_ty+\Div J=0$ in $L^p$ by
\Cref{prop:constrained-moment-equivalence}.

The Sobolev product rule gives
\[
 \partial_tz_j
 =-\Div(J^{\sf T}\nabla_yz_j)
   +\sum_kD_y^2z_j[\partial_ky,J^k].
\]
The pointwise identities \eqref{eq:z-flux-identity} and
\eqref{eq:quadratic-correction} convert this formula into
\eqref{eq:root-system}; the normal trace follows from the Sobolev chain rule.
\end{proof}

\section{Entropy symmetrization of the moment system}

The matrix $A_N(U)$ in \eqref{eq:moment-system} is generally not symmetric.  The
Hessian of the mixing entropy provides a symmetrizer.

\subsection{Entropy Hessian}

On the interior $\cT^\circ$ of the moment simplex define
\begin{equation}\label{eq:entropy-h}
        h(U):=\sum_{i=1}^N y_i(U)\log y_i(U).
\end{equation}
Since $U\mapsto y(U)$ is affine, the Hessian is
\begin{equation}\label{eq:H-formula}
        H(U):=D^2h(U)={\mathsf R^{\sf T}} Y^{-1}{\mathsf R},
\end{equation}
where {$\mathsf R:\R^{N-1}\to E$} is the constant derivative of the inverse moment map $U\mapsto y(U)$, and $Y=\diag(y_1,\ldots,y_N)$.  Equivalently,
\begin{equation}\label{eq:H-ai}
        {H(U)=\sum_{i=1}^N\frac{\nabla_U y_i\otimes\nabla_U y_i}{y_i(U)}.}
\end{equation}
Thus $H(U)$ is symmetric positive definite on $\cT^\circ$.

\subsection{Flux-space dissipation and the entropy metric}

For $V,W\in E$, write $V=Yu$ and $W=Yw$ with $u,w\in E_y$ and define
\begin{equation}\label{eq:dissipation-bilinear}
 \cD_y(V,W):=\langle \mathcal L_yu,w\rangle_y
 =\frac12\sum_{i,j=1}^N f_{ij}y_i y_j
 \left(\frac{V_i}{y_i}-\frac{V_j}{y_j}\right)
 \left(\frac{W_i}{y_i}-\frac{W_j}{y_j}\right).
\end{equation}
Thus $\cD_y(V,V)=\cD_y^{\mathrm{fr}}(u)$.  The intertwining identity
$B(y)Y=-Y\mathcal L_y$ gives immediately
\begin{equation}\label{eq:D-B-identity}
 \cD_y(V,W)=-\langle B(y)V,W\rangle_{y,*},
 \qquad
 \langle V,W\rangle_{y,*}:=V^{\sf T}Y^{-1}W.
\end{equation}
Since $\mathcal L_y|_{E_y}$ is self-adjoint and strictly positive in
$\langle\cdot,\cdot\rangle_y$, the constrained Maxwell--Stefan operator
$B(y)|_E$ is self-adjoint and strictly negative in the entropy metric
$\langle\cdot,\cdot\rangle_{y,*}$.  Hence
$-\bigl(B(y)|_E\bigr)^{-1}$ is self-adjoint and strictly positive in the same
metric.  The moment symmetrization below is exactly the pullback of this
operator statement by the affine coordinate map.

\subsection{Symmetrization by the entropy Hessian}

\begin{theorem}[Entropy-Hessian symmetrization]\label{thm:entropy-symmetrization}
Let $A_N(U)$ be the moment coefficient matrix from \eqref{eq:AN-constrained-inverse-identity}.  Then, for every $U\in\cT^\circ$,
\begin{equation}\label{eq:HA-symmetric}
        H(U)A_N(U)=A_N(U)^{\sf T} H(U).
\end{equation}
Moreover, $H(U)A_N(U)$ is positive definite.  Consequently, $A_N(U)$ is
self-adjoint and positive in the inner product
{$\langle v_1,v_2\rangle_{H(U)}:=v_1^{\sf T}H(U)v_2$}.
\end{theorem}

\begin{proof}
For $c,d\in\mathbb R^{N-1}$, use
$H=\mathsf R^{\sf T}Y^{-1}\mathsf R$ and
$A_N=-\mathsf L(B|_E)^{-1}\mathsf R$.  Since
$\mathsf R\mathsf L=I_E$,
\[
 d^{\sf T}H(U)A_N(U)c
 =\left\langle \mathsf Rd,\mathsf R A_N(U)c\right\rangle_{y,*}
 =-\left\langle \mathsf Rd,\bigl(B(y)|_E\bigr)^{-1}\mathsf Rc\right\rangle_{y,*}.
\]
The operator $-\bigl(B(y)|_E\bigr)^{-1}$ is self-adjoint and strictly
positive in $\langle\cdot,\cdot\rangle_{y,*}$ by
\eqref{eq:D-B-identity}.  The displayed form is therefore symmetric in
$c,d$ and strictly positive for $c=d\ne0$, proving
\eqref{eq:HA-symmetric} and the positivity of $HA_N$.
\end{proof}

\subsection{Entropy variables}

Define the entropy variables (cf. \cite{JungelStelzer2013})
\begin{equation}\label{eq:entropy-variables}
        w:=Dh(U)\in\R^{N-1}.
\end{equation}
The Legendre dual potential is
\begin{equation}\label{eq:Phi-dual}
        {h^*(w)}:=\log\left(\sum_{i=1}^N
        \exp\left(\sum_{\ell=1}^{N-1}g_i^\ell w_\ell\right)\right).
\end{equation}
Then
\begin{equation}\label{eq:dual-relations}
        U=D{h^*}(w),
        \qquad
        h(U)=U\cdot w-{h^*(w)},
\end{equation}
and
\begin{equation}\label{eq:y-softmax}
        y_i(w)=
        \frac{\exp\left(\sum_{\ell=1}^{N-1}g_i^\ell w_\ell\right)}
             {\sum_{r=1}^N\exp\left(\sum_{\ell=1}^{N-1}g_r^\ell w_\ell\right)}.
\end{equation}
We refer to \eqref{eq:y-softmax} as the softmax map, i.e. the normalized
exponential parametrization of the open simplex.
Furthermore,
\begin{equation}\label{eq:K-H-inverse}
        K(w):=D^2{h^*}(w)=H(U(w))^{-1}.
\end{equation}
The vectors $\Gamma_i$ in \eqref{eq:moment-vertices} are affinely independent;
hence $D^2{h^*}(w)$ is positive definite for every $w\in\mathbb R^{N-1}$.
Standard Legendre duality therefore gives mutually inverse real-analytic
diffeomorphisms
\[
 D{h^*}:\mathbb R^{N-1}\longrightarrow\mathcal T^\circ,
 \qquad
 Dh:\mathcal T^\circ\longrightarrow\mathbb R^{N-1}.
\]
We shall refer to $\mathcal W:=\mathbb R^{N-1}$ as the entropy-variable state
space.

The moment system becomes
\begin{equation}\label{eq:entropy-variable-system}
        K(w)\partial_t w-
        \Div\bigl(B_N(w)\nabla w\bigr)=0,
        \qquad
        B_N(w):=A_N(U(w))K(w).
\end{equation}
By Theorem~\ref{thm:entropy-symmetrization}, $B_N(w)$ is symmetric positive definite.  For every $\mathcal K_U\Subset\cT^\circ$, both $K(w)$ and $B_N(w)$ are smooth and uniformly positive definite on $Dh(\mathcal K_U)\Subset\mathcal W$.

\begin{corollary}[Equivalence of all three formulations]
\label{cor:formulation-equivalence}
In the regularity classes of \Cref{prop:constrained-moment-equivalence}, the
constrained Maxwell--Stefan system, the moment system, and the entropy-variable
system \eqref{eq:entropy-variable-system} are equivalent as long as their
states remain in the corresponding open state spaces.  Their specieswise
no-flux, moment no-flux, and homogeneous Neumann boundary conditions are
equivalent.
\end{corollary}

\begin{proof}
The constrained and moment formulations are equivalent by
\Cref{prop:constrained-moment-equivalence}.  The analytic diffeomorphisms
$U=D{h^*}(w)$ and $w=Dh(U)$ give
$\partial_tU=K(w)\partial_tw$ and $\nabla U=K(w)\nabla w$.
Substitution in the moment equation yields
\eqref{eq:entropy-variable-system}, with $B_N=A_NK$, and the converse follows
by the inverse change of variables.  Since $K(w)$ and $A_N(U)$ are invertible,
{$A_N(U)\partial_\nu U=0$}, $\partial_\nu U=0$, and
$\partial_\nu w=0$ are equivalent.  The specieswise boundary condition is
then equivalent by \Cref{prop:constrained-moment-equivalence}.
\end{proof}

\begin{corollary}[Entropy dissipation identity]\label{cor:entropy-dissipation}
{For a matrix $M$ acting on the component indices we use
$\nabla U:M\nabla U:=\sum_{k=1}^d(\partial_{x_k}U)^{\sf T}M\partial_{x_k}U$.}
For smooth solutions satisfying the no-flux condition,
\begin{equation}\label{eq:entropy-dissipation}
        \frac{d}{dt}\int_\Omega h(U)\,dx
        +\int_\Omega \nabla U:H(U)A_N(U)\nabla U\,dx=0.
\end{equation}
Equivalently,
\begin{equation}\label{eq:entropy-dissipation-w}
        \frac{d}{dt}\int_\Omega h(U)\,dx
        +\int_\Omega \nabla w:B_N(w)\nabla w\,dx=0.
\end{equation}
\end{corollary}

\begin{proof}
Multiply \eqref{eq:moment-system} by $w=Dh(U)$, integrate over $\Omega$, and use the no-flux boundary condition.  Since $\nabla w=H(U)\nabla U$, one obtains \eqref{eq:entropy-dissipation}.  The entropy-variable form follows from $\nabla U=K(w)\nabla w$ and $B_N=A_NK$.
\end{proof}

\begin{corollary}[Boundary-uniform Fisher-information coercivity]
\label{cor:boundary-uniform-Fisher}
Let $f_{\max}:=\max_{i\ne j}f_{ij}$.  For every positive smooth solution,
\begin{equation}\label{eq:general-Fisher-coercivity}
 \int_\Omega \nabla U:H(U)A_N(U)\nabla U\,dx
 \ge \frac1{f_{\max}}
 \int_\Omega\sum_{i=1}^N\frac{|\nabla y_i|^2}{y_i}\,dx
 =\frac4{f_{\max}}\sum_{i=1}^N
 \|\nabla\sqrt{y_i}\|_2^2.
\end{equation}
Consequently the ordinary mixing entropy controls the concentration Fisher
information without any compact-interior assumption.
\end{corollary}

\begin{proof}
Fix a spatial direction $k$ and write $J^k=Yu^k$ with $u^k\in E_y$.  By
$B(y)Y=-Y\mathcal L_y$,
\[
 \partial_ky=-Y\mathcal L_yu^k,
 \qquad
 \sum_i\frac{|\partial_ky_i|^2}{y_i}
 =\|\mathcal L_yu^k\|_y^2.
\]
Moreover, by \eqref{eq:D-B-identity} and the definition of $A_N$,
\[
 (\partial_kU)^{\sf T}H(U)A_N(U)\partial_kU
 =\cD_y(J^k,J^k)
 =\langle\mathcal L_yu^k,u^k\rangle_y.
\]
On $E_y$, the weighted variance identity and
\eqref{eq:rayleigh-general} give
\[
 0<\mathcal L_y\le f_{\max}I
 \quad\text{in }\langle\cdot,\cdot\rangle_y.
\]
Hence the spectral theorem yields
$\|\mathcal L_yu^k\|_y^2\le
 f_{\max}\langle\mathcal L_yu^k,u^k\rangle_y$.
Summing over $k$ proves \eqref{eq:general-Fisher-coercivity}.
\end{proof}

\section{Compact root ranges and analytical setup}
\label{sec:compact-range-setup}

The global analysis starts from the forward-invariant rectangles of
\Cref{thm:root-rectangle}.  A compact interlacing root range provides more
than positivity: all coordinate changes are uniformly regular, all root gaps
are bounded away from zero, and the entropy metric provides a uniformly
coercive quadratic form.  These facts will be used below in a scalar De Giorgi
argument for each individual root.

\subsection{Positive-time compact ranges}

Let $T_{\max}\in(0,\infty]$ and let
\[
        y:[0,T_{\max})\times\overline\Omega\to \cS^\circ
\]
be continuous in $C^1$ up to the initial time and, on every compact
positive-time strip, belong to the strong root class of
\Cref{thm:root-rectangle}.  For $\tau<T_{\max}$ define
\[
 a_j^\tau:=\min_{\overline\Omega}z_j(\tau,\cdot),\qquad
 b_j^\tau:=\max_{\overline\Omega}z_j(\tau,\cdot),
\]
and
\begin{equation}\label{eq:R-tau}
 R_\tau:=\prod_{j=1}^{N-1}[a_j^\tau,b_j^\tau]
 \Subset\prod_{j=1}^{N-1}(g_j,g_{j+1}).
\end{equation}

\begin{theorem}[Propagation of a compact root rectangle]\label{thm:compact-range-after-tau}
Suppose that $y(\tau,x)\in\cS^\circ$ for every $x\in\overline\Omega$.  Then
\begin{equation}\label{eq:R-tau-invariance}
 z(t,x)\in R_\tau,
 \qquad (t,x)\in[\tau,T_{\max})\times\overline\Omega.
\end{equation}
Consequently there is $c_\tau>0$ such that
\begin{equation}\label{eq:eventual-uniform-positivity}
 y_i(t,x)\ge c_\tau,
 \qquad i=1,\ldots,N,
 \qquad (t,x)\in[\tau,T_{\max})\times\overline\Omega.
\end{equation}
\end{theorem}

\begin{proof}
Apply \Cref{thm:root-rectangle} to the time-shifted solution
$\widetilde y(s,x)=y(\tau+s,x)$.  The inverse root map
\eqref{eq:y-inverse-roots} is continuous and strictly positive on the compact
rectangle $R_\tau$, so
\[
 c_\tau:=\min_{1\le i\le N}\min_{z\in R_\tau}y_i(z)>0.
\]
\end{proof}

\begin{corollary}[Uniformly positive initial data]\label{cor:positive-initial-compact-range}
If $y(0,\cdot)$ is continuous and takes values in a compact subset of
$\cS^\circ$, then its initial root rectangle $\mathcal R_0$ is forward invariant and
there is $c_0>0$ such that $y_i(t,x)\ge c_0$ for all $i$ and all
$0\le t<T_{\max}$.
\end{corollary}

\subsection{Uniform coordinate and entropy constants on a root box}

Fix for the remainder of the regularity analysis a compact root rectangle
\begin{equation}\label{eq:fixed-root-box-dossier}
 R=\prod_{j=1}^{m}[a_j,b_j]
 \Subset\prod_{j=1}^{m}(g_j,g_{j+1}),
 \qquad m:=N-1.
\end{equation}
Let $\cK_y,\cK_U,\cK_w$ be its images under the concentration, moment, and
entropy coordinate maps.  In this section, constants carrying a subscript $R$ may depend on this rectangle,
on $g=(g_1,\ldots,g_N)$, and on $N$, but not on the particular solution or on
time.  Later analytical estimates state any additional dependence on the space
dimension or local boundary geometry explicitly.

\begin{lemma}[Equivalence of first derivatives on compact subsets]\label{lem:derivative-equivalence}
There is $C_R\ge1$ such that every smooth state with $z\in R$ satisfies
\begin{equation}\label{eq:derivative-equivalence}
 C_R^{-1}|\nabla U|
 \le |\nabla y|+|\nabla z|+|\nabla w|
 \le C_R|\nabla U|.
\end{equation}
The same assertion holds with any one of $y,z,U,w$ used as reference variable.
\end{lemma}

\begin{proof}
The maps $z\mapsto y\mapsto U\mapsto w$ are smooth diffeomorphisms between
compact subsets of their state spaces.  Their Jacobians and inverse Jacobians
are uniformly bounded.
\end{proof}

\begin{proposition}[Uniform entropy coercivity on a compact root box]
\label{prop:uniform-coercivity}
There are $0<c_R\le C_R<\infty$ such that, for $U\in\cK_U$ and
{$v\in\mathbb R^m$},
\begin{equation}\label{eq:H-HA-uniform-dossier}
 c_R|{v}|^2\le {v}^{\sf T}H(U){v}\le C_R|{v}|^2,
 \qquad
 c_R|{v}|^2\le {v}^{\sf T}H(U)A_N(U){v}\le C_R|{v}|^2.
\end{equation}
\end{proposition}

\begin{proof}
Both matrix inequalities follow immediately from
\Cref{thm:entropy-symmetrization} and compactness of $\cK_U$.
\end{proof}

\begin{corollary}[{Mixing entropy and the concentration-affine root Hessian}]
\label{cor:mixing-entropy-multi-EPD}
On the whole interlacing root {box} $\cZ$, the pulled-back mixing entropy
\[
 \hz(z):=\sum_{\ell=1}^N y_\ell(z)\log y_\ell(z)
\]
satisfies, for all $j\ne k$,
\begin{equation}\label{eq:mixing-entropy-multi-EPD}
 (z_k-z_j)(\hz)_{jk}=(\hz)_j-(\hz)_k,
\end{equation}
and its diagonal second derivatives have the explicit positive form
\begin{equation}\label{eq:mixing-entropy-root-diagonal-positivity}
 (\hz)_{jj}(z)
 =\sum_{\ell=1}^N\frac{y_\ell(z)}{(g_\ell-z_j)^2}>0,
 \qquad j=1,\ldots,m.
\end{equation}
The same differential identities hold after adding an arbitrary constant to
$\hz$.
\end{corollary}

\begin{proof}
Let
\[
        \widetilde h(y):=\sum_{\ell=1}^Ny_\ell\log y_\ell,
        \qquad \hz(z)=\widetilde h(y(z)).
\]
On the tangent space $E$, $D_y^2\widetilde h=Y^{-1}$.  Hence, by the concentration-affine
root Hessian \eqref{eq:affine-root-Hessian-def}, the root frame
$e_j=-Yq_j$, and the weighted orthogonality
\eqref{eq:q-weighted-orthogonality},
\begin{equation}\label{eq:entropy-affine-Hessian-diagonal}
 \mathcal H^{\mathrm{aff}}_{jk}(\hz)
 =e_j^{\sf T}Y^{-1}e_k
 =q_j^{\sf T}Yq_k
 =\frac{\delta_{jk}}{\gamma_j}.
\end{equation}
For $j\ne k$, the component formula
\eqref{eq:affine-root-Hessian-components} therefore gives exactly
\eqref{eq:mixing-entropy-multi-EPD}.  For $j=k$ it gives
$(\hz)_{jj}=\gamma_j^{-1}$, which is
\eqref{eq:mixing-entropy-root-diagonal-positivity} by
\eqref{eq:gamma-def}.  Adding a constant does not change any of these
differential identities.
\end{proof}

The positivity assertion concerns these diagonal EPD entries and,
consequently, the diagonalized entropy production below; it does \emph{not} assert that the full
Euclidean Hessian $D_z^2\hz$ is positive definite in root coordinates.
Consequently, when the root production form $\mathfrak Q_\Psi$ is introduced
in \Cref{eq:Q-Psi-definition}, all of its mixed coefficients vanish for $\Psi=\hz$.
On each compact root rectangle, the remaining diagonal coefficients are uniformly positive and
bounded; the resulting coercivity estimate is recorded in \eqref{eq:root-entropy-qform-coercive}.

\section{Local strong solutions and continuation}
\label{sec:classical-local-theory}

Throughout this section, $m=N-1$, $p>d+2$, and
$\Omega\subset\mathbb R^d$ is a bounded domain of class $C^2$.  No stronger
boundary regularity is used in the local existence or continuation
alternative; the $C^{2+\alpha}$ assumption needed for positive-time
conormal smoothing is introduced separately in \Cref{sec:DG-to-strong}.

We use time-first anisotropic H\"older notation:
$C_{t,x}^{\theta/2,\theta}$ denotes parabolic H\"older regularity with
exponent $\theta/2$ in time and $\theta$ in space.  Thus, for example,
$C_{t,x}^{(1+\theta)/2,1+\theta}$ is the first-order parabolic H\"older class.
By \Cref{cor:formulation-equivalence}, solving the entropy-variable equation
below and transforming back reconstructs the original constrained
Maxwell--Stefan fluxes and balances.

In entropy variables the system is
\begin{equation}\label{eq:entropy-system-recalled}
 K(w)\partial_tw-\Div(B_N(w)\nabla w)=0,
 \qquad \partial_\nu w=0,
\end{equation}
where $w$ takes values in $\mathcal W=\mathbb R^m$ and $K(w),B_N(w)$ are
symmetric positive definite.  Since $B_N(w)$ acts only in component space, the
no-flux condition $B_N(w)\nabla w\,\nu=0$ is equivalent to
$\partial_\nu w=0$.

Set
\begin{equation}\label{eq:maximal-spaces}
 X_0:=L^p(\Omega;\mathbb R^m),\qquad
 X_1:=\{v\in W^{2,p}(\Omega;\mathbb R^m):\partial_\nu v=0\},
 \qquad X_{\gamma,p}:=(X_0,X_1)_{1-1/p,p}.
\end{equation}
The real-interpolation theorem for elliptic realizations with boundary
conditions, specifically Theorem~5.2 together with the boundary-space
definition~(5.23) and the $C^2$-boundary refinement in Remark~5.3(c) of
\cite{Amann1993Boundary}, gives, provided
$2\theta\neq 1+1/p$,
\[
 (L^p(\Omega),W^{2,p}_N(\Omega))_{\theta,q}
 =B^{2\theta}_{p,q,N}(\Omega),
\]
where the subscript $N$ means that the homogeneous normal derivative is imposed
exactly when $2\theta>1+1/p$.  At $\theta=1-1/p$ and $q=p$ this yields
\begin{equation}\label{eq:trace-space-identification}
 X_{\gamma,p}
 =\{v\in B^{2-2/p}_{p,p}(\Omega;\mathbb R^m):\partial_\nu v=0\}
\end{equation}
with equivalent norms.  Since $d\ge2$ and $p>d+2$, one has $p>4$, so the
exceptional equality $2-2/p=1+1/p$ is excluded, and
$2-2/p>1+1/p$.  {Thus the normal derivative trace is part of the
initial trace space rather than an additional compatibility condition.}  Because $p>d+2$,
\begin{equation}\label{eq:trace-holder-embedding}
 X_{\gamma,p}\hookrightarrow C^{1+{\eta}}(\overline\Omega;\mathbb R^m)
 \qquad\left(0<{\eta}<1-\frac{d+2}{p}\right).
\end{equation}

\begin{lemma}[Trace compatibility and quantitative restart control]
\label{lem:trace-restart-control}
Under the standing assumptions $p>d+2$ and $\Omega\in C^2$, the trace space
$X_{\gamma,p}$ already contains the homogeneous Neumann compatibility:
\[
 X_{\gamma,p}
 =\{v\in B^{2-2/p}_{p,p}(\Omega;\mathbb R^m):\partial_\nu v=0\}.
\]
Moreover,
\begin{equation}\label{eq:X1-to-Xgamma-restart}
 X_1\hookrightarrow X_{\gamma,p}\hookrightarrow C^{1+{\eta}}(\overline\Omega),
 \qquad 0<{\eta}<1-\frac{d+2}{p},
\end{equation}
and, for every {$0<\beta<1$},
\begin{equation}\label{eq:C2beta-to-X1-restart}
 \partial_\nu v=0,\quad v\in C^{2+\beta}(\overline\Omega;\mathbb R^m)
 \quad\Longrightarrow\quad
 \|v\|_{X_1}\le C_{\Omega,p,\beta}\|v\|_{C^{2+\beta}}.
\end{equation}
Consequently, a family whose pointwise ranges lie in one compact subset of
$\mathcal W$ and which is uniformly bounded in $C^{2+\beta}$ with homogeneous
Neumann data satisfies the uniform state-range and $X_{\gamma,p}$ hypotheses
needed for a family-uniform local restart.
\end{lemma}

\begin{proof}
The first identity is \eqref{eq:trace-space-identification}; the inequality
$2-2/p>1+1/p$ ensures that the normal derivative trace is part of the
interpolation space.  The first embedding in
\eqref{eq:X1-to-Xgamma-restart} is the canonical interpolation embedding and
the second is \eqref{eq:trace-holder-embedding}.  Since $\Omega$ is bounded,
$C^{2+\beta}(\overline\Omega)\hookrightarrow W^{2,p}(\Omega)$ continuously;
with $\partial_\nu v=0$ this is exactly
\eqref{eq:C2beta-to-X1-restart}.  The final assertion follows by combining
these bounds with the compact pointwise state-range hypothesis.
\end{proof}

Expanding the divergence gives
\begin{equation}\label{eq:entropy-system-nondivergence}
 \partial_tw-C(w)\Delta w=F(w,\nabla w),
 \qquad C(w):=K(w)^{-1}B_N(w),
\end{equation}
where
\begin{equation}\label{eq:entropy-quadratic-F}
 F(w,\nabla w)
 :=K(w)^{-1}\sum_{k=1}^d
 DB_N(w)[\partial_k w] \,\partial_k w.
\end{equation}
The matrix $C(w)=K(w)^{-1}A_N(U(w))K(w)$ is similar to $A_N(U(w))$ and has
the positive eigenvalues $\lambda_j(U(w))^{-1}$.

\begin{lemma}[Exact strong-Sobolev equivalence of the entropy formulations]
\label{lem:entropy-div-nondiv-equivalence}
Let $I\Subset(0,T)$ and suppose
\[
 w\in W^{1,p}(I;X_0)\cap L^p(I;X_1)\cap C(\overline I;X_{\gamma,p})
\]
has range in a compact subset of $\mathcal W$.  Then
$B_N(w)\nabla w\in L^p(I;W^{1,p}(\Omega))$ componentwise and, almost
everywhere on $I\times\Omega$,
\begin{equation}\label{eq:entropy-divergence-expansion-strong}
 \Div(B_N(w)\nabla w)
 =B_N(w)\Delta w
  +\sum_{k=1}^d
   DB_N(w)[\partial_k w]\,\partial_k w.
\end{equation}
Consequently the divergence-form entropy system
\eqref{eq:entropy-system-recalled} and the nondivergence equation
\eqref{eq:entropy-system-nondivergence}--\eqref{eq:entropy-quadratic-F}
are equivalent in the maximal-$L^p$ strong class, with the same homogeneous
Neumann boundary condition.
\end{lemma}

\begin{proof}
Because $p>d+2$, the trace embedding
\eqref{eq:trace-holder-embedding} and
$w\in C(\overline I;X_{\gamma,p})$ give
$w,\nabla w\in L^\infty(I\times\Omega)$.  On the compact state range all
first derivatives of $B_N$ are bounded.  For every pair of spatial indices
$k,\ell$, the Sobolev product and chain rules therefore give
\[
 \partial_\ell\bigl(B_N(w)\partial_k w\bigr)
 =DB_N(w)[\partial_\ell w]\,\partial_k w
  +B_N(w)\partial_{\ell k}w\in L^p(I\times\Omega).
\]
Thus the flux has the asserted $W^{1,p}$ regularity, and summing the preceding
identity with $\ell=k$ proves
\eqref{eq:entropy-divergence-expansion-strong}.  Since $K(w)$ is uniformly
invertible on the same compact range, multiplication by $K(w)^{-1}$ gives
exactly
\[
 \partial_tw-K(w)^{-1}B_N(w)\Delta w
 =K(w)^{-1}\sum_{k=1}^d
 DB_N(w)[\partial_k w]\,\partial_k w,
\]
with no derivative of $K$: the matrix $K(w)$ multiplies the time derivative
and is not contained in the spatial divergence.  Conversely, multiplying
this identity by $K(w)$ and using
\eqref{eq:entropy-divergence-expansion-strong} reconstructs the divergence
form.  Finally $B_N(w)$ is invertible pointwise, so
$B_N(w)\nabla w\,\nu=0$ is equivalent to $\partial_\nu w=0$.
\end{proof}

\begin{definition}[Strong and classical solutions]\label{def:solution-classes}
Let $0<T\le\infty$.  A \emph{maximal-$L^p$ strong solution} on $[0,T)$ is a function
\[
 w\in W^{1,p}_{\rm loc}([0,T);X_0)\cap L^p_{\rm loc}([0,T);X_1)
       \cap C([0,T);X_{\gamma,p})
\]
which satisfies \eqref{eq:entropy-system-nondivergence} in $X_0$ for almost every
time and has trace $w(0)=w_0$.  A \emph{positive-time classical solution} is a
strong solution which belongs to $C_{t,x}^{1+\sigma/2,2+\sigma}([\tau,T']\times
\overline\Omega)$ for every $0<\tau<T'<T$ and some $\sigma=\sigma(\tau,T')>0$.
When a solution is called classical on a parabolic cylinder, this latter
regularity is meant on a neighbourhood of the closed cylinder, except that no
regularity across its lower time face is required.
\end{definition}

\begin{lemma}[Uniform parameter ellipticity and complementing condition]
\label{lem:uniform-complementing}
Let $\mathcal K_w\Subset\mathcal W$.  There are $0<c<C<\infty$ and an angle
$\vartheta<\pi/2$, depending only on $\mathcal K_w$, such that for every
$w_*\in\mathcal K_w$ the spectrum of $C(w_*)$ lies in $[c,C]$ and the frozen
Neumann problem
\[
 \lambda v-C(w_*)\Delta v=f,\qquad \partial_\nu v=0,
\]
is parameter elliptic in $\Sigma_{\pi-\vartheta}$ and satisfies the full
parameter-dependent Lopatinskii--Shapiro condition, uniformly in $w_*$.  After
complexification, the half-space boundary solution map extends holomorphically
to one uniform complex neighbourhood of the normalized real parameter set.
\end{lemma}

\begin{proof}
Put $P_*=K(w_*)^{1/2}$.  The canonical similarity transform
\begin{equation}\label{eq:canonical-symmetric-similarity}
 P_*C(w_*)P_*^{-1}=K(w_*)^{-1/2}B_N(w_*)K(w_*)^{-1/2}=:S_*
\end{equation}
is symmetric positive definite.  Compactness and smooth dependence imply
uniform bounds for $P_*$ and $P_*^{-1}$ and $cI\le S_*\le CI$.  This proves uniform
normal ellipticity without choosing eigenvectors through possible eigenvalue
crossings.

For the half-space model, an orthogonal diagonalization of the single frozen
matrix $S_*$ reduces the transformed equation to scalar equations
\[
 \lambda {\psi_j}-\mu_j(\partial_d^2-|\xi'|^2){\psi_j}=0,
 \qquad c\le\mu_j\le C.
\]
The decaying mode equals {$c_j e^{-\rho_jx_d}$}, where
$\rho_j=(|\xi'|^2+\lambda/\mu_j)^{1/2}$ is chosen with positive real part.
For $(\lambda,\xi')\ne(0,0)$ one has $\rho_j\ne0$.  The full
Lopatinskii--Shapiro condition asks for existence and uniqueness for every
prescribed boundary datum, not only homogeneous uniqueness.  If
{$\partial_d\psi_j(0)=h_j$}, then explicitly
\begin{equation}\label{eq:LS-Neumann-explicit-amplitude}
 {c_j=-\frac{h_j}{\rho_j}}.
\end{equation}
Thus existence and uniqueness are immediate.  Moreover, sector geometry and
$\mu_j\in[c,C]$ give
\[
 |\rho_j|\asymp |\xi'|+|\lambda|^{1/2},
 \qquad
 \operatorname{Re}\rho_j\ge c_\vartheta
 (|\xi'|+|\lambda|^{1/2}),
\]
uniformly after normalization.  Here and below, $A\asymp B$ means two-sided
comparability up to positive multiplicative constants; a subscript, as in
$\asymp_R$, records the parameters on which those constants may depend.  Hence the stable subspace and Neumann
boundary map are uniformly transverse.

We also close directly the complex-neighbourhood point needed for holomorphic
boundary solution operators.  Augner notes \cite[Remark~4.4]{Augner2024} that this technical extension was not addressed in the original
Denk--Hieber--Pr\"uss argument.  For a complexified tangential frequency
$\zeta\in\mathbb C^{d-1}$ put
\[
 M_*(\zeta,\lambda)
 :=(\zeta\cdot\zeta)I+\lambda S_*^{-1}.
\]
On the real normalized set
$|\xi'|+|\lambda|^{1/2}=1$ its spectrum is
$\{|\xi'|^2+\lambda/\mu_j\}_j$.  Uniformly for
$w_*\in\mathcal K_w$, choose once and for all
$\vartheta<\phi<\pi/2$.  This compact spectral family has positive distance
from the branch cut $(-\infty,0]$ when
$\lambda\in\overline\Sigma_{\pi-\phi}\Subset\Sigma_{\pi-\vartheta}$.
By continuity, the same remains true
for $\zeta$ in one uniform complex neighbourhood of the real normalized
parameter set.  The principal matrix square root
\[
 {\mathcal R_*(\zeta,\lambda)}:=M_*(\zeta,\lambda)^{1/2}
\]
is therefore holomorphic there, uniformly invertible, and has spectrum in
$\{\operatorname{Re}z>0\}$.  The stable solution is
\[
 {\psi(x_d)=e^{-x_d\mathcal R_*}c},
 \qquad
 {\psi'(0)=-\mathcal R_*c=h},
 \qquad
 {c=-\mathcal R_*^{-1}h}.
\]
This proves both the full parameter-dependent Lopatinskii--Shapiro condition
and the required holomorphic complex-neighbourhood extension, without any
choice of eigenvectors through possible crossings.  The constants remain
uniform because $S_*$ and $P_*$ and $P_*^{-1}$ range over compact matrix families.
\end{proof}

\begin{lemma}[Uniform linear maximal regularity]\label{lem:uniform-linear-MR}
Fix a compact $\mathcal K_w\Subset\mathcal W$.  There is an open neighbourhood
$\mathcal O_w\Subset\mathcal W$ of $\mathcal K_w$ with the following property.
For every $v\in X_{\gamma,p}$ satisfying $v(\overline\Omega)\subset\mathcal O_w$,
the realization
\[
 \mathcal A(v)u=-C(v(\cdot))\Delta u,\qquad D(\mathcal A(v))=X_1,
\]
has maximal $L^p$ regularity on $X_0$.  On sets on which
$\|v\|_{X_{\gamma,p}}$ is bounded and $v(\overline\Omega)\subset\mathcal O_w$,
the maximal-regularity estimates are uniform on every fixed finite time
horizon $[0,T_0]$, and hence uniformly for all $0<T\le T_0$.
\end{lemma}

\begin{proof}
Choose $\mathcal O_w$ with compact closure in $\mathcal W$.  For each fixed
$v\in X_{\gamma,p}$ with range in $\mathcal O_w$, the embedding
\eqref{eq:trace-holder-embedding} gives
$C(v(\cdot))\in C^{1+{\eta}}(\overline\Omega)$.  The principal symbol is
$|\xi|^2C(v(x))$.  The pointwise argument of
\Cref{lem:uniform-complementing}, with the canonical square root
$K(v(x))^{1/2}$, verifies parameter ellipticity and the
Lopatinskii--Shapiro condition.  Since the Neumann boundary operator is fixed,
after complexifying the finite-dimensional component space Augner's bounded-domain
maximal-regularity theorem \cite[Theorem~5.2]{Augner2024} applies with
$E=\mathbb C^m$, differential-order parameter $m_{\rm A}=1$, one boundary
operator $B_1=\partial_\nu I_E$, and temporal exponent chosen equal to the
spatial exponent $p$.  In Augner's mixed-order notation the only nonzero
boundary projection is $P_{1,1}=I_E$.  The real subspace is invariant and is
recovered by restriction.  This theorem follows the Denk--Hieber--Pr\"uss
strategy while making the boundary-symbol holomorphy issue explicit.  The
original second-order theorem \cite[Theorem~8.2]{DenkHieberPruess2003} gives
the same realization once that point is supplied by
\Cref{lem:uniform-complementing}.  Let us check the second-order hypotheses
explicitly.  The domain has $C^2$ boundary; the
top-order coefficients $C(v(x))\delta_{k\ell}$ are bounded and
uniformly continuous (indeed $C^{1+{\eta}}$); there are no lower-order
coefficients; and the order-one Neumann boundary operator has coefficients
$\nu_k(x)I_m\in C^1(\partial\Omega;\mathcal L(E))$, as required for a
$C^2$ boundary.  Parameter ellipticity and the Lopatinskii--Shapiro condition
are exactly those verified above.  Thus the realization in that theorem is
the common domain $X_1$ in \eqref{eq:maximal-spaces}.  With the choice
$\vartheta<\phi<\pi/2$ above, the required estimates hold uniformly on the
closed parameter sector $\overline\Sigma_{\pi-\phi}$, while the corresponding
ellipticity angle is strictly smaller than $\pi/2$.  Hence
\cite[Theorem~5.2]{Augner2024} yields maximal $L^p$ regularity for a scalar
shift $\mu+\mathcal A(v)$.  On a finite interval this
is exactly equivalent to maximal regularity of the unshifted realization, and
we record the elementary transfer because it is the only shift argument used
here.  If
\[
 u_t+\mathcal A(v)u=f,\qquad u(0)=0,
\]
set $q(t)=e^{-\mu t}u(t)$.  Then
\[
 q_t+(\mu+\mathcal A(v))q=e^{-\mu t}f,\qquad q(0)=0.
\]
Conversely $u=e^{\mu t}q$.  Multiplication by $e^{\pm\mu t}$ is an
isomorphism of both $L^p(0,T;X_0)$ and
$W^{1,p}(0,T;X_0)\cap L^p(0,T;X_1)$, with constants depending only on
$|\mu|T$.  Hence shifted maximal regularity implies unshifted maximal
regularity on every finite interval.  No additional operator-valued
multiplier or perturbation theorem is invoked at this stage.

It remains to justify the asserted uniformity.  Fix $M<\infty$ and choose
{$0<\eta'<\eta$}.  The set
\[
 \mathfrak B_M:=\{v\in X_{\gamma,p}:\|v\|_{X_{\gamma,p}}\le M,
                  \ v(\overline\Omega)\subset\overline{\mathcal O_w}\}
\]
is relatively compact in $C^{1+{\eta'}}(\overline\Omega)$ and hence the
coefficient family $\{C(v):v\in\mathfrak B_M\}$ is relatively compact in
$C^{1+{\eta'}}(\overline\Omega)$.  Every coefficient field in the closure
still takes values in the same compact normally elliptic matrix class and
satisfies the same complementing condition.  The preceding Augner/Denk--Hieber--Pr\"uss argument therefore applies to every member of this
closure.  Moreover,
\begin{equation}\label{eq:operator-coefficient-continuity}
 \|\mathcal A(v)-\mathcal A(\widetilde v)\|_{\mathcal L(X_1,X_0)}
 \le C_M\|C(v)-C(\widetilde v)\|_{L^\infty}.
\end{equation}
The family-uniform bound on a fixed horizon can be obtained directly, without
an additional perturbation theorem.  Fix $T_0>0$ and write
\[
 \mathbb E_{T_0}^0
 :=\{u\in W^{1,p}(0,T_0;X_0)\cap L^p(0,T_0;X_1):u(0)=0\},
 \qquad
 \mathbb F_{T_0}:=L^p(0,T_0;X_0).
\]
For a frozen realization $A$ define
\[
 \mathcal L_A:\mathbb E_{T_0}^0\to\mathbb F_{T_0},
 \qquad \mathcal L_Au:=\partial_tu+Au.
\]
Maximal regularity is exactly the statement that $\mathcal L_A$ is an
isomorphism.  If $A_0$ has maximal regularity, then for another realization
$A$ on the same domain,
\[
 \mathcal L_A\mathcal L_{A_0}^{-1}
 =I+\mathcal M_{A-A_0}\mathcal L_{A_0}^{-1},
\]
where
$\|\mathcal M_{A-A_0}\|_{\mathcal L(\mathbb E_{T_0}^0,\mathbb F_{T_0})}
\le\|A-A_0\|_{\mathcal L(X_1,X_0)}$.  Hence
\begin{equation}\label{eq:MR-Neumann-neighbourhood}
 \|A-A_0\|_{\mathcal L(X_1,X_0)}
 <\frac{1}{2\|\mathcal L_{A_0}^{-1}\|}
 \quad\Longrightarrow\quad
 \|\mathcal L_A^{-1}\|
 \le2\|\mathcal L_{A_0}^{-1}\|,
\end{equation}
by the Neumann series.  The compact coefficient closure above has compact
image in $\mathcal L(X_1,X_0)$ by
\eqref{eq:operator-coefficient-continuity}; covering it by finitely many
neighbourhoods of the form \eqref{eq:MR-Neumann-neighbourhood} gives one
maximal-regularity constant for the entire frozen family on $(0,T_0)$.

The same constant works for every shorter interval $(0,T)$, $T\le T_0$: extend
a forcing $f\in L^p(0,T;X_0)$ by zero to $(0,T_0)$, solve the zero-trace
problem there, and restrict the solution to $(0,T)$.  This proves exactly the
family-uniform zero-trace estimate needed for the fixed-point and restart
arguments.  The separate uniform extension operator for maximal-regularity
functions, used below only for the short-time H\"older embedding, is supplied
by \cite[Lemma~7.2]{Amann2005MaxReg}.
\end{proof}

\begin{theorem}[Local strong solutions and continuation alternative]
\label{thm:classical-local-theory}
Let $p>d+2$ and $w_0\in X_{\gamma,p}$.  Then
\eqref{eq:entropy-system-recalled} has a unique maximal-$L^p$ strong solution
\begin{equation}\label{eq:maximal-Lp-class}
 w\in W^{1,p}_{\rm loc}([0,T_{\max});X_0)
       \cap L^p_{\rm loc}([0,T_{\max});X_1)
       \cap C([0,T_{\max});X_{\gamma,p}).
\end{equation}
For every compact $\mathcal K_w\Subset\mathcal W$ and every $M<\infty$ there
is $\tau_*=\tau_*(\mathcal K_w,M)>0$ such that every datum $v_0$ with
\[
 v_0(\overline\Omega)\subset\mathcal K_w,
 \qquad \|v_0\|_{X_{\gamma,p}}\le M,
\]
generates a solution on $[0,\tau_*]$.  Consequently, if $T_{\max}<\infty$ and
\[
 w([\tau,T_{\max})\times\overline\Omega)\subset\mathcal K_w
\]
for some $\tau<T_{\max}$ and compact $\mathcal K_w$, then
\begin{equation}\label{eq:norm-blow-up}
 \lim_{t\uparrow T_{\max}}\|w(t)\|_{X_{\gamma,p}}=\infty.
\end{equation}
\end{theorem}

\begin{proof}
For $v\in X_{\gamma,p}$ set
\[
 \mathcal A(v)u=-C(v)\Delta u,
 \qquad
 \mathcal F(v)=K(v)^{-1}\sum_{k=1}^d
 DB_N(v)[\partial_k v] \,\partial_k v.
\]
On every bounded $X_{\gamma,p}$ set with values in a fixed relatively compact
open subset of $\mathcal W$, smoothness of $K,B_N$ and
$X_{\gamma,p}\hookrightarrow C^{1+{\eta}}$ imply
\begin{align}
 \|\mathcal A(v)-\mathcal A(\widetilde v)\|_{\mathcal L(X_1,X_0)}
 &\le L\|v-\widetilde v\|_{X_{\gamma,p}},
 \label{eq:A-local-Lipschitz}\\
 \|\mathcal F(v)-\mathcal F(\widetilde v)\|_{X_0}
 &\le L\|v-\widetilde v\|_{X_{\gamma,p}}.
 \label{eq:F-local-Lipschitz}
\end{align}
The quadratic term is estimated using the bounded embedding into
$W^{1,\infty}$.

We now prove local existence directly by freezing at the initial state; no
nonautonomous maximal-regularity theorem is needed.  {The same contraction also yields the family-uniform restart time used in the
global argument.}  For a finite interval $J=(0,T)$ put
\[
 \mathbb E_T:=W^{1,p}(J;X_0)\cap L^p(J;X_1),
 \qquad
 \mathbb E_T^0:=\{h\in\mathbb E_T:h(0)=0\}.
\]
The trace theorem gives
$\mathbb E_T\hookrightarrow C(\overline J;X_{\gamma,p})$.
For a single datum $w_0$, choose a compact set
$\mathcal K_w\Subset\mathcal W$ containing its pointwise range in its
interior and choose $M>\|w_0\|_{X_{\gamma,p}}$.  {The argument establishes a common local existence statement for all data whose ranges lie in
$\mathcal K_w$ and whose $X_{\gamma,p}$ norms are at most $M$.}  Applied to
this one-element family, it gives the asserted local solution.

The fixed-point equation is obtained by setting
$\mathcal A_0:=\mathcal A(v_0)$ and writing the unknown as
$u=u_*^{v_0}+h$, where $u_*^{v_0}$ solves the homogeneous frozen equation
with initial value $v_0$.  Then $h\in\mathbb E_T^0$ must satisfy
\[
 \partial_t h+\mathcal A_0h
 = [\mathcal A_0-\mathcal A(u_*^{v_0}+h)](u_*^{v_0}+h)
   +\mathcal F(u_*^{v_0}+h).
\]
Since $\mathcal A_0$ has maximal $L^p$ regularity by
\Cref{lem:uniform-linear-MR}, the zero-trace solution operator for the left
side is bounded $L^p(J;X_0)\to\mathbb E_T^0$.  The estimates below show that
its composition with the right side is a strict contraction for sufficiently
small $T$.  Thus the local quasilinear problem is obtained from frozen
maximal regularity alone.

We now carry out that contraction quantitatively and simultaneously prove the
family-uniform lower existence time.  Choose
$\mathcal K_w\Subset\mathcal O_w\Subset\mathcal W$ and put
\[
 d_0:=\operatorname{dist}(\mathcal K_w,\partial\mathcal O_w)>0.
\]
Fix $T_0>0$.  For an admissible initial value $v_0$, let $u_*^{v_0}$ be the
solution of the frozen homogeneous problem
\begin{equation}\label{eq:uniform-reference-problem}
 \partial_tu_*^{v_0}+\mathcal A(v_0)u_*^{v_0}=0,
 \qquad u_*^{v_0}(0)=v_0,
\end{equation}
on $(0,T_0)$.  The uniform maximal-regularity estimate from
\Cref{lem:uniform-linear-MR}, together with a fixed bounded coretraction of
the trace map $\mathbb E_{T_0}\to X_{\gamma,p}$---the trace is a
retraction and an explicit semigroup coretraction is given in
\cite[(2.5)--(2.6)]{Amann2004NonautonomousMR}---gives
\begin{equation}\label{eq:uniform-reference-MR-bound}
 \sup_{v_0}\|u_*^{v_0}\|_{\mathbb E_{T_0}}\le {M_{\rm ref}}={M_{\rm ref}}(\mathcal K_w,M,T_0).
\end{equation}
Indeed, after subtracting the coretraction of $v_0$, one applies the uniform
zero-trace maximal-regularity estimate; the uniform bound of
$\mathcal A(v_0)$ in $\mathcal L(X_1,X_0)$ controls the resulting forcing.
Here and below the supremum is over
$v_0(\overline\Omega)\subset\mathcal K_w$ with
$\|v_0\|_{X_{\gamma,p}}\le M$.  Enlarging {$M_{\rm ref}$} if necessary, we assume
{$M_{\rm ref}\ge M$}.

Choose $\xi$ and a time H\"older exponent $\theta_0$ such that
\[
 \frac{1+d/p}{2}<\xi<1-\frac1p,
 \qquad
 0<\theta_0<1-\frac1p-\xi.
\]
The standard maximal-regularity interpolation embedding gives, on a fixed
time interval,
\[
 W^{1,p}(J;X_0)\cap L^p(J;X_1)
 \hookrightarrow
 C^{\theta_0}\bigl(\overline J;(X_0,X_1)_{\xi,1}\bigr)
 \hookrightarrow C^{\theta_0}(\overline J;C^1(\overline\Omega)),
\]
where the second embedding follows by exactness of real interpolation and
\[
 (X_0,X_1)_{\xi,1}\hookrightarrow
 (L^p(\Omega),W^{2,p}(\Omega))_{\xi,1}
 =B^{2\xi}_{p,1}(\Omega)
 \hookrightarrow C^1(\overline\Omega)
\]
because $2\xi>1+d/p$.  The strict inequality in the choice of $\theta_0$
is intentional: the standard time-H\"older embedding is asserted below the
endpoint $1-1/p-\xi$; no endpoint statement is needed here.  This is precisely the sub-endpoint case of the maximal-regularity
embedding \cite[(5.8)]{Amann2004NonautonomousMR}.

The constant can be chosen uniformly for $0<T\le T_0$.  Indeed, \cite[Lemma~7.2]{Amann2005MaxReg} constructs extension operators
whose norms are uniform in the short time $T$ and gives exactly the estimate
\[
 \|\operatorname{ext}_T u\|_{\mathbb E_{T_0}}
 \le C\bigl(\|u\|_{\mathbb E_T}
       +\|u(0)\|_{X_{\gamma,p}}\bigr),
\]
with $C$ independent of $0<T\le T_0$.
Applying the fixed-interval embedding to this extension yields
\begin{equation}\label{eq:uniform-short-time-C1-control}
 \|u(t)-u(s)\|_{C^1(\overline\Omega)}
 \le C|t-s|^{\theta_0}
 \bigl(\|u\|_{\mathbb E_T}+\|u(0)\|_{X_{\gamma,p}}\bigr),
 \qquad 0\le s,t\le T.
\end{equation}
For functions with zero initial trace, the initial-trace term in the uniform
extension estimate is absent.
Consequently,
\begin{equation}\label{eq:uniform-reference-state-control}
 \sup_{v_0}\sup_{0\le t\le T}
 \|u_*^{v_0}(t)-v_0\|_{C^1}
 \le C{M_{\rm ref}}T^{\theta_0}.
\end{equation}
This is the uniform short-time state control that is not supplied by a
pointwise local-existence statement alone.

For completeness, we record the uniform contraction.  Write
$u=u_*^{v_0}+h$, where $h(0)=0$, and let
$\mathbb E_T^0:=\{h\in\mathbb E_T:h(0)=0\}$.  Solving with the frozen
operator $\partial_t+\mathcal A(v_0)$ defines the fixed-point map
\begin{equation}\label{eq:uniform-fixed-point-map}
 \mathfrak T_{v_0,T}(h)
 :=(\partial_t+\mathcal A(v_0))^{-1}_0
 \left(
 [\mathcal A(v_0)-\mathcal A(u_*^{v_0}+h)](u_*^{v_0}+h)
 +\mathcal F(u_*^{v_0}+h)
 \right),
\end{equation}
where the subscript denotes zero initial trace.  Its inverse norm is uniform
by \Cref{lem:uniform-linear-MR}.  If $\|h\|_{\mathbb E_T}\le r$, then
\eqref{eq:uniform-short-time-C1-control}--\eqref{eq:uniform-reference-state-control} give
\begin{equation}\label{eq:uniform-fixed-point-state-distance}
 \sup_{0\le t\le T}
 \|u_*^{v_0}(t)+h(t)-v_0\|_{C^1}
 \le C({M_{\rm ref}}+r)T^{\theta_0}.
\end{equation}
Choose first a fixed $r>0$ and then $T$ so small that the right-hand side is
less than $d_0/2$.  Thus every function in the fixed-point ball takes values
in $\mathcal O_w$, uniformly in $v_0$.  Smoothness of the pointwise
coefficient gives
\[
 \|\mathcal A(v)-\mathcal A(\widetilde v)\|_{\mathcal L(X_1,X_0)}
 \le L\|v-\widetilde v\|_{C(\overline\Omega)}
\]
on this state neighbourhood.  Combining this estimate with
\eqref{eq:F-local-Lipschitz}, the uniform trace estimates
\[
 \|u\|_{C([0,T];X_{\gamma,p})}
 \le C\bigl(\|u\|_{\mathbb E_T}+\|u(0)\|_{X_{\gamma,p}}\bigr),
 \qquad
 \|h\|_{C([0,T];X_{\gamma,p})}\le C\|h\|_{\mathbb E_T}
 \quad(h\in\mathbb E_T^0),
\]
and \eqref{eq:uniform-fixed-point-state-distance} yields constants independent
of $v_0$ such that
\begin{align}
 \|\mathfrak T_{v_0,T}(h)\|_{\mathbb E_T}
 &\le C\left[T^{\theta_0}({M_{\rm ref}}+r)^2
       +T^{1/p}({M_{\rm ref}}+r)^2\right],
 \label{eq:uniform-fixed-point-self-map}\\
 \|\mathfrak T_{v_0,T}(h_1)-\mathfrak T_{v_0,T}(h_2)\|_{\mathbb E_T}
 &\le C\left[T^{\theta_0}({M_{\rm ref}}+r)+T^{1/p}({M_{\rm ref}}+r)\right]
 \|h_1-h_2\|_{\mathbb E_T}.
 \label{eq:uniform-fixed-point-contraction}
\end{align}
The factor $T^{\theta_0}$ comes from the coefficient difference in
$C([0,T];\mathcal L(X_1,X_0))$, while $T^{1/p}$ comes from integrating the
quadratic lower-order term in time; the displayed factors {$M_{\rm ref}+r$} record its
size on the fixed-point ball.  Taking $T$ smaller once more
makes \eqref{eq:uniform-fixed-point-map} a strict contraction of the closed
$r$-ball in $\mathbb E_T^0$ into itself.  Banach's fixed-point theorem gives a
strong solution of the nondivergence equation on $[0,T]$.  The fixed-point
state control keeps its range in $\mathcal O_w$, so
\Cref{lem:entropy-div-nondiv-equivalence} reconstructs the divergence-form
entropy system \eqref{eq:entropy-system-recalled} exactly, including the
homogeneous Neumann condition.  The same estimates applied to the difference
of two solutions with the same initial trace give uniqueness after possibly
shortening the interval; iteration in time yields uniqueness on every common
existence interval.  We have therefore proved the asserted common existence
time $\tau_*(\mathcal K_w,M)>0$ directly.

Starting from any single datum and repeatedly restarting at positive times
constructs a unique maximal interval $[0,T_{\max})$ and a solution in
\eqref{eq:maximal-Lp-class}.  If $T_{\max}<\infty$ while the late-time range
stays in a compact $\mathcal K_w\Subset\mathcal W$, then the norm satisfies
\eqref{eq:norm-blow-up}.  Indeed, if that limit failed, there would
exist $M<\infty$ and a sequence $t_n\uparrow T_{\max}$ such that
$\|w(t_n)\|_{X_{\gamma,p}}\le M$.  The just-proved family-uniform lifespan
then supplies the same number $\tau_*(\mathcal K_w,M)>0$ for the problem
restarted from every datum $w(t_n)$.  For $n$ sufficiently large,
$T_{\max}-t_n<\tau_*/2$; uniqueness identifies the restarted solution with
the original one on their common interval and consequently extends it past
$T_{\max}$, contradicting maximality.  This proves the stated continuation
alternative with limit $+\infty$, rather than only an unbounded
limsup.

Positive-time classicality is not used in the local contraction or in the
continuation alternative.  It will be obtained later, after the
weak-solution conormal estimates have been established in
\Cref{sec:DG-to-strong}.  Thus the local maximal interval and the
family-uniform restart statement rest only on frozen maximal regularity and
the contraction argument above.
\end{proof}

\begin{remark}[Compact range for entropy-variable initial data]
\label{rem:entropy-data-compact-range}
For the data in \Cref{thm:classical-local-theory}, the embedding
\eqref{eq:trace-holder-embedding} gives a bounded range for $w_0$.  The
softmax formula \eqref{eq:y-softmax} therefore places
$y_0(\overline\Omega)$ in a compact subset of $\mathcal S^\circ$, so the
initial root range lies in a compact rectangle $\mathcal R_0$.

For the corresponding maximal-$L^p$ solution, on every finite time interval
$[0,T']\Subset[0,T_{\max})$ the continuity
$w\in C([0,T'];X_{\gamma,p})$ and
$X_{\gamma,p}\hookrightarrow C^1(\overline\Omega)$ give, after smooth
coordinate composition, $z\in C([0,T'];C^1(\overline\Omega))$.  On every
compact positive-time strip, \Cref{cor:formulation-equivalence} turns the
entropy solution into a strong moment solution and
\Cref{cor:strong-sobolev-root-system} supplies the distributional root system
and homogeneous Neumann trace.  Thus all hypotheses of
\Cref{thm:root-rectangle,thm:compact-range-after-tau} are verified for the
maximal-$L^p$ solution, and $\mathcal R_0$ is forward invariant from time zero.
\end{remark}

\section{Entropy-stabilized one-sided multi-EPD truncations and root regularity}
\label{sec:multi-EPD-dossier}

Throughout \Cref{sec:multi-EPD-dossier,sec:root-DeGiorgi},
$\Omega\subset\mathbb R^d$ is a bounded domain of class $C^2$,
$p>d+2$, $m=N-1$, and
\[
 R=\prod_{j=1}^m[a_j,b_j]
 \Subset\prod_{j=1}^m(g_j,g_{j+1})
\]
is a fixed compact root rectangle whenever a root-box constant is used.

This section establishes the estimates used in the subsequent De Giorgi
argument.  The algebraic cancellations and level estimates are proved directly
in the positive-time maximal-$L^p$ strong-solution class, without assuming
prior classical root regularity or using a Morrey-to-$L^p$ bootstrap.  The
resulting Caccioppoli and logarithmic estimates, together with shrinking and
critical mass, yield a positive H\"older exponent for every root.

{In this section the EPD relations are used as a design equation for nonlinear
level-set test functionals, rather than as a device for generating
conservation-law entropy families.  They diagonalize the entropy production in
the selected root direction before localization.}

{Recall from \eqref{eq:root-diffusion-aj} that, for $z\in R$,}
\begin{equation}\label{eq:aj-Cjk-dossier}
 {\mathfrak a_j(z)=\frac1{\lambda_j(z)},\qquad
 C_{jk}(z)=\frac{\mathfrak a_j(z)+\mathfrak a_k(z)}{z_j-z_k}=-C_{kj}(z).}
\end{equation}
Then the root system \eqref{eq:root-system} reads
\begin{equation}\label{eq:root-system-dossier}
 \partial_tz_j-\Div({\mathfrak a_j(z)}\nabla z_j)
 +\sum_{k\ne j}C_{jk}(z)\nabla z_j\cdot\nabla z_k=0.
\end{equation}
All ${\mathfrak a_j}$, $C_{jk}$ and their derivatives are bounded on $R$, and
${\mathfrak a_j}\ge a_*>0$.  Moreover, the trace identity
\begin{equation}\label{eq:lambda-independent-selected-root}
 \lambda_j=e_1-\sum_{\ell\ne j}z_\ell,
\end{equation}
so ${\mathfrak a_j}$ is independent of the selected root $z_j$.

\subsection{A chain-rule production identity}

For a twice differentiable scalar state function $\Psi$ and
$\Xi=(\Xi_1,\ldots,\Xi_m)$ with $\Xi_j\in\mathbb R^d$, define
\begin{align}
 \mathfrak Q_\Psi(z;\Xi)
 :={}&\frac12\sum_{j,k=1}^m
 (\mathfrak a_j+\mathfrak a_k)
 \mathcal H^{\mathrm{aff}}_{jk}(\Psi)\,
 \Xi_j\cdot\Xi_k \notag\\
 ={}&\sum_{j=1}^m \mathfrak a_j\Psi_{jj}|\Xi_j|^2
 +\sum_{1\le j<k\le m}(\mathfrak a_j+\mathfrak a_k)
 \left(\Psi_{jk}+\frac{\Psi_j-\Psi_k}{z_j-z_k}\right)
 \Xi_j\cdot\Xi_k.
 \label{eq:Q-Psi-definition}
\end{align}
Thus the production form is the concentration-affine Hessian from
\eqref{eq:affine-root-Hessian-def} weighted by the scalar diffusion
coefficients.  In particular, the multi-EPD equations
\eqref{eq:EPD-as-affine-diagonal} are exactly the condition that this
production be diagonal in the spectral root frame.

\begin{lemma}[Strong root chain rule, including the one-sided truncations]
\label{lem:root-chain-rule-dossier}
{Let $I=(t_-,t_+)$} and let $z$ take values in the compact root rectangle $R$.
Assume, for some $p>d+2$, that
\begin{equation}\label{eq:strong-root-chain-class}
 z\in W^{1,p}(I;L^p(\Omega;\mathbb R^m))
 \cap L^p(I;W^{2,p}(\Omega;\mathbb R^m))
 \cap C(\overline I;C^1(\overline\Omega;\mathbb R^m)),
\end{equation}
that \eqref{eq:root-system-dossier} holds distributionally, and that
$\partial_\nu z_j=0$ in the Sobolev trace sense.

\emph{(i)} If $\Psi\in C^2(R)$, then
\begin{equation}\label{eq:Psi-local-identity}
 \partial_t\Psi(z)
 -\Div\left(\sum_{j=1}^m {\mathfrak a_j(z)}\Psi_j(z)\nabla z_j\right)
 +\mathfrak Q_\Psi(z;\nabla z)=0
\end{equation}
holds in $\mathcal D'(I\times\Omega)$; all nondivergence terms in this identity
belong to $L^p(I\times\Omega)+L^\infty(I\times\Omega)$.

\emph{(ii)} Fix an index $i$, a level $k$, and a sign.  Let
$F\in C^2(R)$ satisfy
\begin{equation}\label{eq:one-sided-branch-vanishing}
 F(z)=0,\qquad D F(z)=0
 \quad\text{whenever }z_i=k.
\end{equation}
Define the upper and lower branches, respectively, by
\[
 \Psi^+(z)=\mathbf 1_{\{z_i>k\}}F(z),
 \qquad
 \Psi^-(z)=\mathbf 1_{\{z_i<k\}}F(z).
\]
Then $\Psi^\pm\in C^{1,1}(R)$ and
\eqref{eq:Psi-local-identity} holds in the sense of distributions, with
$D^2\Psi^\pm=D^2F$ on the active side and $D^2\Psi^\pm=0$ on the inactive
side and on the interface (the latter value being an immaterial a.e.
representative).  In particular, no hypersurface measure is produced on
$\{z_i=k\}$.

In either case, the entropy flux in \eqref{eq:Psi-local-identity} belongs
locally to $L^p(I;W^{1,p}(\Omega;\mathbb R^d))$ and has zero normal trace.
Moreover {$t\mapsto\int_\Omega\zeta_0(x)\Psi(z(t,x))\,dx$} is absolutely
continuous for every bounded spatial cutoff {$\zeta_0$}; hence the identity may be
integrated between arbitrary time slices in $\overline I$.
\end{lemma}

\begin{proof}
{We first justify the exact production identity at the strong
maximal-regularity level, without using scalar Schauder regularity.}  By
\eqref{eq:strong-root-chain-class} and $p>d+2$, the gradients are bounded on
the closed strip.  The coefficients ${\mathfrak a_j},C_{jk}$ and all state derivatives
appearing below are bounded on $R$.  Thus
${\mathfrak a_j(z)}\nabla z_j\in L^p(I;W^{1,p})$, the quadratic root terms are bounded,
and every term in the root equation is represented by an $L^p$ function.

For $\Psi\in C^2(R)$, the Bochner chain rule gives
\[
 \partial_t\Psi(z)=\sum_j\Psi_j(z)\partial_tz_j
 \quad\text{in }L^p(I\times\Omega),
\]
and the Sobolev spatial chain rule gives
\[
 \nabla\Psi_j(z)=\sum_k\Psi_{jk}(z)\nabla z_k.
\]
Multiply the $j$-th root equation by the bounded function $\Psi_j(z)$ and
use the distributional product rule
\[
 \Psi_j\Div({\mathfrak a_j}\nabla z_j)
 =\Div({\mathfrak a_j}\Psi_j\nabla z_j)-{\mathfrak a_j}\nabla\Psi_j\cdot\nabla z_j.
\]
After summing in $j$, the diagonal Hessian terms are
${\mathfrak a_j}\Psi_{jj}|\nabla z_j|^2$, while for $j<k$ the mixed Hessian coefficient is
$({\mathfrak a_j}+{\mathfrak a_k})\Psi_{jk}$.  The explicit quadratic terms pair as
\[
 C_{jk}(\Psi_j-\Psi_k)\nabla z_j\cdot\nabla z_k
 =({\mathfrak a_j}+{\mathfrak a_k})\frac{\Psi_j-\Psi_k}{z_j-z_k}
   \nabla z_j\cdot\nabla z_k.
\]
This proves part~(i) distributionally.

For part~(ii), in target space
\[
 D\Psi^\pm=\mathbf 1_{\rm active}DF,
\]
where $\mathbf 1_{\rm active}$ denotes the indicator of $\{z_i>k\}$ in the
upper branch and of $\{z_i<k\}$ in the lower branch.  Because $DF=0$ on $\{z_i=k\}$ and $D^2F$ is bounded, this vector field is
globally Lipschitz on $R$ and its weak target derivative is
\begin{equation}\label{eq:weak-Hessian-one-sided-branch}
 D(D\Psi^\pm)=\mathbf 1_{\rm active}D^2F.
\end{equation}
There is no surface measure because the trace of $D\Psi^\pm$ has no jump.
The Lipschitz Sobolev chain rule therefore applies to
$D\Psi^\pm(z)$.  On the preimage
$E=\{(t,x):z_i(t,x)=k\}$ all components of $D\Psi^\pm(z)$ vanish.
For completeness, the zero-set property used here is the elementary
Stampacchia consequence of the positive/negative-part chain rules: if
$G\in W^{1,p}$, then $\nabla G=0$ almost everywhere on $\{G=0\}$.
Applying this componentwise to $G=D_\ell\Psi^\pm(z)$ gives
$\nabla_x(D\Psi^\pm(z))=0$ almost everywhere on $E$.  Choosing the right-hand
side of \eqref{eq:weak-Hessian-one-sided-branch} to be zero on the target
interface yields globally
\[
 \nabla_x\Psi_j^\pm(z)
 =\sum_{\ell=1}^m\Psi_{j\ell}^\pm(z)\nabla_x z_\ell
 \qquad\text{a.e.}
\]
The first-order Bochner time chain rule uses only that
$\Psi^\pm\in C^1$ with bounded derivative.  Repeating the algebra from
part~(i) proves the distributional identity.  This is the strong-solution
version of the familiar identities
$\nabla(z_i-k)_+=\mathbf 1_{\{z_i>k\}}\nabla z_i$ and
$\nabla(k-z_i)_+=-\mathbf 1_{\{z_i<k\}}\nabla z_i$.

For either part, differentiating the entropy flux spatially produces only
$D^2z\in L^p$ and products of bounded first derivatives, so the flux belongs
to $L^p(I;W^{1,p})$.  Its normal trace is
$\sum_j {\mathfrak a_j}\Psi_j\partial_\nu z_j=0$.  Finally,
$\Psi(z)\in W^{1,p}(I;L^p(\Omega))$ and
$\partial_t\Psi(z)=D\Psi(z)\cdot z_t$; multiplying by a bounded spatial cutoff
and integrating in $x$ gives the asserted absolute continuity in time.
Standard smooth time cutoffs therefore justify integration of
\eqref{eq:Psi-local-identity} up to arbitrary time slices.  No classical
approximation of the nonlinear root system is used.
\end{proof}

Applying the just-proved multi-EPD identity
\eqref{eq:mixing-entropy-multi-EPD} in the production formula
\eqref{eq:Q-Psi-definition} cancels every mixed coefficient and gives the
pointwise identity
\begin{equation}\label{eq:root-entropy-production-explicit}
 \mathfrak Q_{\hz}(z;{\Xi})
 =\sum_{j=1}^m {\mathfrak a_j(z)}(\hz)_{jj}(z)|{\Xi_j}|^2.
\end{equation}
{By \eqref{eq:mixing-entropy-root-diagonal-positivity}, compactness of $R$, and
the uniform bounds for $\mathfrak a_j$ on $R$, there are constants}
$0<\kappa_R\le K_R$ such that {for every $\Xi\in(\mathbb R^d)^m$}
\begin{equation}\label{eq:root-entropy-qform-coercive}
 \kappa_R\sum_{j=1}^m|{\Xi_j}|^2
 \le \mathfrak Q_{\hz}(z;{\Xi})
 \le K_R\sum_{j=1}^m|{\Xi_j}|^2,
 \qquad z\in R.
\end{equation}

\subsection{The multi-EPD truncation family}

Fix $i\in\{1,\ldots,m\}$ and $k\in[a_i,b_i]$.  Define
\begin{equation}\label{eq:multi-Phi-dossier}
 \Phi_{i,k}(z)
 :=(-1)^{i-1}\int_k^{z_i}(z_i-{\sigma})
 \prod_{j\ne i}(z_j-{\sigma})\,\dd{\sigma}.
\end{equation}
For the upper and lower truncations set
\begin{equation}\label{eq:Phi-plus-minus-dossier}
 \Phi_{i,k}^+(z):=\mathbf 1_{\{z_i>k\}}\Phi_{i,k}(z),
 \qquad
 \Phi_{i,k}^-(z):=\mathbf 1_{\{z_i<k\}}\Phi_{i,k}(z).
\end{equation}
Because both $\Phi_{i,k}$ and its first derivatives vanish on $z_i=k$, these
functions belong to $C^{1,1}(R)$.  Their weak Hessians are the classical
Hessians of $\Phi_{i,k}$ on the active side and zero on the inactive side,
with no hypersurface measure on $z_i=k$.  This is precisely the special
one-sided situation covered by part~(ii) of
\Cref{lem:root-chain-rule-dossier}.  The distinction matters: a generic smooth
mollification would not preserve the exact multi-EPD identity below, whereas
the direct truncation chain rule does.

\begin{lemma}[Exact multi-EPD identities and truncation bounds]
\label{lem:multi-EPD-dossier}
There are constants $0<c_R\le C_R<\infty$ such that, with
$r=(z_i-k)_+$ for the upper truncation and $r=(k-z_i)_+$ for the lower
truncation, one has on the corresponding active side
\begin{align}
 c_Rr^2&\le \Phi_{i,k}^{\pm}(z)\le C_Rr^2,
 \label{eq:Phi-r2-comparison-dossier}\\
 |\partial_i\Phi_{i,k}|&\le C_Rr,
 \qquad
 |\partial_j\Phi_{i,k}|\le C_Rr^2\quad(j\ne i),
 \label{eq:Phi-first-derivative-bounds-dossier}\\
 \partial_{ii}\Phi_{i,k}
 &=\prod_{j\ne i}|z_j-z_i|\ge c_R,
 \qquad
 \partial_{jj}\Phi_{i,k}=0\quad(j\ne i),
 \label{eq:Phi-diagonal-Hessian-dossier}
\end{align}
and, for every $a\ne b$,
\begin{equation}\label{eq:multi-EPD-dossier}
 (z_b-z_a)\partial_{ab}\Phi_{i,k}
 =\partial_a\Phi_{i,k}-\partial_b\Phi_{i,k}.
\end{equation}
Consequently
\begin{equation}\label{eq:Q-Phi-exact-dossier}
 \mathfrak Q_{\Phi_{i,k}^{\pm}}(z;\nabla z)
 ={\mathfrak a_i(z)}\partial_{ii}\Phi_{i,k}(z)|\nabla r|^2
 \ge c_R|\nabla r|^2
\end{equation}
almost everywhere.
\end{lemma}

\begin{proof}
Differentiate \eqref{eq:multi-Phi-dossier} under the integral sign.  The upper
endpoint produces no term because of the factor $z_i-{\sigma}$, and hence
\[
 \partial_i\Phi_{i,k}=(-1)^{i-1}\int_k^{z_i}
 \prod_{j\ne i}(z_j-{\sigma})\,\dd{\sigma},
\]
while for $j\ne i$,
\[
 \partial_j\Phi_{i,k}=(-1)^{i-1}\int_k^{z_i}(z_i-{\sigma})
 \prod_{\ell\ne i,j}(z_\ell-{\sigma})\,\dd{\sigma}.
\]
A second $z_i$ derivative gives
$(-1)^{i-1}\prod_{j\ne i}(z_j-z_i)=\prod_{j\ne i}|z_j-z_i|$, and a second
$z_j$ derivative vanishes for $j\ne i$.  The compact interlacing rectangle
supplies uniform upper and lower bounds for all factors along the segment
between $k$ and $z_i$, proving
\eqref{eq:Phi-r2-comparison-dossier}--\eqref{eq:Phi-diagonal-Hessian-dossier}.

For $a=i$, $b=j\ne i$, subtraction of the first-derivative formulas yields
\[
 \partial_i\Phi_{i,k}-\partial_j\Phi_{i,k}
 =(z_j-z_i)\partial_{ij}\Phi_{i,k}.
\]
If $a,b\ne i$, the same subtraction removes one of the two linear factors in
the integrand.  This proves \eqref{eq:multi-EPD-dossier}.  Substitution in
\eqref{eq:Q-Psi-definition} cancels every mixed pair exactly and leaves only
the selected diagonal term.
\end{proof}

\begin{remark}[Localization for three or more root variables]
For $m\ge3$, integration by parts in the transverse-root fluxes
${\mathfrak a_j}\Phi_j\nabla z_j$ differentiates coefficients depending on additional
roots and therefore produces gradients not controlled by the pure truncation.
Consequently this integration-by-parts argument does not yield the required
Caccioppoli estimate.  The stabilization below provides the weighted
transverse-root dissipation required in the localized estimate without this
secondary integration by parts.
\end{remark}

\subsection{Entropy stabilization of the one-sided truncation}

{For $m\ge3$, the pure EPD truncation controls the selected-gradient production
but not the transverse-root flux terms in the localized estimate.  The
following construction provides the required transverse-gradient control.}

By \Cref{cor:mixing-entropy-multi-EPD}, the mixing entropy itself satisfies
the same pairwise multi-EPD equations as the pure truncation potentials, but
{with strictly positive diagonal second derivatives in every root direction.}  We use it
as a small coercive stabilizer.  Choose a constant $C_h$ so large that
\begin{equation}\label{eq:hhat-positive-dossier}
 \widehat h(z):=\hz(z)+C_h\ge1
 \qquad(z\in R).
\end{equation}
For the upper and lower signs, respectively, set
\[
 r=(z_i-k)_+,\qquad r=(k-z_i)_+,\qquad {\chi:=\frac12r^2}.
\]
On the active set ${\chi_i}=\sigma r$, where $\sigma\in\{1,-1\}$, and
${\chi_{ii}}=1$, while all other first derivatives of ${\chi}$ vanish.

\begin{lemma}[Product identity for the root production form]
\label{lem:Q-product-dossier}
On the active set,
\begin{equation}\label{eq:Q-product-dossier}
 \mathfrak Q_{{\chi}\widehat h}
 ={\chi}\mathfrak Q_{\hz}+\widehat h\,\mathfrak Q_{{\chi}}
 +2{\mathfrak a_i}{\chi_i}(\hz)_i|\nabla z_i|^2
 +{\chi_i}\sum_{j\ne i}({\mathfrak a_i}+{\mathfrak a_j})(\hz)_j\nabla z_i\cdot\nabla z_j.
\end{equation}
where
\begin{equation}\label{eq:Q-chi-dossier}
 \mathfrak Q_{{\chi}}
 ={\mathfrak a_i}|\nabla r|^2
 +{\chi_i}\sum_{j\ne i}\frac{{\mathfrak a_i}+{\mathfrak a_j}}{z_i-z_j}
 \nabla z_i\cdot\nabla z_j.
\end{equation}
\end{lemma}

\begin{proof}
Use
$({\chi}\widehat h)_{jk}={\chi_{jk}}\widehat h+{\chi_j}\widehat h_k
+{\chi_k}\widehat h_j+{\chi}\widehat h_{jk}$ and the corresponding first-derivative
identity in \eqref{eq:Q-Psi-definition}.  Since ${\chi_j}=0$ for $j\ne i$, only
the displayed terms remain.  Formula \eqref{eq:Q-chi-dossier} follows directly
from the definition of $\mathfrak Q_{{\chi}}$.
\end{proof}

\begin{proposition}[Coercive entropy-stabilized truncation]
\label{prop:stabilized-truncation-coercivity}
There is $\varepsilon_R>0$ such that for every
$0<\varepsilon\le\varepsilon_R$, every $i$, every $k\in[a_i,b_i]$, and either
sign, the function
\begin{equation}\label{eq:Psi-stabilized-dossier}
 \Psi_{i,k}^{\pm}(z)
 :=\Phi_{i,k}^{\pm}(z)+\varepsilon {\chi(z)}\widehat h(z)
\end{equation}
is itself exactly a one-sided branch of the form covered by
\Cref{lem:root-chain-rule-dossier}: on the active side it equals
\[
 F_{i,k,\varepsilon}(z)
 :=\Phi_{i,k}(z)+\frac{\varepsilon}{2}(z_i-k)^2\widehat h(z),
\]
and it is zero on the inactive side.  Both
$F_{i,k,\varepsilon}$ and $DF_{i,k,\varepsilon}$ vanish on $z_i=k$.
Consequently the exact production identity applies to
$\Psi_{i,k}^{\pm}$ with no interface measure.  It satisfies, almost everywhere on the active set,
\begin{equation}\label{eq:Psi-coercivity-dossier}
 \mathfrak Q_{\Psi_{i,k}^{\pm}}(z;\nabla z)
 \ge c_R|\nabla r|^2
 +c_R\varepsilon r^2\sum_{j\ne i}|\nabla z_j|^2.
\end{equation}
Moreover
\begin{equation}\label{eq:Psi-size-flux-bounds-dossier}
 c_Rr^2\le\Psi_{i,k}^{\pm}\le C_Rr^2,
 \qquad
 |\Psi_i|\le C_Rr,
 \qquad
 |\Psi_j|\le C_Rr^2\quad(j\ne i).
\end{equation}
\end{proposition}

\begin{proof}
{Write $\Xi_j=\nabla z_j$} and let
\[
 D_R:=\max_{1\le \ell\le m}(b_\ell-a_\ell),
 \qquad 0\le r\le D_R.
\]
By \eqref{eq:root-entropy-qform-coercive},
\begin{equation}\label{eq:qh-good-dossier}
 {\chi}\mathfrak Q_{\hz}(z;{\Xi})
 \ge \frac{\kappa_R}{2}r^2\sum_{j=1}^m|{\Xi_j}|^2.
\end{equation}
Let $c_0>0$ be a uniform lower bound for
${\mathfrak a_i}\partial_{ii}\Phi_{i,k}$ on all admissible $i,k,z$.  By
\Cref{lem:Q-product-dossier}, boundedness of $\widehat h$, $D\hz$, the ${\mathfrak a_j}$,
and the reciprocal root gaps on $R$, there is $C_0<\infty$, independent of
$i,k$ and of the sign, such that after discarding the nonnegative term
$\widehat h {\mathfrak a_i}|{\Xi_i}|^2$,
\begin{equation}\label{eq:qh-lower-explicit-dossier}
 \mathfrak Q_{{\chi}\widehat h}(z;{\Xi})
 \ge \frac{\kappa_R}{2}r^2\sum_{j=1}^m|{\Xi_j}|^2
      -C_0r|{\Xi_i}|^2
      -C_0r|{\Xi_i}|\sum_{j\ne i}|{\Xi_j}|.
\end{equation}
Combining this with the exact EPD identity gives
\begin{equation}\label{eq:Psi-preabsorb-dossier}
 \mathfrak Q_{\Psi_{i,k}^{\pm}}(z;{\Xi})
 \ge c_0|{\Xi_i}|^2+\frac{\varepsilon\kappa_R}{2}r^2\sum_{j=1}^m|{\Xi_j}|^2
 -\varepsilon C_0r|{\Xi_i}|^2-\varepsilon C_0r|{\Xi_i}|\sum_{j\ne i}|{\Xi_j}|.
\end{equation}
Choose first
\begin{equation}\label{eq:epsilon-first-choice-dossier}
 \varepsilon\le \frac{c_0}{4C_0\max\{D_R,1\}}.
\end{equation}
Then the third term on the right of
\eqref{eq:Psi-preabsorb-dossier} is at most $(c_0/4)|{\Xi_i}|^2$.
If $m\ge2$, Young's inequality applied to
$|\Xi_i|\sum_{j\ne i}|\Xi_j|$, followed by Cauchy--Schwarz in the finite sum,
yields a constant $C_1=C_1(R)$ such that
\begin{equation}\label{eq:cross-young-explicit-dossier}
 \varepsilon C_0r|{\Xi_i}|\sum_{j\ne i}|{\Xi_j}|
 \le \frac{c_0}{4}|{\Xi_i}|^2
      +C_1\varepsilon^2r^2\sum_{j\ne i}|{\Xi_j}|^2.
\end{equation}
For $m=1$ the sum is empty.  Reducing $\varepsilon$ once more so that
$C_1\varepsilon\le\kappa_R/4$ (with no restriction needed when $m=1$), we
obtain
\[
 \mathfrak Q_{\Psi_{i,k}^{\pm}}(z;{\Xi})
 \ge \frac{c_0}{2}|{\Xi_i}|^2
 +\frac{\varepsilon\kappa_R}{4}r^2\sum_{j\ne i}|{\Xi_j}|^2.
\]
On the active set ${\Xi_i}=\nabla z_i=\pm\nabla r$, which proves
\eqref{eq:Psi-coercivity-dossier} after decreasing the generic constant
$c_R$.  This construction gives one common $\varepsilon_R>0$ because only
finitely many selected indices occur and all constants above are uniform on
the compact root rectangle.

Finally, $\widehat h\ge1$, {$\chi=r^2/2$}, and the pure EPD truncation is
nonnegative and comparable with $r^2$.  Hence
$\Psi_{i,k}^{\pm}\asymp_R r^2$.  Moreover
\[
 \partial_i({\chi}\widehat h)={\chi_i}\widehat h+{\chi}(\hz)_i=O_R(r),
 \qquad
 \partial_j({\chi}\widehat h)={\chi}(\hz)_j=O_R(r^2)\quad(j\ne i),
\]
where $r\le D_R$ is used in the first estimate.  Together with
\eqref{eq:Phi-first-derivative-bounds-dossier} this proves
\eqref{eq:Psi-size-flux-bounds-dossier}.
\end{proof}

\section{Scalar De Giorgi regularity for every additive root}
\label{sec:root-DeGiorgi}

The quadratic estimate alone does not provide the required two-time
propagation estimate, because the temporal energy of the stabilized truncation
is only uniformly equivalent to $(z_i-k)_\pm^2$.  We therefore derive the
time-propagation estimate directly from the scalar root equation.  The
additional positive term
\[
 (z_i-k)_\pm^2\sum_{j\ne i}|\nabla z_j|^2
\]
created by the entropy stabilization controls the cross-root drift in the
logarithmic test of the selected scalar root equation.

For $x_0\in\overline\Omega$ and $\rho>0$ put
\[
 D_\rho(x_0):=B_\rho(x_0)\cap\Omega,
 \qquad
 Q_\rho^\Omega(t_0,x_0):=(t_0-\rho^2,t_0)\times D_\rho(x_0).
\]
We suppress the centers when no confusion can arise.

\begin{lemma}[Uniform local geometry up to the physical boundary]
\label{lem:uniform-local-geometry-DG}
There are $\rho_\Omega>0$ and constants depending only on $d$ and the
$C^2$-geometry of $\Omega$ such that, for every
$x_0\in\overline\Omega$ and $0<\rho\le\rho_\Omega$:
\begin{enumerate}
\item $|D_\rho(x_0)|\asymp_\Omega \rho^d$, and there is a uniform thin-annulus modulus: for every $\varepsilon>0$ there is $\lambda_\varepsilon\in(0,1)$, depending only on $\varepsilon$ and the local geometry, such that
\[
 |D_\rho(x_0)\setminus D_{\lambda_\varepsilon\rho}(x_0)|
 \le \varepsilon |D_\rho(x_0)|
 \qquad(0<\rho\le\rho_\Omega);
\]
\item the Sobolev and parabolic Sobolev inequalities on $D_\rho(x_0)$ hold
with constants independent of $x_0$ and $\rho$, for functions which vanish
near the artificial boundary $\partial B_\rho(x_0)\cap\Omega$ but need not
vanish on $\partial\Omega$.  In particular, if $I$ is an interval and
{$\phi$} has this spatial support property, then
\begin{equation}\label{eq:boundary-parabolic-sobolev-DG}
 \iint_{I\times D_\rho}|{\phi}|^{2(d+2)/d}
 \le C_\Omega
 \left(\operatorname*{ess\,sup}_{t\in I}
       \int_{D_\rho}|{\phi(t)}|^2\right)^{2/d}
 \iint_{I\times D_\rho}|\nabla {\phi}|^2,
\end{equation}
\item for {$\varphi\in W^{1,1}(D_\rho)$} and $h<\ell$,
\begin{equation}\label{eq:relative-DeGiorgi-isoperimetric}
 (\ell-h)
 |\{{\varphi}\le h\}\cap D_\rho|
 |\{{\varphi}\ge \ell\}\cap D_\rho|
 \le C_\Omega\rho^{d+1}
 \int_{D_\rho\cap\{h<{\varphi}<\ell\}}|\nabla {\varphi}|\,dx.
\end{equation}
\end{enumerate}
\end{lemma}

\begin{proof}
For interior balls these are the standard Euclidean estimates.  Near
$\partial\Omega$, use a finite $C^2$ graph atlas and rescale every cap by
$x=x_0+\rho X$.  Because the rescaled graph has uniformly bounded Lipschitz
constant (and, as $\rho\downarrow0$, uniformly small curvature), the family
of rescaled sets
\[
 \rho^{-1}(D_\rho(x_0)-x_0)
\]
is a uniformly Lipschitz family.  At the flat-model level the possible center heights range over the compact
cap family $B_1\cap\{X_d>c\}$, $c\in[-1,0]$, together with the full ball;
curvature contributes only a uniformly small $C^1$ perturbation after
rescaling.  After decreasing $\rho_\Omega$ once, the flattening maps and their
inverses have one common bi-Lipschitz bound after rescaling.  The rescaled
sets therefore have a uniform Lipschitz character and a uniform non-thinning
(measure-density) constant; equivalently, they form a uniformly controlled
perturbation of the half-ball cap family.  Uniform Sobolev--Poincar\'e and
relative isoperimetric inequalities for classes with fixed cone data are
proved in \cite[Theorem~3.1 and Corollaries~3.1--3.2]{Thomas2015}.
Applying these estimates to the normalized caps, and using stability under the
uniformly bi-Lipschitz flattening maps, gives constants uniform for the entire
family.  The level-set form \eqref{eq:relative-DeGiorgi-isoperimetric} follows
from the relative isoperimetric inequality by the standard truncation/coarea
argument and scaling; no artificial-boundary support assumption is needed for
{$\varphi\in W^{1,1}(D_\rho)$}.
For the functions used in the parabolic Sobolev step, which do vanish before
the artificial spherical boundary, one may equivalently flatten, reflect
evenly across the physical boundary, and apply the Euclidean inequality on a
fixed enlarged ball; integrating the resulting spatial inequality in time
gives \eqref{eq:boundary-parabolic-sobolev-DG}.  Compactness of the finite
atlas gives one common $\rho_\Omega$ and the uniform measure-density estimate
$|D_\rho(x_0)|\ge c_\Omega\rho^d$.  The thin-annulus assertion then follows
even more directly from
\[
 |D_\rho\setminus D_{\lambda\rho}|
 \le |B_\rho\setminus B_{\lambda\rho}|
 =\omega_d(1-\lambda^d)\rho^d
 \le \frac{\omega_d}{c_\Omega}(1-\lambda^d)|D_\rho|.
\]  Below we use only these geometric consequences of boundary flattening---the uniform
Lipschitz character, volume-density bound, and thin-annulus estimate.  We do not
extend the Maxwell--Stefan solution across the physical boundary by reflection.
\end{proof}

\subsection{Quadratic and two-time stabilized estimates}

The one-sided Caccioppoli estimate from the preceding section remains our
basic level-set estimate.  We record in addition the two-time form which
retains the intrinsic stabilized energy at the initial slice.

\begin{proposition}[One-sided root Caccioppoli inequality]
\label{prop:root-Caccioppoli-dossier}
Let $z$ satisfy the strong root hypotheses of
\Cref{lem:root-chain-rule-dossier} on a neighbourhood of
$Q_\rho^\Omega$, assume that its range lies in the compact root rectangle
$R$, and fix a root $i$ and a level $k\in[a_i,b_i]$.  For the upper and lower
truncations put $r=(z_i-k)_+$ and $r=(k-z_i)_+$, respectively.  There are constants $c_R,C_R>0$ such that,
for every $C^1$ cutoff $0\le\eta\le1$ which vanishes near the lower time face
and the artificial lateral boundary but may be nonzero on the physical
boundary,
\begin{align}
 &\operatorname*{ess\,sup}_{s}
   \int_{D_\rho}\eta^2r^2(s)
 +c_R\iint_{Q_\rho^\Omega}\eta^2|\nabla r|^2 \notag\\
 &\quad
 +c_R\varepsilon_R\iint_{Q_\rho^\Omega}
   \eta^2r^2\sum_{j\ne i}|\nabla z_j|^2
 \le C_R\iint_{Q_\rho^\Omega}
 r^2\bigl(|\nabla\eta|^2+\eta|\partial_t\eta|\bigr).
 \label{eq:root-Caccioppoli-cutoff-dossier-v11}
\end{align}
Consequently, for $0<\varrho<\rho$,
\begin{align}
 &\operatorname*{ess\,sup}_{t\in(t_0-\varrho^2,t_0)}
   \int_{D_\varrho}r^2(t)
 +\iint_{Q_\varrho^\Omega}|\nabla r|^2 \notag\\
 &\quad
 +c_R\varepsilon_R\iint_{Q_\varrho^\Omega}
 r^2\sum_{j\ne i}|\nabla z_j|^2
 \le \frac{C_R}{(\rho-\varrho)^2}
 \iint_{Q_\rho^\Omega}r^2.
 \label{eq:root-Caccioppoli-dossier-v11}
\end{align}
The constants are uniform for interior and boundary cylinders below
$\rho_\Omega$.
\end{proposition}

\begin{proof}
Use
$\Psi=\Psi_{i,k}^{\pm}$ from
\Cref{prop:stabilized-truncation-coercivity} with the fixed
$\varepsilon=\varepsilon_R$ in the distributional identity
\eqref{eq:Psi-local-identity}, multiply by $\eta^2$, and integrate up to an
arbitrary time.  The physical boundary contributes nothing because every
root has homogeneous Neumann data.  The flux bound
\[
 \left|\sum_j {\mathfrak a_j}\Psi_j\nabla z_j\right|
 \le C_R\left(r|\nabla z_i|+r^2\sum_{j\ne i}|\nabla z_j|\right)
\]
and Young's inequality absorb the selected flux by
$|\nabla r|^2$ and every {transverse-root} flux by
$\varepsilon_Rr^2|\nabla z_j|^2$.  Finally
$\Psi\asymp_R r^2$.  Choosing standard nested cutoffs gives
\eqref{eq:root-Caccioppoli-dossier-v11}.
\end{proof}

\begin{lemma}[Two-time stabilized level energy]
\label{lem:two-time-stabilized-energy}
Under the hypotheses of
\Cref{prop:root-Caccioppoli-dossier}, {let $t_1<t_2$ lie in the time interval of existence} and let $\zeta=\zeta(x)$ be a time-independent spatial cutoff which
vanishes near the artificial boundary and may be nonzero on
$\partial\Omega$.  Then, with $r=(z_i-k)_\pm$ and
$\Psi=\Psi_{i,k}^{\pm}$,
\begin{align}
 &\int \zeta^2\Psi(z({t_2}))
 +c_R\int_{{t_1}}^{{t_2}}\!\!\int \zeta^2|\nabla r|^2
 +c_R\varepsilon_R\int_{{t_1}}^{{t_2}}\!\!\int
     \zeta^2r^2\sum_{j\ne i}|\nabla z_j|^2 \notag\\
 &\qquad\le
 \int \zeta^2\Psi(z({t_1}))
 +C_R\int_{{t_1}}^{{t_2}}\!\!\int r^2|\nabla\zeta|^2.
 \label{eq:two-time-stabilized-energy}
\end{align}
All spatial integrals are taken over the corresponding physical domain
slice.
\end{lemma}

\begin{proof}
Integrate the exact stabilized chain-rule identity between {$t_1$ and $t_2$} and
multiply by $\zeta^2$.  There is now no lower-time cutoff, so the initial
term is retained with coefficient one.  The spatial localization term is
estimated exactly as in the proof of
\Cref{prop:root-Caccioppoli-dossier}.  The coercivity
\eqref{eq:Psi-coercivity-dossier} absorbs the selected and {transverse-root}
pieces, while the physical Neumann boundary contributes no flux.
\end{proof}

\subsection{A scalar logarithmic estimate from the root equation}

Fix one root $u:=z_i$.  In divergence form its equation is
\begin{equation}\label{eq:selected-root-scalar-drift}
 u_t-\Div({\mathfrak a_i(z)}\nabla u)
 +\sum_{j\ne i}C_{ij}(z)\nabla u\cdot\nabla z_j=0.
\end{equation}
All ${\mathfrak a_i}$, ${\mathfrak a_i^{-1}}$ and $C_{ij}$ are uniformly bounded on the compact
root rectangle.  We shall use the following elementary orientation fact to
avoid treating the upper and lower phases by an informal symmetry argument.
If $\varsigma\in\{-1,1\}$ and $c\in\mathbb R$, then along the given
solution
\begin{equation}\label{eq:oriented-selected-root-equation}
 {\widetilde u}:=\varsigma z_i+c
 \quad\Longrightarrow\quad
 {\widetilde u_t}-\Div({\mathfrak a_i(z)}\nabla {\widetilde u})
 +\sum_{j\ne i}C_{ij}(z)\nabla {\widetilde u}\cdot\nabla z_j=0,
 \qquad \partial_\nu {\widetilde u}=0.
\end{equation}
Indeed, this is $\varsigma$ times the $i$-th root equation; the coefficient
field is evaluated along the original state $z$ and is unchanged.  Thus every
upper-phase argument for an oriented variable is also a lower-phase argument
for the original root.  The next lemma controls the drift in
\eqref{eq:selected-root-scalar-drift}--\eqref{eq:oriented-selected-root-equation}
by the additional weighted dissipation in
\eqref{eq:two-time-stabilized-energy}.

Fix once and for all a support parameter $\sigma_0\in(0,1/2)$; below we
may take $\sigma_0=1/4$.  For $H>0$, $0<\delta<\sigma_0/2$, and
$0\le v\le H$, define the scaled logarithmic square
\begin{align}
 \ell_{H,\delta}(v)
 &:={\left[\log\frac{(1-\sigma_0+\delta)H}
 {H-v+\delta H}\right]}_+,
 \label{eq:scaled-log-ell}\\
 {\Lambda_{H,\delta}}(v)&:=H^2\ell_{H,\delta}(v)^2.
 \label{eq:scaled-log-F}
\end{align}
Thus ${\Lambda_{H,\delta}}=0$ on $[0,\sigma_0H]$.  On the active interval
$(\sigma_0H,H]$, putting
$D(v):=H-v+\delta H$, direct differentiation gives
\begin{equation}\label{eq:scaled-log-derivatives}
 {\Lambda_{H,\delta}}'(v)=\frac{2H^2\ell_{H,\delta}(v)}{D(v)},
 \qquad {\Lambda_{H,\delta}}''(v)=\frac{2H^2(1+\ell_{H,\delta}(v))}{D(v)^2}.
\end{equation}
and hence
\begin{equation}\label{eq:scaled-log-ratio}
 \frac{({\Lambda_{H,\delta}}'(v))^2}{{\Lambda_{H,\delta}}''(v)}
 =2H^2\frac{\ell_{H,\delta}(v)^2}{1+\ell_{H,\delta}(v)}
 \le 2\sigma_0^{-2}L_\delta\,v^2,
 \qquad
 L_\delta:=\log\frac{1-\sigma_0+\delta}{\delta},
\end{equation}
where the quotient is defined as zero on the inactive interval.  By
\eqref{eq:scaled-log-ratio}, the error generated by the root drift is a factor
$L_\delta$ times the same weighted square $v^2|\nabla z_j|^2$ controlled by
the stabilized truncation energy.
The function ${\Lambda_{H,\delta}}$ is $C^{1,1}$ on $[0,H]$; the one-sided second
derivative in \eqref{eq:scaled-log-derivatives} is its weak second derivative
on the active side.  The standard scalar Lipschitz chain rule applies, and
there is no interface measure at $v=\sigma_0H$ because
${\Lambda_{H,\delta}}'(\sigma_0H)=0$.

\begin{lemma}[Scaled logarithmic-square root estimate]
\label{lem:root-logarithmic-estimate}
{Let $t_1<t_2$} and let $k\in[a_i,b_i]$.  For the upper and lower
truncations let $v=(u-k)_+$ and $v=(k-u)_+$, respectively, and suppose that
$0\le v\le H$ on the cylinder under consideration.  Let $\eta,\zeta$ be time-independent spatial cutoffs with
$0\le\eta,\zeta\le1$, $\zeta\equiv1$ on a neighborhood of
$\operatorname{supp}\eta$, and both vanishing near the artificial spatial
boundary.  Then, for $0<\delta<\sigma_0/2$,
\begin{align}
 &\int \eta^2{\Lambda_{H,\delta}}(v({t_2}))
 +c_R\int_{{t_1}}^{{t_2}}\!\!\int
   \eta^2{\Lambda_{H,\delta}}''(v)|\nabla v|^2
 \notag\\
 &\quad\le
 \int \eta^2{\Lambda_{H,\delta}}(v({t_1}))
 +C_RL_\delta\left[\int \zeta^2v^2({t_1})
 +\int_{{t_1}}^{{t_2}}\!\!\int
 v^2\bigl(|\nabla\eta|^2+|\nabla\zeta|^2\bigr)\right].
 \label{eq:root-logarithmic-estimate}
\end{align}
The constants are independent of $H$, $\delta$, {$t_1$}, and {$t_2$}.  The same
estimate holds on a boundary patch, with no additional boundary term.
\end{lemma}

\begin{proof}
We treat $v=(u-k)_+$; for the lower truncation one tests with
$-\eta^2{\Lambda_{H,\delta}}'(v)$.  On the active set of $v$, multiply
\eqref{eq:selected-root-scalar-drift} by
$\eta^2{\Lambda_{H,\delta}}'(v)$ and integrate from {$t_1$ to $t_2$}.  The temporal term is
${\Lambda_{H,\delta}}(v)$.  After integration by parts, the principal diffusion
contributes
\[
 \int_{{t_1}}^{{t_2}}\!\!\int
 {\mathfrak a_i}\eta^2{\Lambda_{H,\delta}}''(v)|\nabla v|^2.
\]
The localization and cross-root drift errors have the forms
\[
 C_R\eta|{\Lambda_{H,\delta}}'(v)|\,|\nabla v|\,|\nabla\eta|,
 \qquad C_R\eta^2|{\Lambda_{H,\delta}}'(v)|\,|\nabla v|
 \sum_{j\ne i}|\nabla z_j|.
\]
Young's inequality, the uniform lower bound for {$\mathfrak a_i$}, and
\eqref{eq:scaled-log-ratio} therefore yield
\begin{align*}
 &\int \eta^2{\Lambda_{H,\delta}}(v({t_2}))
 +c_R\int_{{t_1}}^{{t_2}}\!\!\int
 \eta^2{\Lambda_{H,\delta}}''(v)|\nabla v|^2\\
 &\qquad\le
 \int \eta^2{\Lambda_{H,\delta}}(v({t_1}))
 +C_RL_\delta\int_{{t_1}}^{{t_2}}\!\!\int
 v^2\left(|\nabla\eta|^2+\eta^2\sum_{j\ne i}|\nabla z_j|^2\right).
\end{align*}
The last integral is controlled by the extra positive term in the two-time
stabilized estimate.  Apply \Cref{lem:two-time-stabilized-energy} at the same
level and sign with cutoff $\zeta$.  Since
$\zeta\equiv1$ on $\operatorname{supp}\eta$ and
$\Psi\asymp_Rv^2$,
\[
 \int_{{t_1}}^{{t_2}}\!\!\int
 \eta^2v^2\sum_{j\ne i}|\nabla z_j|^2
 \le C_R
 \left(
   \int\zeta^2v^2({t_1})
   +\int_{{t_1}}^{{t_2}}\!\!\int v^2|\nabla\zeta|^2
 \right),
\]
where the fixed factor $\varepsilon_R^{-1}$ has been absorbed into $C_R$.
Substitution proves \eqref{eq:root-logarithmic-estimate}.

The scalar $C^{1,1}$ chain rule used above is harmless: unlike the
multivariable EPD cancellation, no target-space identity is being preserved
through an approximation.  Equivalently, one may use the weak derivative of
${\Lambda_{H,\delta}}'$; on the level set $v=\sigma_0H$ the standard Stampacchia
property gives $\nabla v=0$ almost everywhere.  At the physical boundary,
$\partial_\nu u=0$, and the auxiliary stabilized entropy flux also has zero
normal component, so no physical boundary term is produced.
\end{proof}

\begin{lemma}[Short-time propagation of a level-set fraction]
\label{lem:root-measure-propagation}
There exist numbers
\[
 \delta_0,\;\lambda_0,\;\epsilon_0,\;\gamma_0\in(0,1),
\]
depending only on the structural data, the compact root rectangle and the
local geometry of $\Omega$, with the following property.  Let
$0<\rho\le\rho_\Omega$, let {$t_*$} be a time, and choose $k,H$ with
\[
 [k,k+H]\subset[a_i,b_i].
\]
Suppose that $u\le k+H$ on $[{t_*},{t_*}+\delta_0\rho^2]\times D_\rho$.
If
\begin{equation}\label{eq:initial-low-fraction}
 |\{u({t_*})\le k\}\cap D_\rho|\ge\frac12|D_\rho|,
\end{equation}
then, for every $t\in[{t_*},{t_*}+\delta_0\rho^2]$,
\begin{equation}\label{eq:propagated-low-fraction}
 |\{u(t)\le k+(1-\epsilon_0)H\}\cap D_{\lambda_0\rho}|
 \ge\gamma_0|D_{\lambda_0\rho}|.
\end{equation}
The symmetric statement holds when the initial half-measure information is
at the upper side of an interval.
\end{lemma}

\begin{proof}
Fix $\sigma_0=1/4$.  Choose $\lambda_0\in(0,1)$ so close to one that
\begin{equation}\label{eq:thin-annulus-choice-scaled-log}
 |D_\rho\setminus D_{\lambda_0\rho}|
 \le \frac1{16}|D_\rho|
\end{equation}
for every admissible center and radius.  This follows uniformly from the
volume-density bound in \Cref{lem:uniform-local-geometry-DG}, since
$|D_\rho\setminus D_{\lambda\rho}|$ is bounded by the Euclidean annulus
$|B_\rho\setminus B_{\lambda\rho}|$.  Set
\[
 \lambda_1:=\lambda_0+\frac{1-\lambda_0}{3},
 \qquad
 \lambda_2:=\lambda_0+\frac{2(1-\lambda_0)}{3},
\]
so that $\lambda_0<\lambda_1<\lambda_2<1$.  Choose time-independent
spatial cutoffs $\eta,\zeta$ with
\[
 \eta\equiv1\ \hbox{on }D_{\lambda_0\rho},\qquad
 \operatorname{supp}\eta\subset D_{\lambda_1\rho},
\]
\[
 \zeta\equiv1\ \hbox{on }D_{\lambda_1\rho},\qquad
 \operatorname{supp}\zeta\subset D_{\lambda_2\rho}.
\]
They may be chosen to vanish near the artificial spherical boundary while
remaining unrestricted on the physical boundary, with
\[
 |\nabla\eta|+|\nabla\zeta|
 \le \frac{C}{(1-\lambda_0)\rho}.
\]
In particular, $\zeta\equiv1$ on a neighborhood of
$\operatorname{supp}\eta$, exactly as required in
\Cref{lem:root-logarithmic-estimate}.
Put $v=(u-k)_+$, so $0\le v\le H$.  At time $t_*$,
\eqref{eq:initial-low-fraction} implies
$|\{v({t_*})>0\}\cap D_\rho|\le|D_\rho|/2$.  Consequently,
\begin{equation}\label{eq:scaled-log-initial-bound}
 \int\eta^2{\Lambda_{H,\delta}}(v({t_*}))
 \le \frac12 H^2L_\delta^2|D_\rho|,
 \qquad
 \int\zeta^2v^2(t_*)\le\frac12H^2|D_\rho|.
\end{equation}
For $0<\delta<\sigma_0/2$ put
\[
 L_\delta^-:=\log\frac{1-\sigma_0+\delta}{2\delta}
 =L_\delta-\log2.
\]
If $v>(1-\delta)H$, then
\begin{equation}\label{eq:scaled-log-target-lower-bound}
 {\Lambda_{H,\delta}}(v)\ge H^2(L_\delta^-)^2.
\end{equation}
Apply \Cref{lem:root-logarithmic-estimate} between {$t_*$} and an arbitrary
$t\in[{t_*},{t_*}+\delta_0\rho^2]$.  Using $v\le H$ and the cutoff bounds gives
\begin{equation}\label{eq:scaled-log-propagation-ratio}
 H^2(L_\delta^-)^2|\{v(t)>(1-\delta)H\}\cap D_{\lambda_0\rho}|
 \le H^2\left[\frac12L_\delta^2+C_RL_\delta\left(1+\frac{\delta_0}{(1-\lambda_0)^2}\right)\right]|D_\rho|.
\end{equation}
The constants are now chosen in a non-circular order.  The number
$\lambda_0$ has already been fixed by the thin-annulus requirement.  Hence
\[
 A_0:=C_R\left(1+(1-\lambda_0)^{-2}\right)
\]
is a structural constant.  Since
\[
 \frac{L_\delta}{L_\delta^-}\longrightarrow1,
 \qquad
 \frac{L_\delta}{(L_\delta^-)^2}\longrightarrow0
 \qquad(\delta\downarrow0),
\]
we may choose one
$\epsilon_0\in(0,\sigma_0/2)$ so small that
\begin{equation}\label{eq:epsilon0-explicit-log-choice}
 \frac{\frac12L_{\epsilon_0}^2+A_0L_{\epsilon_0}}
 {(L_{\epsilon_0}^-)^2}\le\frac9{16}.
\end{equation}
We then set $\delta=\epsilon_0$ in the logarithmic test and choose, only
afterward,
\[
 \delta_0:=\frac14\min\{1,\lambda_0^2\}.
\]
Because
$C_R(1+\delta_0(1-\lambda_0)^{-2})\le A_0$ after enlarging the harmless
structural constant in the definition of $A_0$, the estimate
\eqref{eq:epsilon0-explicit-log-choice} applies to
\eqref{eq:scaled-log-propagation-ratio}.
By \eqref{eq:thin-annulus-choice-scaled-log},
$|D_{\lambda_0\rho}|\ge15|D_\rho|/16$.  Hence
\[
 \frac{|\{v(t)>(1-\epsilon_0)H\}\cap D_{\lambda_0\rho}|}
 {|D_{\lambda_0\rho}|}
 \le \frac{9}{15}=\frac35.
\]
Thus \eqref{eq:propagated-low-fraction} holds with, for example,
$\gamma_0=2/5$.  The lower-side version follows by applying the same argument
to $v=(k-u)_+$ at the corresponding root level $k\in[a_i,b_i]$ and testing
the root equation with the opposite sign.
\end{proof}

\subsection{Critical mass, shrinking, and oscillation decay}

\begin{lemma}[Critical mass from the quadratic Caccioppoli estimate]
\label{lem:root-critical-mass}
There is a constant $\nu_*>0$, depending only on $d$, $R$, and $\Omega$, such that the following holds.  Suppose
$Q_{2\rho}^\Omega(t_0,x_0)$ lies in a positive-time strip, let
$[k,k+H]\subset[a_i,b_i]$, let $0\le (u-k)_+\le H$ there, and assume
\[
 \frac{|\{u>k\}\cap Q_{2\rho}^\Omega|}
 {|Q_{2\rho}^\Omega|}\le\nu_*.
\]
Then
\begin{equation}\label{eq:critical-mass-conclusion}
 \operatorname*{ess\,sup}_{Q_\rho^\Omega}(u-k)_+\le\frac H2.
\end{equation}
The corresponding conclusion for the lower truncation $(k-u)_+$ holds as well.
\end{lemma}

\begin{proof}
We write the De Giorgi iteration with all scales explicit.  Put
\[
 \rho_n:=\rho(1+2^{-n}),\qquad
 k_n:=k+\frac H2(1-2^{-n}),\qquad n=0,1,2,\ldots .
\]
Thus $\rho_0=2\rho$, $\rho_n\downarrow\rho$, $k_0=k$, and
$k_n\uparrow k+H/2$.  Choose parabolic cutoffs $\eta_n$ which are one on
$Q_{\rho_{n+1}}^\Omega$, vanish near the lower and artificial lateral
boundary of $Q_{\rho_n}^\Omega$, and satisfy
\[
 |\nabla\eta_n|\le C2^n/\rho,\qquad
 |\partial_t\eta_n|\le C4^n/\rho^2.
\]
Apply \eqref{eq:root-Caccioppoli-cutoff-dossier-v11} to
{$r_n:=(u-k_n)_+$}.  The uniform parabolic Sobolev inequality
\eqref{eq:boundary-parabolic-sobolev-DG}, valid for every $d\ge2$ with exponent
{$q_*=2(d+2)/d$}, gives
\[
 \iint |{r_n}\eta_n|^{{q_*}}
 \le C
 \left(\operatorname*{ess\,sup}_t\int {r_n}^2\eta_n^2\right)^{2/d}
 \iint |\nabla({r_n}\eta_n)|^2.
\]
For completeness, the resulting recursion can be read off without any
hidden interpolation exponent.  The cutoff estimate gives
\[
 E_n:=\operatorname*{ess\,sup}_t\int ({r_n}\eta_n)^2
      +\iint |\nabla({r_n}\eta_n)|^2
 \le C_R\frac{4^n}{\rho^2}H^2|A_n|,
 \qquad
 A_n:=\{u>k_n\}\cap Q_{\rho_n}^\Omega.
\]
Hence the parabolic Sobolev inequality
\eqref{eq:boundary-parabolic-sobolev-DG} implies
\[
 \iint |{r_n}\eta_n|^{{q_*}}\le C_R E_n^{1+2/d}.
\]
On $A_{n+1}$ one has $\eta_n=1$ and
${r_n}\ge k_{n+1}-k_n=H2^{-n-2}$.  Therefore
\[
 |A_{n+1}|
 \le C_R b^n\rho^{-2(d+2)/d}|A_n|^{1+2/d}
\]
for one $b>1$.  Since $|Q_{2\rho}^\Omega|\asymp_\Omega
\rho^{d+2}$, the normalized quantities
\[
 Y_n:=\frac{|A_n|}{|Q_{2\rho}^\Omega|}
\]
satisfy
\[
 Y_{n+1}\le C_R b^nY_n^{1+2/d}.
\]
If
$Y_{n+1}\le A B^nY_n^{1+\delta_*}$ with $A>0$, $B>1$, $\delta_*>0$, then the
induction
\[
 Y_n\le A^{-1/\delta_*}
 B^{-n/\delta_*-1/\delta_*^2}
\]
holds provided
$Y_0\le A^{-1/\delta_*}B^{-1/\delta_*^2}$.  Apply this with
$\delta_*=2/d$ and with $A,B$ equal to the structural constants in the
preceding recursion.  Hence one may choose a structural
$\nu_*>0$ so that $Y_0\le\nu_*$ implies $Y_n\to0$.  The limiting
level is $k+H/2$, which proves
\eqref{eq:critical-mass-conclusion}.  The boundary constants are uniform by
\Cref{lem:uniform-local-geometry-DG}.
\end{proof}

\begin{lemma}[De Giorgi shrinking under a persistent opposite phase]
\label{lem:root-shrinking}
Fix $\gamma\in(0,1)$, $\delta\in(0,1)$, and a spatial buffer factor
$\Lambda>1$.  Let
$I=(s,s+\delta\rho^2)$ and suppose
\begin{equation}\label{eq:persistent-low-phase}
 |\{u(t)\le k_0\}\cap D_\rho|\ge\gamma|D_\rho|
 \qquad\text{for every }t\in I.
\end{equation}
Assume also that $D_{\Lambda\rho}$ lies in the same local chart,
\[
 [k_0,k_0+M]\subset[a_i,b_i],
\]
and
\[
 u\le k_0+M
 \qquad\text{on }I\times D_{\Lambda\rho}.
\]
Define
\[
 k_j:=k_0+M\bigl(1-2^{-j}\bigr),\qquad j=0,1,2,\ldots.
\]
Then
\begin{equation}\label{eq:shrinking-measure-estimate}
 \frac{|\{u>k_J\}\cap(I\times D_\rho)|}
 {|I\times D_\rho|}
 \le \frac{C_R(\Lambda)}{\gamma\sqrt{\delta J}},
 \qquad J\ge1.
\end{equation}
The same statement holds with every inequality reversed.
\end{lemma}

\begin{proof}
Set $A_j(t)=\{u(t)>k_j\}\cap D_\rho$ and
$E_j(t)=\{k_j<u(t)<k_{j+1}\}\cap D_\rho$.  Because
$\{u\le k_0\}\subset\{u\le k_j\}$, the relative isoperimetric inequality
\eqref{eq:relative-DeGiorgi-isoperimetric} gives, for every $t\in I$,
\[
 (k_{j+1}-k_j)|A_{j+1}(t)|
 \le \frac{C_\Omega\rho}{\gamma}
 \int_{E_j(t)}|\nabla u|\,dx.
\]
Integrating in time and applying Cauchy--Schwarz yields
\begin{equation}\label{eq:shrinking-pre-sum}
 (k_{j+1}-k_j)|A_{j+1}|
 \le \frac{C_\Omega\rho}{\gamma}
 \left(\iint_{I\times D_\rho}|\nabla(u-k_j)_+|^2\right)^{1/2}
 |E_j|^{1/2},
\end{equation}
where $A_j,E_j$ now denote the corresponding space--time sets.

The gradient estimate uses the \emph{two-time} stabilized inequality, so no
hidden temporal buffer is needed.  Choose a spatial cutoff $\zeta$ which is
one on $D_\rho$, supported in $D_{\Lambda\rho}$ relative to the artificial
boundary, and satisfies
$|\nabla\zeta|\le C_\Lambda/\rho$.  Since
\[
 0\le(u-k_j)_+\le M2^{-j}
 \qquad\text{on }I\times D_{\Lambda\rho},
\]
\Cref{lem:two-time-stabilized-energy}, applied between $s$ and
$s+\delta\rho^2$, gives
\begin{equation}\label{eq:shrinking-gradient-bound-buffered}
 \iint_{I\times D_\rho}|\nabla(u-k_j)_+|^2
 \le C_R(\Lambda)\,M^2 4^{-j}\rho^d.
\end{equation}
Indeed, the initial stabilized energy is bounded by
$C M^2 4^{-j}\rho^d$.  The spatial cutoff error is bounded by
$C M^2 4^{-j}\delta\rho^d$, because the time length is
$\delta\rho^2$ and $|\nabla\zeta|\le C_\Lambda/\rho$; here
$\delta<1$.
Since $k_{j+1}-k_j=M2^{-j-1}$,
\eqref{eq:shrinking-pre-sum} and
\eqref{eq:shrinking-gradient-bound-buffered} imply
\[
 |A_{j+1}|^2
 \le \frac{C_R(\Lambda)}{\gamma^2}\rho^{d+2}|E_j|.
\]
For $j=0,\ldots,J-1$, the set $A_J$ is contained in $A_{j+1}$, while the
transition sets $E_j$ are pairwise disjoint.  Summing gives
\[
 J|A_J|^2
 \le \frac{C_R(\Lambda)}{\gamma^2}\rho^{d+2}|I\times D_\rho|.
\]
Since $|I\times D_\rho|\asymp_\Omega\delta\rho^{d+2}$, division by its square
proves \eqref{eq:shrinking-measure-estimate}.
\end{proof}

\begin{proposition}[One-step oscillation reduction]
\label{prop:root-oscillation-reduction}
There are structural constants $\vartheta_*,\eta_*\in(0,1)$ and
$\rho_*>0$ such that the following holds.  Let $u=z_i$ be any selected root
and let $Q_{4\rho}^\Omega(t_0,x_0)$ be any interior or boundary cylinder
contained in a positive-time strip, with $0<\rho\le\rho_*$.  The radius
$\rho_*$ is chosen so that $4\rho_*\le\rho_\Omega$, hence every spatial set
used below belongs to one of the uniform local-geometry charts.  Then
\begin{equation}\label{eq:one-step-oscillation-reduction}
 \operatorname*{ess\,osc}_{Q_{\vartheta_*\rho}^\Omega(t_0,x_0)}u
 \le(1-\eta_*)
 \operatorname*{ess\,osc}_{Q_{4\rho}^\Omega(t_0,x_0)}u.
\end{equation}
The constants are independent of the distance of $t_0$ to a finite upper
time face.
\end{proposition}

\begin{proof}
Let
\[
 \mu^+=\operatorname*{ess\,sup}_{Q_{4\rho}^\Omega}u,
 \qquad
 \mu^-=\operatorname*{ess\,inf}_{Q_{4\rho}^\Omega}u,
 \qquad \omega=\mu^+-\mu^-.
\]
There is nothing to prove if $\omega=0$.  On the positive-time cylinder the
strong representative is continuous, so the essential extrema agree with the
ordinary extrema on every compactly contained closed subcylinder; in
particular the slice values used below lie between $\mu^-$ and $\mu^+$.
Put {$\mu^0:=(\mu^++\mu^-)/2$} and choose
{$t_*:=t_0-\delta_0\rho^2$}, where $\delta_0$ is furnished by
\Cref{lem:root-measure-propagation}.  At the time {$t_*$}, one of the two sets
\[
 \{u({t_*})\le {\mu^0}\}\cap D_\rho,
 \qquad
 \{u({t_*})\ge {\mu^0}\}\cap D_\rho
\]
has measure at least $|D_\rho|/2$.  We treat the first alternative.  Apply
\Cref{lem:root-measure-propagation} with
$k=\mu^0$ and $H=\omega/2$.  There are fixed
$\lambda_0,\epsilon_0,\gamma_0>0$ such that for every
$t\in(s,t_0)$,
\begin{equation}\label{eq:phase-persistence-oscillation-proof}
 |\{u(t)\le \mu^+-M\}\cap D_{\lambda_0\rho}|
 \ge\gamma_0|D_{\lambda_0\rho}|,
 \qquad M:=\frac{\epsilon_0\omega}{2}.
\end{equation}
Use \Cref{lem:root-shrinking} on this terminal slab with phase fraction
$\gamma=\gamma_0$, spatial radius $\bar\rho:=\lambda_0\rho$, base level
$k_0=\mu^+-M$, aspect ratio
$\bar\delta:=\delta_0/\lambda_0^2\in(0,1)$, and spatial buffer
$D_\rho=D_{\Lambda\bar\rho}$ with $\Lambda=\lambda_0^{-1}$.  The required
upper bound holds on the buffered cylinder because $u\le\mu^+$ on
$Q_{4\rho}^\Omega$.  All constants are therefore structural.

Choose explicitly the critical-mass radius
\begin{equation}\label{eq:critical-cylinder-radius-choice}
 \rho_c:=
 \frac14\min\{\lambda_0,\sqrt{\delta_0}\}\rho .
\end{equation}
Then
\[
 2\rho_c\le \frac12\lambda_0\rho=\frac12\bar\rho,
 \qquad
 4\rho_c^2\le\frac14\delta_0\rho^2,
\]
so the backward cylinder
$Q_{2\rho_c}^\Omega(t_0,x_0)$ is contained in the terminal slab
$(s,t_0)\times D_{\bar\rho}$.  By the volume-density bounds in
\Cref{lem:uniform-local-geometry-DG}, there is a structural number
$\chi_*>0$ such that
\begin{equation}\label{eq:critical-cylinder-slab-ratio}
 \frac{|Q_{2\rho_c}^\Omega(t_0,x_0)|}
 {|(s,t_0)\times D_{\bar\rho}|}\ge\chi_* .
\end{equation}
Indeed, both numerator and denominator are fixed structural multiples of
$\rho^{d+2}$.

The shrinking lemma gives
\[
 \frac{|\{u>\mu^+-M2^{-J}\}\cap
 ((s,t_0)\times D_{\bar\rho})|}
 {|(s,t_0)\times D_{\bar\rho}|}
 \le
 \frac{C_*}{\sqrt{J}},
\]
where $C_*$ is structural because
$\gamma_0,\bar\delta,\Lambda$ are structural.  Choose once and for all
$J_*$ so large that
\begin{equation}\label{eq:Jstar-choice-critical-mass}
 \frac{C_*}{\sqrt{J_*}}\le \nu_*\chi_* .
\end{equation}
Then, by \eqref{eq:critical-cylinder-slab-ratio},
\[
 \frac{|\{u>\mu^+-M2^{-J_*}\}\cap
 Q_{2\rho_c}^\Omega|}
 {|Q_{2\rho_c}^\Omega|}
 \le \nu_* .
\]
Thus the hypothesis of \Cref{lem:root-critical-mass} is verified literally.
{Applying that lemma with
\[
 k=\mu^+-M2^{-J_*},
 \qquad H=M2^{-J_*},
\]
gives}
\[
 \operatorname*{ess\,sup}_{Q_{\rho_c}^\Omega}u
 \le \mu^+-\frac{M}{2^{J_*+1}}.
\]
The lower bound $u\ge\mu^-$ is unchanged, so the oscillation decreases by
the fixed fraction
\[
 \eta_*:=\frac{\epsilon_0}{2^{J_*+2}}.
\]
For the second half-measure alternative use the sign-reversed versions of
\Cref{lem:root-measure-propagation,lem:root-shrinking,lem:root-critical-mass}
at the corresponding levels of the original root $u$.  Equivalently one may
write $w=-u+\mu^+$ as in \eqref{eq:oriented-selected-root-equation}; every
truncation $(w-\kappa)_+$ is then exactly a lower truncation $(k-u)_+$ at the
root level $k=\mu^+-\kappa\in[a_i,b_i]$.  Hence the same stabilized
estimates apply, and the resulting upper-bound reduction for $w$ is precisely
a lower-bound increase for $u$ by the same structural fraction.
Taking $\vartheta_*=\rho_c/\rho$ proves
\eqref{eq:one-step-oscillation-reduction}.  Every cylinder used is backward
from $t_0$, so none of the constants sees the distance to a finite future
time face.
\end{proof}

{
\begin{lemma}[Oscillation decay implies terminal boundary H\"older control]
\label{lem:oscillation-to-terminal-holder}
Let $u$ be continuous on $[\tau,T)\times\overline\Omega$ and satisfy
$|u|\le M_R$ there.  Assume that there are $q,\sigma\in(0,1)$ and $r_*>0$
such that, whenever $x_0\in\overline\Omega$, $0<r\le r_*$ and
$(t_0-r^2,t_0)\subset(\tau,T)$,
\begin{equation}\label{eq:abstract-oscillation-contraction-blue}
 \operatorname*{osc}_{Q_{qr}^\Omega(t_0,x_0)}u
 \le \sigma\,
 \operatorname*{osc}_{Q_r^\Omega(t_0,x_0)}u .
\end{equation}
Assume also the uniform boundary-volume and local-geometry bounds of
\Cref{lem:uniform-local-geometry-DG}.  Then, for every $\tau_1>\tau$, there
are $\theta\in(0,1)$ and $C<\infty$, depending only on
$q,\sigma,r_*,M_R,\tau_1-\tau$ and the uniform geometry, such that
\begin{equation}\label{eq:abstract-terminal-holder-blue}
 |u(t,x)-u(s,y)|
 \le C\bigl(|x-y|+|t-s|^{1/2}\bigr)^\theta
\end{equation}
for all $(t,x),(s,y)\in[\tau_1,T)\times\overline\Omega$.  If $T<\infty$,
the constant is independent of the distances of $t$ and {$t_*$} from the upper
time face $T$.
\end{lemma}

\begin{proof}
Set
\[
 \theta_0:=\frac{\log(1/\sigma)}{\log(1/q)}>0,
 \qquad \theta:=\min\{\theta_0,1/2\}.
\]
Iteration of \eqref{eq:abstract-oscillation-contraction-blue} gives
\[
 \operatorname*{osc}_{Q_{q^nr}^\Omega(t_0,x_0)}u
 \le \sigma^n\operatorname*{osc}_{Q_r^\Omega(t_0,x_0)}u
 \le 2M_R q^{n\theta_0}
\]
for every $n\ge0$ for which the outer cylinder is admissible.  Hence, if
$0<\varrho\le r$ and $q^{n+1}r<\varrho\le q^nr$, monotonicity of oscillation
under restriction yields
\begin{equation}\label{eq:campanato-oscillation-blue}
 \operatorname*{osc}_{Q_\varrho^\Omega(t_0,x_0)}u
 \le C M_R\left(\frac{\varrho}{r}\right)^\theta .
\end{equation}

Fix $\tau_1>\tau$ and choose
\[
 0<r_1\le
 \min\left\{r_*,\frac14\sqrt{\tau_1-\tau},\rho_\Omega\right\}.
\]
Every backward cylinder of radius at most $r_1$ with upper time in
$[\tau_1,T)$ therefore remains above the lower time face $\tau$ and lies in
the uniform local-geometry regime.  Let $X=(t,x)$ and $Y=(s,y)$ belong to
$[\tau_1,T)\times\overline\Omega$, and set
$d_p(X,Y):=|x-y|+|t-s|^{1/2}$.  The assertion is trivial if $X=Y$.  If
$d_p(X,Y)\ge r_1/8$, boundedness gives
\[
 |u(X)-u(Y)|\le2M_R
 \le 2M_R(8/r_1)^\theta d_p(X,Y)^\theta.
\]
Suppose that $0<d_p(X,Y)<r_1/8$ and, after interchanging $X$ and $Y$ if
necessary, that $t\ge s$.  The closure of the backward cylinder
$Q_{4d_p(X,Y)}^\Omega(t,x)$ contains both points, and its radius is at most
$r_1/2$.  Applying \eqref{eq:campanato-oscillation-blue} with starting
radius $r_1$ gives
\[
 |u(X)-u(Y)|
 \le C M_R r_1^{-\theta}d_p(X,Y)^\theta.
\]
If one or both points lie on the physical boundary, or if $X$ lies on the
upper face of the open backward cylinder, the same estimate follows by
approximation from within the cylinder and continuity of $u$.  Thus no
forward cylinder and no extension beyond $T$ is introduced.  Combining the
two distance regimes proves \eqref{eq:abstract-terminal-holder-blue} and the
claimed uniformity at a finite upper time face.
\end{proof}
}

\begin{theorem}[De Giorgi regularity of the roots]
\label{thm:root-DeGiorgi-holder}
Let $\Omega\subset\mathbb R^d$ be a bounded domain of class
$C^2$, and fix $p>d+2$.  Let $z$ be a positive-time strong root solution: on
every compact strip $I\Subset(\tau,T)$ it satisfies
\eqref{eq:strong-root-chain-class}, the distributional additive root system,
and the homogeneous Neumann trace.  Assume that its range on
$[\tau,T)\times\overline\Omega$ lies in a compact root rectangle $R$.
Then there are
$\theta=\theta(d,N,R,g,\Omega)\in(0,1)$ and, for every $\tau_1>\tau$, a
constant $C<\infty$ such that
\begin{equation}\label{eq:root-holder-terminal-strip}
 \|z\|_{C_{t,x}^{\theta/2,\theta}
 ([\tau_1,T)\times\overline\Omega)}\le C.
\end{equation}
If $T<\infty$, the constant is independent of the distance to the upper time
face $T$.
\end{theorem}

{
\begin{proof}
Fix a root $u=z_i$.  By \Cref{prop:root-oscillation-reduction}, after writing
$r=4\rho$ there are structural numbers
\[
 q:=\frac{\vartheta_*}{4}\in(0,1),\qquad
 \sigma:=1-\eta_*\in(0,1),\qquad r_*:=4\rho_*>0,
\]
such that every admissible interior or boundary backward cylinder satisfies
\[
 \operatorname*{ess\,osc}_{Q_{qr}^\Omega}u
 \le \sigma\operatorname*{ess\,osc}_{Q_r^\Omega}u,
 \qquad 0<r\le r_*.
\]
By \eqref{eq:strong-root-chain-class}, $u$ has a continuous representative on
every compact positive-time strip, so essential and ordinary oscillations
agree there.  The compact root rectangle $R$ supplies a uniform bound for
$u$, while \Cref{rem:DG-constant-ledger} shows that $q,\sigma,r_*$ are
independent of the location of the backward cylinder in time.

Apply \Cref{lem:oscillation-to-terminal-holder}.  For every $\tau_1>\tau$ it
yields an exponent
\[
 0<\theta\le
 \min\left\{\frac{\log(1/(1-\eta_*))}{\log(4/\vartheta_*)},\frac12\right\}
\]
and a constant $C_i$ such that
\[
 |z_i(t,x)-z_i(s,y)|
 \le C_i\bigl(|x-y|+|t-s|^{1/2}\bigr)^\theta
\]
throughout $[\tau_1,T)\times\overline\Omega$.  Taking the minimum exponent
and maximum constant over the finitely many roots proves
\eqref{eq:root-holder-terminal-strip}.  Since the bridge lemma uses only
backward cylinders, neither the exponent nor the constant requires a positive
distance from a finite upper time face $T$.
\end{proof}
}

\begin{remark}[Dependence of the oscillation constants]
\label{rem:DG-constant-ledger}
The constants in the Caccioppoli, logarithmic, propagation, shrinking,
critical-mass, and oscillation estimates depend only on the compact root
rectangle $R$, the additive generators (equivalently the coefficient bounds
and root-gap bounds on $R$), $N$, $d$, and the uniform local geometry of
$\Omega$.  They do \emph{not} depend on $H$, on the oscillation size, on the
chosen level, on the time location of the cylinder, or on any a-priori
$C^1$ norm of the strong solution.  The latter regularity is used only to
justify the strong chain and testing identities on compact positive-time
strips.
The time-propagation input in the oscillation argument is
\Cref{lem:root-logarithmic-estimate}.  Its logarithmic temporal energy comes
directly from the scalar root equation, while the cross-root drift is
controlled by the additional stabilized dissipation
$r^2\sum_{j\ne i}|\nabla z_j|^2$.  No Morrey-to-high-$L^p$ embedding or
replacement of the intrinsic stabilized energy by a Euclidean square is used.
\end{remark}

\section{From De Giorgi continuity to terminal strong regularity}
\label{sec:DG-to-strong}

Throughout this section, $p>d+2$ and
$\Omega\subset\mathbb R^d$ is a bounded domain of class
$C^{2+\alpha}$ for a fixed $0<\alpha<1$.

The moment system is not assumed to admit one constant Euclidean symmetrizer
on the entire moment simplex.  After \Cref{thm:root-DeGiorgi-holder}, however,
only a constant symmetrizer on a small state neighborhood is needed.  The
entropy Hessian supplies one canonically at every frozen state.

\begin{lemma}[Uniform local constant symmetrizer]\label{lem:local-constant-symmetrizer}
Let {$\mathcal K\Subset\cT^\circ$}.  There are $\delta_*>0$ and $\nu_*>0$ such that for
every {$U_0\in\mathcal K$} and every {$U\in\mathcal K$} with $|U-U_0|<\delta_*$,
\begin{equation}\label{eq:local-fixed-symmetrizer}
 \frac12\left(H(U_0)A_N(U)+A_N(U)^{\sf T}H(U_0)\right)
 \ge \nu_*I.
\end{equation}
The constants are uniform in $U_0$.
\end{lemma}

\begin{proof}
At $U=U_0$, \Cref{thm:entropy-symmetrization} makes
$H(U_0)A_N(U_0)$ symmetric positive definite.  Its smallest eigenvalue has a
positive minimum on the compact set {$\mathcal K$}.  Uniform continuity of $A_N$ and
boundedness of $H(U_0)$ on {$\mathcal K$} then give \eqref{eq:local-fixed-symmetrizer}
for one common radius $\delta_*$.
\end{proof}

\begin{lemma}[Weak-solution conormal Schauder regularity]
\label{lem:localized-conormal-schauder}
{Let $0<\beta<1$} and let
\[
 \mathcal Q_R^+:=(-R^2,0)\times(B_R'\times(0,R)).
\]
Suppose that $v:\mathcal Q_4^+\to\mathbb R^m$ is a bounded weak solution of
\begin{equation}\label{eq:abstract-conormal-system}
 \partial_tv-\partial_\gamma
 \bigl(\mathbb A^{\gamma\delta}(t,x)\partial_\delta v\bigr)=0
\end{equation}
with the homogeneous conormal condition on the flat face.  Assume that the
coefficient tensor is uniformly strongly parabolic and bounded.

\begin{enumerate}
\item If
\[
 \mathbb A\in C_{t,x}^{{\beta/2,\beta}}
 (\overline{\mathcal Q_4^+})
\]
with fixed ellipticity and H\"older bounds, then
\begin{equation}\label{eq:localized-conormal-Lipschitz}
 \|\nabla v\|_{L^\infty(\mathcal Q_1^+)}
 \le C\|v\|_{L^\infty(\mathcal Q_4^+)},
\end{equation}
and, for every {$0<\beta_1<\beta$},
\begin{equation}\label{eq:localized-conormal-schauder}
 \|v\|_{C_{t,x}^{{(1+\beta_1)/2,1+\beta_1}}
 (\overline{\mathcal Q_1^+})}
 \le C_{{\beta_1}}\|v\|_{L^\infty(\mathcal Q_4^+)}.
\end{equation}

\item If, more strongly,
\[
 \mathbb A\in C_{t,x}^{{(1+\beta)/2,1+\beta}}
 (\overline{\mathcal Q_4^+}),
\]
then, for every {$0<\beta_2<\beta$},
\begin{equation}\label{eq:localized-conormal-second-schauder}
 \|v\|_{C_{t,x}^{{1+\beta_2/2,2+\beta_2}}
 (\overline{\mathcal Q_1^+})}
 \le C_{{\beta_2}}\|v\|_{L^\infty(\mathcal Q_4^+)}.
\end{equation}
\end{enumerate}
The analogous statements hold on full cylinders.
\end{lemma}

\begin{proof}
{We first record the specialization of the Dong--Jeon notation
\cite{DongJeon2026}.  Their spatial dimension $n$ is the present $d$, their
system size is the present $m$, their weight exponent in equation~(1.1) is set
to $0$, their coefficient matrices $A^{\gamma\delta}$ are the present tensor
$\mathbb A^{\gamma\delta}$, and their divergence datum $\mathbf g$ is set
to zero.  At weight exponent $0$ the weighted measure is ordinary
space--time Lebesgue measure, and their boundedness/parabolicity condition
(1.2) is precisely the bounded strong-parabolicity hypothesis imposed here.
Their weak formulation uses test functions up to the flat face and encodes the
homogeneous conormal condition.  Translation and parabolic scaling convert
their fixed half-cylinder estimates to the $\mathcal Q_4^+$--$\mathcal Q_1^+$
formulation used below, with constants depending only on the displayed
ellipticity and coefficient-modulus bounds.}

For the spatial Lipschitz estimate, Theorem~3.1 together with the quantitative
estimate in Lemma~3.5 of \cite{DongJeon2026} applies.  A
$C_{t,x}^{{\beta/2,\beta}}$ coefficient has a Dini mean-oscillation modulus
bounded by a structural multiple of $r^{{\beta}}$, and the $L^1$ norm appearing
in their estimate is bounded by the displayed $L^\infty$ norm of $v$.
Parabolic rescaling therefore gives
\eqref{eq:localized-conormal-Lipschitz} with a constant depending only on the
stated bounds.

{For the first H\"older gain, fix
{$0<\beta_1<\beta$} and apply
\cite[Theorem~4.1 with $k=1$]{DongJeon2026} with their H\"older exponent
$\beta={\beta_1}$.  In that case the coefficient/data assumption is
$C_{t,x}^{{\beta_1/2,\beta_1}}$ and the conclusion is
$C_{t,x}^{{(1+\beta_1)/2,1+\beta_1}}$ on a smaller half-cylinder.  Our
stronger assumption $\mathbb A\in C_{t,x}^{{\beta/2,\beta}}$ gives the
required coefficient regularity and $\mathbf g=0$ belongs to every required
data class.  The quantitative modulus summarized in their formula~(1.4),
together with the fixed ellipticity, coefficient bounds and the bound for
$v$, yields \eqref{eq:localized-conormal-schauder} uniformly in the present
family.

For the second gain, fix {$0<\beta_2<\beta$} and use
\cite[Theorem~4.1 with $k=2$]{DongJeon2026}; the $k=2$ case is proved there
as Theorem~4.2.  Its coefficient/data assumption is
$C_{t,x}^{{(1+\beta_2)/2,1+\beta_2}}$ and its conclusion is
$C_{t,x}^{{1+\beta_2/2,2+\beta_2}}$.  This is exactly the regularity assumed
and asserted in part~(2), with $\mathbf g=0$.  The same quantitative
dependence gives \eqref{eq:localized-conormal-second-schauder} uniformly.}

The cited paper is formulated for at least two spatial variables.  There are
two elementary dimension/geometry reductions which also give the asserted
full-cylinder estimates without invoking a second regularity source.
For a full cylinder, introduce one new
\emph{normal} variable {$\sigma>0$}, extend $v$ and the original coefficients
constantly in {$\sigma$}, set the new normal--normal coefficient equal to $I_m$, and
set all new mixed coefficients equal to zero.  The extended function is a
weak solution on a half-cylinder, satisfies the homogeneous conormal
condition at {$\sigma=0$}, and the augmented tensor is strongly parabolic with the
same type of bounds.  {Applying the half-cylinder result and evaluating on any fixed interior slice
{$\sigma=\sigma_0\in(0,R)$} gives the full-cylinder statements.  Because the
half-cylinders above are product cylinders $B_R'\times(0,R)$, such a slice is
exactly a full cylinder in the original spatial variables; the extended
solution and all original-coordinate derivatives are independent of {$\sigma$}.}

\end{proof}

\begin{lemma}[Local first-order regularity after root H\"older continuity]
\label{lem:symmetrizable-schauder-bootstrap}
Let {$\mathcal K\Subset\mathcal T^\circ$} be compact, let $0<\beta<1$, and let
$0<\beta_1<\min\{\beta,\alpha\}$.  Suppose that $U$ is a maximal-$L^p$
strong solution of
\begin{equation}\label{eq:bootstrap-moment-system}
 \partial_tU-\Div(A_N(U)\nabla U)=0,
 \qquad \partial_\nu U=0,
\end{equation}
on $(a,b)\times\Omega$, and fix $0<\delta<(b-a)/3$.  Assume
\begin{equation}\label{eq:bootstrap-assumptions}
 U((a,b)\times\overline\Omega)\subset {\mathcal K},
 \qquad
 \|U\|_{C_{t,x}^{\beta/2,\beta}
 ([a,b)\times\overline\Omega)}\le M.
\end{equation}
Then
\begin{equation}\label{eq:bootstrap-gradient}
 \sup_{(t,x)\in[a+2\delta,b)\times\Omega}|\nabla U(t,x)|
 \le C({\mathcal K},M,\delta,\beta,\Omega,g),
\end{equation}
and, in addition,
\begin{equation}\label{eq:bootstrap-C1a}
 \|U\|_{C_{t,x}^{(1+\beta_1)/2,1+\beta_1}
 ([a+2\delta,b)\times\overline\Omega)}
 \le C({\mathcal K},M,\delta,\beta,\beta_1,\Omega,g).
\end{equation}
The gradient bound \eqref{eq:bootstrap-gradient}, rather than the stronger
H\"older estimate \eqref{eq:bootstrap-C1a}, is the only part used in the
global continuation argument.
\end{lemma}

\begin{proof}
Compactness and \Cref{thm:entropy-symmetrization} give constants
$\nu>0$, $\varepsilon>0$, and $c_H,C_H>0$ such that
\begin{equation}\label{eq:bootstrap-fixed-symmetrizer}
 \frac12\bigl(H({U_*})A_N(U)+A_N(U)^{\sf T}H({U_*})\bigr)\ge\nu I,
 \qquad c_HI\le H({U_*})\le C_HI,
\end{equation}
whenever {$U_*,U\in\mathcal K$} and $|U-{U_*}|\le\varepsilon$.
The H\"older bound in \eqref{eq:bootstrap-assumptions} supplies a radius
$r_0>0$ such that the oscillation of $U$ on every interior or boundary
cylinder of radius $4r_0$ is at most $\varepsilon$.  We also choose
\begin{equation}\label{eq:bootstrap-radius-time}
 16r_0^2<\delta .
\end{equation}
Consequently every backward cylinder used below with upper time
$t_0\ge a+2\delta$ stays a fixed positive distance from the initial face.

Fix one such cylinder, choose a state $U_*$ in the range of $U$ there, put
{$S_*:=H(U_*)^{1/2}$}, and set $v={S_*}(U-U_*)$.  In an interior chart the coefficient
tensor is
\[
 \mathbb A^{\gamma\delta}(t,x)
 =\delta_{\gamma\delta}{S_*} A_N(U(t,x)){S_*^{-1}}.
\]
The symmetric part is
\[
 \sym\!\left({S_*}A_N(U){S_*^{-1}}\right)
 ={S_*^{-1}}\sym\!\left(H(U_*)A_N(U)\right){S_*^{-1}},
\]
so \eqref{eq:bootstrap-fixed-symmetrizer} gives one Euclidean strong
parabolicity constant on the whole cylinder.

Near the boundary choose a boundary point {$x_*$}, an orientation-preserving rotation
{$Q_*\in SO(d)$}, and a local graph function $\psi$, and use the volume-preserving
flattening chart
\[
 {\Theta(\xi):=x_*+Q_*(\xi',\xi_d+\psi(\xi'))}, \qquad \det D{\Theta}=1.
\]
Set $G=(D{\Theta})^{-1}(D{\Theta})^{-T}$.  Since $\det D{\Theta}=1$, the Piola transform
puts the divergence operator directly in flat divergence form.  The transformed tensor is
\begin{equation}\label{eq:bootstrap-boundary-tensor}
 \mathbb A^{\gamma\delta}(t,\xi)
 =G^{\gamma\delta}(\xi)
  {S_*} A_N(U(t,{\Theta(\xi)})){S_*^{-1}}.
\end{equation}
The physical no-flux condition becomes exactly the homogeneous conormal
condition associated with \eqref{eq:bootstrap-boundary-tensor}, namely
\[
 \sum_{\gamma,\delta=1}^d
 (e_d)_\gamma\,\mathbb A^{\gamma\delta}\,\partial_{\xi_\delta}v=0
 \qquad\text{on }\{\xi_d=0\}.
\]
If
{$\widetilde A:=S_*A_N(U)S_*^{-1}=\widetilde A_{\rm s}+\widetilde A_{\rm a}$}, with
$\widetilde A_{\rm s}^{\sf T}=\widetilde A_{\rm s}$ and
$\widetilde A_{\rm a}^{\sf T}=-\widetilde A_{\rm a}$, then for every real array
$\zeta=(\zeta_\gamma)_{\gamma=1}^d$,
\[
 \sum_{\gamma,\delta}G^{\gamma\delta}
 \zeta_\gamma^{\sf T}{\widetilde A_{\rm a}}\zeta_\delta=0
\]
because $G^{\gamma\delta}=G^{\delta\gamma}$.  Hence only {$\widetilde A_{\rm s}$} contributes to
the quadratic form, and \eqref{eq:bootstrap-fixed-symmetrizer} together with
the uniform positive definiteness of $G$ proves the strong parabolicity
required in \Cref{lem:localized-conormal-schauder}.

For every $0<\beta_1<\min\{\beta,\alpha\}$ the tensor
\eqref{eq:bootstrap-boundary-tensor} has a uniform
$C_{t,x}^{\beta_1/2,\beta_1}$ norm: this follows from
\eqref{eq:bootstrap-assumptions}, smoothness of $A_N$ on {$\mathcal K$}, and the
$C^{1+\alpha}$ control of the boundary metric $G$.  The transformed solution is a weak conormal solution.  Indeed,
$p>d+2>2$, so on every finite cylinder the maximal-$L^p$ strong class gives
the local $L^2$ control of the solution and its spatial gradient required in
the weak formulation; the divergence identity and the smooth coordinate
change give exactly the conormal weak identity.  {At this point
every hypothesis of \Cref{lem:localized-conormal-schauder} has been accounted
for quantitatively: the compact state range bounds $v$;
\eqref{eq:bootstrap-fixed-symmetrizer} together with the chart metric gives a
uniform strong-parabolicity constant; the root/moment H\"older modulus and
the $C^{1+\alpha}$ chart metric give the coefficient H\"older norm; the
transformed divergence identity gives the weak solution; and the physical
no-flux condition is the flat conormal condition.}  After parabolic rescaling,
\Cref{lem:localized-conormal-schauder} therefore gives both a
uniform spatial gradient bound and a uniform
$C_{t,x}^{(1+\beta_1)/2,1+\beta_1}$ bound on the concentric smaller
cylinder.  The same lemma gives the corresponding interior estimates.

No terminal boundary condition is involved: all estimates use backward
cylinders and remain valid when their upper face approaches $b$.  A finite
spatial covering and the fixed radius $r_0$ therefore give
\eqref{eq:bootstrap-gradient} and \eqref{eq:bootstrap-C1a} uniformly on
$[a+2\delta,b)$.  The constants are independent of the number of time slabs
because the same local estimate and the same compact state set are used
throughout.
\end{proof}

\subsection{Terminal trace control after the spatial Lipschitz gain}

To exclude a finite maximal existence time, the continuation argument requires
a uniform bound in $X_{\gamma,p}$ up to that time.  The De Giorgi estimate
supplies a positive H\"older exponent in time.  A uniform spatial gradient
bound then controls the quadratic nondivergence forcing and places the frozen
coefficient fields in a compact family.  Full first-order Schauder regularity
is used only for subsequent smoothing.

\begin{lemma}[Terminal trace bound from time H\"older and spatial Lipschitz control]
\label{lem:terminal-trace-from-first-Schauder}
Let $b<\infty$, let $\delta_t>0$, and let $w$ be a maximal-$L^p$ strong
solution of \eqref{eq:entropy-system-nondivergence} on $[a,b)$ such that
\[
 w([a,b)\times\overline\Omega)\subset\mathcal K_w\Subset\mathcal W,
\]
\[
 [w]_{C_t^{\delta_t}([a,b);C(\overline\Omega))}
 \le M_0,
 \qquad
 \sup_{a\le t<b}\|\nabla w(t)\|_{L^\infty(\Omega)}\le M_1.
\]
Then
\begin{equation}\label{eq:terminal-Xgamma-from-lipschitz}
 \sup_{a\le t<b}\|w(t)\|_{X_{\gamma,p}}<\infty.
\end{equation}
The bound depends on the displayed data, on $b-a$, and on the structural
parameters, but not on the distance of $t$ to $b$.
\end{lemma}

\begin{proof}
Write
\[
 A(t)u:=-C(w(t,\cdot))\Delta u,\qquad D(A(t))=X_1,
 \qquad F(t):=F(w(t),\nabla w(t)).
\]
The compact state range and the spatial gradient bound give one uniform
Lipschitz constant for the coefficient fields $C(w(t,\cdot))$.  Hence this
family is relatively compact in $C(\overline\Omega)$ by Arzel\`a--Ascoli,
and every member of its closure is still Lipschitz with the same constant.
Augner's bounded-domain maximal-regularity theorem requires only continuity
of the top-order coefficients; see Assumption~5.1(a) and Theorem~5.2 of
\cite{Augner2024}.  Every coefficient field in the closure takes values in
the same compact normally elliptic matrix class and satisfies the same full
Lopatinskii--Shapiro condition from
\Cref{lem:uniform-complementing}.  Therefore the proof of
\Cref{lem:uniform-linear-MR}, followed by its finite Neumann-series covering
argument in $\mathcal L(X_1,X_0)$, gives one zero-trace maximal-regularity
constant for all frozen operators $A(t)$ on intervals of length at most a
fixed $T_0>0$.  No a-priori $X_{\gamma,p}$ bound of $w(t)$ is used in this
compactness step.

The standard trace coretraction, together with the uniform zero-trace
estimate, gives the corresponding nonzero-trace estimate: there is
$C_{\rm MR}<\infty$ such that, for every interval $I=(s,s+h)$ with
$0<h\le T_0$, every frozen $A(s)$, and every
$u_t+A(s)u=f$ on $I$,
\begin{equation}\label{eq:uniform-nonzero-trace-MR-terminal}
 \|u\|_{W^{1,p}(I;X_0)\cap L^p(I;X_1)}
 \le C_{\rm MR}\left(
 \|u(s)\|_{X_{\gamma,p}}+\|f\|_{L^p(I;X_0)}\right).
\end{equation}
For completeness, take a bounded coretraction of the trace into the
maximal-regularity space on $(0,T_0)$, subtract it from $u$, and apply the
uniform zero-trace estimate.  The uniform extension to shorter intervals is
the one used in the local theory; see \cite[Lemma~7.2]{Amann2005MaxReg}.  The
same argument gives a uniform trace embedding
\begin{equation}\label{eq:uniform-trace-embedding-terminal}
 \sup_{t\in I}\|u(t)\|_{X_{\gamma,p}}
 \le C_{\rm tr}\left(
 \|u\|_{W^{1,p}(I;X_0)\cap L^p(I;X_1)}
 +\|u(s)\|_{X_{\gamma,p}}\right).
\end{equation}

Smoothness of $C$ on the compact state range and the assumed time
H\"older control imply
\begin{equation}\label{eq:terminal-operator-time-holder}
 \|A(t)-A(s)\|_{\mathcal L(X_1,X_0)}
 \le L\|w(t)-w(s)\|_{C(\overline\Omega)}
 \le L M_0|t-s|^{\delta_t}.
\end{equation}
Furthermore the spatial Lipschitz bound gives, through the quadratic formula
\eqref{eq:entropy-quadratic-F},
\begin{equation}\label{eq:terminal-F-uniform}
 \sup_{a\le t<b}\|F(t)\|_{X_0}\le C_F(M_1,\mathcal K_w)<\infty.
\end{equation}
Choose $h_0\le T_0$ so small that
$C_{\rm MR}LM_0h_0^{\delta_t}\le1/2$.
On any $I=(s,s+h)\subset(a,b)$ with $h\le h_0$, freeze at $s$:
\[
 w_t+A(s)w=F(t)+(A(s)-A(t))w.
\]
Using \eqref{eq:uniform-nonzero-trace-MR-terminal},
\eqref{eq:terminal-operator-time-holder}, and
\eqref{eq:terminal-F-uniform}, and absorbing the last term, yields
\begin{equation}\label{eq:terminal-E-iteration}
 \|w\|_{W^{1,p}(I;X_0)\cap L^p(I;X_1)}
 \le 2C_{\rm MR}\left(
 \|w(s)\|_{X_{\gamma,p}}+C_F h^{1/p}\right).
\end{equation}
Combining this with \eqref{eq:uniform-trace-embedding-terminal} gives constants
{$C_{\rm it}\ge1$ and $C_0<\infty$}, independent of $s$ and $h\le h_0$, such that
\begin{equation}\label{eq:terminal-trace-recursion}
 \sup_{t\in[s,s+h]}\|w(t)\|_{X_{\gamma,p}}
 \le {C_{\rm it}}\|w(s)\|_{X_{\gamma,p}}+{C_0}.
\end{equation}

Starting at the fixed time $a$, cover the finite interval $[a,b)$ by at most
{$n_I:=1+\lceil(b-a)/h_0\rceil$} consecutive intervals of length at most $h_0$.
Iterating \eqref{eq:terminal-trace-recursion} finitely many times gives a
finite bound depending on {$n_I$}, $\|w(a)\|_{X_{\gamma,p}}$, {$C_{\rm it}$}, and {$C_0$}.
{No contractivity condition {$C_{\rm it}<1$} is required: $[a,b)$ has finite
length, so only finitely many applications of the affine recursion occur.  The
argument is a terminal continuation estimate, not an infinite-time stability
iteration.}  Applying the same estimate on the last truncated interval up to an arbitrary
$t<b$ shows that the bound is uniform as $t\uparrow b$.  This proves
\eqref{eq:terminal-Xgamma-from-lipschitz}.
\end{proof}

\subsection{A second conormal gain for the moment system}

The first weak-solution conormal gain makes the moment variables
$C_{t,x}^{(1+\beta)/2,1+\beta}$.  The coefficient tensor in the same locally
symmetrized divergence system therefore has one full parabolic H\"older
derivative.  The $k=2$ case of the published Dong--Jeon theorem
\cite[Theorem~4.1; see also Theorem~4.2]{DongJeon2026} can now be applied
directly to the moment system; no scalar oblique-derivative Schauder theorem
is needed.

\begin{lemma}[Second-order conormal system bootstrap]
\label{lem:moment-second-Schauder-bootstrap}
Let {$\mathcal K\Subset\mathcal T^\circ$}, let $0<\beta<1$, and let
$0<\beta_2<\min\{\beta,\alpha\}$.  Suppose that $U$ is a maximal-$L^p$
strong solution of
\[
 \partial_tU-\Div(A_N(U)\nabla U)=0,
 \qquad \partial_\nu U=0,
\]
on $(a,b)\times\Omega$, and fix $0<\delta<(b-a)/4$.  Assume
\[
 U((a,b)\times\overline\Omega)\subset {\mathcal K},
 \qquad
 \|U\|_{C_{t,x}^{(1+\beta)/2,1+\beta}
 ([a,b)\times\overline\Omega)}\le M.
\]
Then
\begin{equation}\label{eq:moment-C2a-bootstrap}
 \|U\|_{C_{t,x}^{1+\beta_2/2,2+\beta_2}
 ([a+3\delta,b)\times\overline\Omega)}
 \le C({\mathcal K},M,\delta,\beta,\beta_2,\Omega,g).
\end{equation}
The same estimate holds, after smooth coordinate composition on the compact
state range, for $y,z,$ and $w$.
\end{lemma}

\begin{proof}
Repeat the localization in
\Cref{lem:symmetrizable-schauder-bootstrap}.  On every sufficiently small
interior or boundary backward cylinder choose a state $U_*$ from the range,
put {$S_*:=H(U_*)^{1/2}$}, and set $v={S_*}(U-U_*)$.  The local constant-symmetrizer
estimate gives one Euclidean strong-parabolicity constant for the transformed
coefficient tensor.

{Using a volume-preserving boundary-flattening chart $\Theta$ of the same type as in
the proof of \Cref{lem:symmetrizable-schauder-bootstrap},} the tensor is
\[
 \mathbb A^{\gamma\delta}(t,\xi)
 =G^{\gamma\delta}(\xi)
   {S_*} A_N(U(t,{\Theta(\xi)})){S_*^{-1}}.
\]
The full first-order parabolic H\"older bound for $U$, smoothness of $A_N$ on
{$\mathcal K$}, and $G\in C^{1+\alpha}$ imply, for every
$\beta_2<\min\{\beta,\alpha\}$,
\[
 \mathbb A\in
 C_{t,x}^{(1+\beta_2)/2,1+\beta_2}
\]
with a uniform norm.  The physical no-flux condition is exactly the
homogeneous conormal condition for this tensor.  Part~(2) of
\Cref{lem:localized-conormal-schauder}, i.e. the $k=2$ H\"older-coefficient
case of \cite[Theorem~4.1]{DongJeon2026}, therefore yields a uniform
$C_{t,x}^{1+\beta_2/2,2+\beta_2}$ estimate on the concentric smaller
cylinder.  The full-cylinder version of the same lemma gives the interior
estimate.

Choose the localization radius also so that the backward cylinders with
upper time at least $a+3\delta$ remain a fixed positive distance from the
initial face.  A finite spatial cover and the usual overlapping backward
cylinders patch the local estimates to
\eqref{eq:moment-C2a-bootstrap}, with constants independent of the distance
to the upper face $b$.  Finally, all coordinate maps are smooth with
uniformly bounded derivatives on the compact state range, so the same
parabolic H\"older estimate transfers to $y,z,$ and $w$.
\end{proof}

\begin{corollary}[Positive-time classicality of maximal-$L^p$ strong solutions]
\label{cor:positive-time-classicality-system}
Assume $p>d+2$ and let $\Omega$ be bounded of class
$C^{2+\alpha}$ for some $0<\alpha<1$.  Let $w$ be a maximal-$L^p$
strong solution on $[0,T_{\max})$.  Then for every
$0<\tau<T<T_{\max}$ there is $\sigma\in(0,\alpha)$ such that
\[
 w,U,y,z\in
 C_{t,x}^{1+\sigma/2,2+\sigma}
 ([\tau,T]\times\overline\Omega).
\]
In particular the solution is classical on every positive-time strip.
\end{corollary}

\begin{proof}
Work on a slightly larger strip $[\tau/2,T']$ with $T<T'<T_{\max}$.
Since $w\in C([\tau/2,T'];X_{\gamma,p})$ and
$X_{\gamma,p}\hookrightarrow C^1(\overline\Omega)$, its pointwise range on
this compact strip is bounded in $\mathcal W=\mathbb R^m$.  The softmax and
moment maps therefore place the associated $y$ and $U$ ranges in compact
subsets of $\mathcal S^\circ$ and $\mathcal T^\circ$, respectively; no
separate compact-state hypothesis is required.  The maximal-regularity class
and $p>d+2$, together with the parabolic Sobolev embedding
\cite[Chapter~II, Lemma~3.3]{LSU1968}, give, for every sufficiently small
\[
 0<\eta<\min\left\{1-\frac{d+2}{p},\alpha\right\},
 \qquad
 w\in C_{t,x}^{(1+\eta)/2,1+\eta}([\tau/2,T']\times\overline\Omega).
\]
Smooth coordinate composition on these compact ranges gives the same estimate
for $U$.  Apply \Cref{lem:moment-second-Schauder-bootstrap} on a finite family of
slightly overlapping sub-strips and choose
$0<\sigma<\min\{\eta,\alpha\}$.  This yields the displayed second-order
estimate on $[\tau,T]$.  Smooth coordinate composition then gives the same
regularity for $w,y,$ and $z$.
\end{proof}

\section{Global strong existence for arbitrary positive data}
\label{sec:global-strong-existence}

The preceding estimates yield the maximal-$L^p$ continuation argument without
any smallness assumption on the initial root range.  {Throughout this section,
``arbitrary positive data'' means {uniformly positive concentration data in the
natural trace class, with no smallness condition on their oscillation or distance from equilibrium}; by
\Cref{rem:initial-trace-equivalence}, this class is exactly equivalent to
$w_0\in X_{\gamma,p}$.}

\begin{remark}[Dependencies of the global continuation argument]
\label{rem:global-proof-dependency-chain}
The continuation argument uses the local solution and family-uniform restart
from \Cref{thm:classical-local-theory}, the invariant root rectangle from
\Cref{cor:strong-sobolev-root-system,thm:root-rectangle}, the root H\"older
estimate from \Cref{sec:multi-EPD-dossier,sec:root-DeGiorgi}, the spatial
Lipschitz estimate of \Cref{lem:symmetrizable-schauder-bootstrap}, and the
terminal trace bound of \Cref{lem:terminal-trace-from-first-Schauder}.  The
second-order conormal $C^{2+\beta}$ bootstrap and the long-time entropy
estimates are not used to exclude a finite maximal existence time.
\end{remark}

\begin{theorem}[Global strong solvability in the additive class]
\label{thm:global-strong-additive}
Let $N\ge2$, $d\ge2$, and
\[
 0<g_1<\cdots<g_N,
 \qquad f_{ij}=g_i+g_j\quad(i\ne j).
\]
Let $\Omega\subset\mathbb R^d$ be bounded and of class $C^{2+\alpha}$ for
some $0<\alpha<1$, and let $p>d+2$.  For every
$w_0\in X_{\gamma,p}$, the entropy-variable problem
\eqref{eq:entropy-system-recalled} has a unique global maximal-$L^p$ strong
solution
\begin{equation}\label{eq:global-strong-class-dossier}
 w\in W^{1,p}_{\rm loc}([0,\infty);X_0)
 \cap L^p_{\rm loc}([0,\infty);X_1)
 \cap C([0,\infty);X_{\gamma,p}).
\end{equation}
The corresponding concentration, root, and moment variables remain in compact
subsets of their open state spaces determined by the initial root rectangle,
and the solution is classical on every positive-time strip.
{Equivalently, the concentration initial data are precisely
$y^0\in B^{2-2/p}_{p,p}(\Omega;\mathbb R^N)$ with
$\sum_i y_i^0=1$, $\inf_{\overline\Omega}y_i^0>0$ for each species, and
$\partial_\nu y_i^0=0$; see
\Cref{cor:global-positive-y-data,rem:initial-trace-equivalence}.}
\end{theorem}

\begin{proof}
Let $[0,T_{\max})$ be the maximal interval supplied by
\Cref{thm:classical-local-theory}.  By
\Cref{rem:entropy-data-compact-range}, the initial roots lie in a compact
rectangle $\mathcal R_0$, the strong root formulation is valid on every positive-time
compact strip, and the invariant-rectangle hypotheses are satisfied.  Hence
\begin{equation}\label{eq:global-proof-root-box}
 z([0,T_{\max})\times\overline\Omega)\subset \mathcal R_0.
\end{equation}
Suppose for contradiction that $T_{\max}<\infty$.  Since the smooth
coordinate map $z\mapsto w$ is bounded on the compact rectangle $\mathcal R_0$,
\eqref{eq:global-proof-root-box} also gives a compact set
$\mathcal K_w\Subset\mathcal W$ such that
\[
 w([0,T_{\max})\times\overline\Omega)\subset\mathcal K_w.
\]
We use the family-uniform restart statement in
\Cref{thm:classical-local-theory}.

Choose $0<\tau_0<T_{\max}/4$.  The strong root formulation verified in
\Cref{rem:entropy-data-compact-range}, together with the compact range
\eqref{eq:global-proof-root-box}, gives the hypotheses of
\Cref{thm:root-DeGiorgi-holder}.  Hence, for some $\theta>0$ and every $\tau_1>\tau_0$,
\begin{equation}\label{eq:global-proof-DG-holder}
 \|z\|_{C_{t,x}^{\theta/2,\theta}
 ([\tau_1,T_{\max})\times\overline\Omega)}\le C.
\end{equation}
Since the map $z\mapsto U$ is smooth on $\mathcal R_0$, the same bound holds for $U$.

Apply the gradient part of
\Cref{lem:symmetrizable-schauder-bootstrap} on a terminal strip with input
H\"older exponent $\theta$ from \eqref{eq:global-proof-DG-holder}.  After
shortening the strip once more,
\begin{equation}\label{eq:global-proof-gradient-bound}
 \sup_{t<T_{\max}}\bigl(
 \|\nabla U(t)\|_{L^\infty}
 +\|\nabla z(t)\|_{L^\infty}
 +\|\nabla w(t)\|_{L^\infty}\bigr)\le C
\end{equation}
on that terminal strip.  The bounds for $z$ and $w$ follow by smooth
coordinate composition on the compact range.  At the same time,
\eqref{eq:global-proof-DG-holder} and the same coordinate changes give
\begin{equation}\label{eq:global-proof-time-holder}
 [w]_{C_t^{\theta/2}(C(\overline\Omega))}\le C
\end{equation}
uniformly up to $T_{\max}$.

Applying \Cref{lem:terminal-trace-from-first-Schauder} with
$\delta_t=\theta/2$, using
\eqref{eq:global-proof-gradient-bound}--\eqref{eq:global-proof-time-holder},
gives a number $M<\infty$ and a late time $t_*$
such that
\[
 \|w(t)\|_{X_{\gamma,p}}\le M,
 \qquad
 w(t,\overline\Omega)\subset\mathcal K_w,
 \qquad t_*\le t<T_{\max}.
\]
Let $\tau_*=\tau_*(\mathcal K_w,M)>0$ be the common lifespan from
\Cref{thm:classical-local-theory}.  Choose
$t_1\in(t_*,T_{\max})$ so close to $T_{\max}$ that
$T_{\max}-t_1<\tau_*/2$.  Restarting the equation from the datum $w(t_1)$
produces a strong solution on $[t_1,t_1+\tau_*]$.  Uniqueness in the local
theory identifies it with the original maximal solution on
$[t_1,T_{\max})$, and hence extends that solution beyond $T_{\max}$, a
contradiction.  Therefore $T_{\max}=\infty$.
Finally, \Cref{cor:positive-time-classicality-system} applied on each finite
positive-time strip gives the asserted classical regularity of the global
solution.
\end{proof}

\begin{corollary}[Concentration-variable formulation]
\label{cor:global-positive-y-data}
Under the assumptions of \Cref{thm:global-strong-additive}, let
$y^0\in B^{2-2/p}_{p,p}(\Omega;\mathbb R^N)$ satisfy
\[
 \sum_{i=1}^Ny_i^0=1,
 \qquad \inf_{x\in\overline\Omega}y_i^0(x)>0,
 \qquad \partial_\nu y_i^0=0\quad(i=1,\ldots,N).
\]
Then the normalized additive Maxwell--Stefan system has a unique global strong
solution with initial value $y^0$, obtained from
\Cref{thm:global-strong-additive} by the entropy coordinate map.  {Uniqueness is understood in the corresponding positive concentration strong
class, equivalently among positive concentration solutions whose entropy transform is a
maximal-$L^p$ strong solution in the sense of \Cref{def:solution-classes}.}
\end{corollary}

\begin{proof}
Put {$\sigma_p:=2-2/p$}. Since $p>d+2$, one has
$B^{{\sigma_p}}_{p,p}(\Omega)\hookrightarrow C^{1+{\eta}}(\overline\Omega)$ for some
{$\eta>0$}.  Uniform positivity therefore places the range of $y^0$ in a
compact subset of $\mathcal S^\circ$.  The moment map $y\mapsto U(y)$ is
affine, and $U\mapsto Dh(U)$ is smooth on the corresponding compact subset of
$\mathcal T^\circ$.  Since $U\mapsto Dh(U)$ is smooth on a neighborhood of
this compact range, the Besov superposition theorem
\cite[Section~5.3]{RunstSickel1996} gives
\[
 w_0:=Dh(U(y^0))\in B^{{\sigma_p}}_{p,p}(\Omega;\mathbb R^{N-1}).
\]
Moreover the componentwise Neumann condition implies
$\partial_\nu U(y^0)=0$, and the classical chain rule, justified by the above
$C^{1+{\eta}}$ embedding, gives
\[
 \partial_\nu w_0
 =H(U(y^0))\,\partial_\nu U(y^0)=0.
\]
Thus $w_0\in X_{\gamma,p}$.  Apply
\Cref{thm:global-strong-additive} and use
\Cref{cor:formulation-equivalence}.  On every finite time interval the
trajectory has compact range in the open state spaces, so smooth composition
preserves the maximal-$L^p$ strong class in both directions; uniqueness in the
stated concentration class follows from uniqueness of the entropy-variable
solution.
\end{proof}

\begin{remark}[Exact equivalence of the initial trace classes]
\label{rem:initial-trace-equivalence}
The preceding implication is reversible.  If $w_0\in X_{\gamma,p}$, then
$w_0\in C^{1+{\eta}}(\overline\Omega)$ and hence has bounded range.  The
softmax map \eqref{eq:y-softmax} sends that range into a compact subset of
$\mathcal S^\circ$, so
\[
 \min_{x\in\overline\Omega} y_i(w_0(x))>0\qquad(i=1,\ldots,N).
\]
Since the softmax map $w\mapsto y(w)$ is smooth on a neighborhood of the
bounded range of $w_0$, the same Besov superposition theorem
\cite[Section~5.3]{RunstSickel1996} gives
$y(w_0)\in B^{2-2/p}_{p,p}(\Omega;\mathbb R^N)$, and
\[
 \partial_\nu y(w_0)=Dy(w_0)\,\partial_\nu w_0=0.
\]
Consequently the positive concentration data in the corollary and the entropy
data $w_0\in X_{\gamma,p}$ are exactly equivalent under the coordinate maps.
\end{remark}

\begin{remark}[Species number and space dimension]
The multi-EPD cancellation and the entropy stabilization are algebraic in the
root variables and do not use the physical space dimension.  Dimension enters
only through the local De Giorgi and weak-conormal estimates and through the
local-theory condition $p>d+2$.  For $N=2$ the multi-EPD potential reduces to the usual
quadratic scalar truncation and the argument collapses to scalar diffusion.
\end{remark}

\begin{remark}[Dependence on the additive structure]
The argument uses three linked consequences of the additive compression
structure: the global interlacing root coordinates, the explicit root system
\eqref{eq:root-system-dossier} with
$C_{jk}=({\mathfrak a_j}+{\mathfrak a_k})/(z_j-z_k)$, and the entropy symmetrizer.  The multi-EPD
identity cancels the mixed terms generated by the root system, while the
mixing entropy controls the higher-species flux terms that remain after
localization.  No analogous global root system is asserted here for general
non-additive $N\ge4$ Maxwell--Stefan coefficients.
\end{remark}

\section{Exponential convergence of global positive solutions}
\label{sec:exponential-relaxation}

Throughout this section, $\Omega\subset\mathbb R^d$ is a bounded domain of class
$C^2$, and $y$ denotes a global solution furnished by
\Cref{thm:global-strong-additive}, continuous in $C^1(\overline\Omega)$ up to its
initial time and classical on every positive-time strip.  For every
$t>0$, the species balances and no-flux boundary condition give
\[
 \frac{d}{dt}\int_\Omega y_i(t,x)\,dx
 =-\int_{\partial\Omega}J_i(t,x)\cdot\nu\,dS=0,
 \qquad i=1,\ldots,N.
\]
Continuity to $t=0$ therefore shows that these masses are conserved for all
$t\ge0$.  Write the conserved species mean as
\[
        \overline y_i=\frac1{|\Omega|}\int_\Omega y_i^0(x)\,dx,
        \qquad i=1,\ldots,N.
\]
Define the relative mixing entropy
\begin{equation}\label{eq:relative-entropy-definition}
 \mathcal H(y\mid\overline y)
 =\int_\Omega\sum_{i=1}^N
 y_i\log\frac{y_i}{\overline y_i}\,dx.
\end{equation}

\begin{lemma}[Uniform entropy equivalence on a compact subset]
\label{lem:relative-entropy-equivalence}
Let $\cK_y\Subset\cS^\circ$ be compact and convex.  There are
$0<c_\cK\le C_\cK<\infty$ such that, for $y,\bar y\in\cK_y$,
\begin{equation}\label{eq:relative-entropy-equivalence}
 c_\cK|y-\bar y|^2
 \le \sum_i y_i\log\frac{y_i}{\bar y_i}
 \le C_\cK|y-\bar y|^2.
\end{equation}
{Equivalently, on the corresponding compact coordinate images the relative
entropy density is comparable to $|z(y)-z(\bar y)|^2$, to $|U(y)-U(\bar y)|^2$, and to
$|w(y)-w(\bar y)|^2$, with constants depending only on the compact set.}
\end{lemma}

\begin{proof}
For fixed $\bar y\in\cK_y$, set
$\phi_{\bar y}(y):=\sum_i y_i\log(y_i/\bar y_i)$.  Its Hessian, restricted to
$E=\{\xi:\sum_i\xi_i=0\}$, is the restriction of
$\operatorname{diag}(1/y_i)$, whose eigenvalues are bounded above and below
on $\cK_y$.  Moreover,
$D\phi_{\bar y}(\bar y)=\one$, and therefore the linear Taylor term vanishes
along the simplex because $\one\cdot(y-\bar y)=0$.  Taylor's formula along the
segment from $\bar y$ to $y$ proves the first statement.  The coordinate
versions follow from uniform bi-Lipschitz bounds on compact sets.
\end{proof}

\begin{lemma}[Logarithmic Sobolev control of the relative mixing entropy]
\label{lem:relative-entropy-log-Sobolev}
Assume that $\Omega$ is bounded and of class $C^2$.  There exists a constant
$C_{\rm LS}=C_{\rm LS}(\Omega)<\infty$ such that for
every nonnegative $\rho$ with $\sqrt\rho\in H^1(\Omega)$ and positive mean
$\bar\rho=|\Omega|^{-1}\int_\Omega\rho\,dx$,
\begin{equation}\label{eq:bounded-domain-log-Sobolev}
 \int_\Omega \rho\log\frac{\rho}{\bar\rho}\,dx
 \le C_{\rm LS}(\Omega)\int_\Omega|\nabla\sqrt\rho|^2\,dx.
\end{equation}
Consequently, for every simplex-valued $y$ whose species means $\bar y_i$ are
positive,
\begin{equation}\label{eq:vector-relative-entropy-LSI}
 \mathcal H(y\mid\bar y)
 \le C_{\rm LS}(\Omega)
 \sum_{i=1}^N\|\nabla\sqrt{y_i}\|_2^2.
\end{equation}
\end{lemma}

\begin{proof}
We include the short tightening argument because it makes the precise domain
hypotheses and the low-dimensional cases explicit.  It suffices first to
argue for smooth $\varphi$; the general $H^1$ case follows by Sobolev density,
truncation, and lower semicontinuity of entropy.  Put
$d\mu=|\Omega|^{-1}\,dx$, write {$\langle\varphi\rangle_\mu:=\int_\Omega\varphi\,d\mu$}, and, for $h\ge0$,
denote by $\operatorname{Ent}_\mu$ the entropy functional with respect to $\mu$:
\[
 \operatorname{Ent}_\mu(h)
 :=\int_\Omega h\log\frac{h}{\int h\,d\mu}\,d\mu.
\]
Connectedness and bounded $C^2$ regularity give the Poincar\'e inequality
\begin{equation}\label{eq:LSI-Poincare-normalized}
 \|\varphi-{\langle\varphi\rangle_\mu}\|_{L^2(\mu)}^2
 \le P_\Omega\|\nabla\varphi\|_{L^2(\mu)}^2
\end{equation}
and a Sobolev inequality, for some finite {$q_{\rm S}>2$},
\begin{equation}\label{eq:LSI-Sobolev-normalized}
 \|\varphi\|_{L^{{q_{\rm S}}}(\mu)}^2
 \le S_1\|\nabla\varphi\|_{L^2(\mu)}^2
      +S_2\|\varphi\|_{L^2(\mu)}^2.
\end{equation}
For $d\ge3$ one may take {$q_{\rm S}=2d/(d-2)$}, while for $d=2$ any fixed finite
{$q_{\rm S}>2$} is admissible.  Enlarge $S_2$ if necessary so that $S_2\ge1$.

First assume $\|\varphi\|_{L^2(\mu)}=1$.  On the set where
$\varphi\ne0$ write
\[
 \operatorname{Ent}_\mu(\varphi^2)
 =\frac{2}{{q_{\rm S}}-2}
   \int_\Omega \varphi^2\log |\varphi|^{{q_{\rm S}}-2}\,d\mu .
\]
Jensen's inequality for the probability measure $\varphi^2\,d\mu$ therefore gives
\[
 \operatorname{Ent}_\mu(\varphi^2)
 \le \frac{2}{{q_{\rm S}}-2}\log\int_\Omega |\varphi|^{{q_{\rm S}}}\,d\mu
 = \frac{{q_{\rm S}}}{{q_{\rm S}}-2}\log\|\varphi\|_{L^{{q_{\rm S}}}(\mu)}^2
 \le \frac{{q_{\rm S}}}{{q_{\rm S}}-2}
       \bigl(\|\varphi\|_{L^{{q_{\rm S}}}(\mu)}^2-1\bigr).
\]
The formula follows for functions with zeros by the usual $\varepsilon$-regularization.
By homogeneity and \eqref{eq:LSI-Sobolev-normalized}, this yields the defective
logarithmic Sobolev estimate
\begin{equation}\label{eq:defective-LSI-bounded-domain}
 \operatorname{Ent}_\mu(\varphi^2)
 \le A_\Omega\|\nabla\varphi\|_{L^2(\mu)}^2
     +B_\Omega\|\varphi\|_{L^2(\mu)}^2.
\end{equation}
Now set {$\psi:=\varphi-\langle\varphi\rangle_\mu$}, where
$\langle\varphi\rangle_\mu:=\int_\Omega\varphi\,d\mu$.  The Rothaus centering inequality
\begin{equation}\label{eq:Rothaus-centering-bounded-domain}
 \operatorname{Ent}_\mu(\varphi^2)
 \le \operatorname{Ent}_\mu({\psi^2})+2\|{\psi}\|_{L^2(\mu)}^2
\end{equation}
{is implied by Rothaus's centering lemma; see \cite[Lemma~9]{Rothaus1985}.
The same centering step is also used in \cite[Lemma~1]{DesvillettesFellner2014}.}  Applying
\eqref{eq:defective-LSI-bounded-domain} to {$\psi$} and then
\eqref{eq:LSI-Poincare-normalized} gives
\[
 \operatorname{Ent}_\mu(\varphi^2)
 \le \bigl(A_\Omega+(B_\Omega+2)P_\Omega\bigr)
       \|\nabla\varphi\|_{L^2(\mu)}^2.
\]
Thus the bounded-domain logarithmic Sobolev constant is finite and depends
only on the domain (and the fixed Sobolev exponent chosen above).

Take $\varphi=\sqrt\rho$.  Since
$\int\varphi^2\,d\mu=\bar\rho$,
\[
 \operatorname{Ent}_\mu(\varphi^2)
 =\frac1{|\Omega|}\int_\Omega
   \rho\log\frac{\rho}{\bar\rho}\,dx.
\]
Multiplying the preceding inequality by $|\Omega|$ proves
\eqref{eq:bounded-domain-log-Sobolev}.  Applying it species by species and
summing proves \eqref{eq:vector-relative-entropy-LSI}.
\end{proof}

\begin{theorem}[Exponential convergence of global positive solutions]
\label{thm:exponential-relaxation}
Under the standing assumptions of this section, set
\[
 \mathcal R_0:=\prod_{j=1}^{N-1}
 \left[\min_{\overline\Omega}z_j(0,\cdot),
       \max_{\overline\Omega}z_j(0,\cdot)\right].
\]
Then
\[
 \mathcal R_0\Subset\prod_{j=1}^{N-1}(g_j,g_{j+1}),
 \qquad
 z([0,\infty)\times\overline\Omega)\subset \mathcal R_0.
\]
Let
\begin{equation}\label{eq:kappa-LS-structural}
 \kappa_{\rm LS}:=
 \frac{4}{f_{\max}C_{\rm LS}(\Omega)}>0.
\end{equation}
Then
\begin{equation}\label{eq:relative-entropy-exponential}
        \mathcal H(y(t)\mid\overline y)
        \le e^{-\kappa_{\rm LS}t}\mathcal H(y^0\mid\overline y),
        \qquad t\ge0.
\end{equation}
Moreover, there is $C_{\mathcal R_0}<\infty$ such that
\begin{equation}\label{eq:L2-exponential-general}
\begin{aligned}
 &\|y(t)-\overline y\|_{L^2}
 +\|z(t)-z(\overline y)\|_{L^2}
 +\|U(t)-U(\overline y)\|_{L^2}
 +\|w(t)-w(\overline y)\|_{L^2}\\
 &\hspace{20mm}\le
 C_{\mathcal R_0}e^{-\kappa_{\rm LS}t/2}\mathcal H(y^0\mid\overline y)^{1/2}.
\end{aligned}
\end{equation}
The entropy decay rate $\kappa_{\rm LS}$ depends only on $f_{\max}$ and the
domain logarithmic-Sobolev constant, not on the distance of the initial state
from the concentration boundary.
\end{theorem}

\begin{proof}
Since $y(0,\cdot)$ is continuous on the compact set $\overline\Omega$ and takes
values in $\mathcal S^\circ$, its root range is compactly contained in the
interlacing box.  Hence $\mathcal R_0$ is a compact root rectangle, and
\Cref{thm:root-rectangle} gives
$z([0,\infty)\times\overline\Omega)\subset \mathcal R_0$.  The rectangle generates a
compact concentration set; its convex hull remains compactly contained in the
simplex and contains $\overline y$.  Conservation of the species means implies
that the relative entropy has the same derivative as the mixing entropy.  For
the positive-time classical solution, apply the entropy identity on
$[\varepsilon,t]$ and let $\varepsilon\downarrow0$ using the continuity of $y$
and of the relative entropy at $t=0$.  The entropy identity and the boundary-uniform Fisher-information estimate
\Cref{cor:boundary-uniform-Fisher} give
\[
 \frac{d}{dt}\mathcal H(y(t)\mid\overline y)+\mathcal D(t)=0,
 \qquad
 \mathcal D(t)\ge
 \frac{4}{f_{\max}}\sum_{i=1}^N\|\nabla\sqrt{y_i(t)}\|_2^2.
\]
By \Cref{lem:relative-entropy-log-Sobolev},
\[
 \mathcal D(t)\ge
 \frac{4}{f_{\max}C_{\rm LS}(\Omega)}
 \mathcal H(y(t)\mid\overline y)
 =\kappa_{\rm LS}\mathcal H(y(t)\mid\overline y).
\]
Gronwall's lemma proves \eqref{eq:relative-entropy-exponential}.  The $L^2$
control of the concentrations does not require compact separation from the
simplex boundary either.  Indeed, for $0<a,b\le1$ the scalar Bregman density
$\psi(a\mid b)=a\log(a/b)-a+b$ satisfies
$\psi''(a)=1/a\ge1$ and therefore
$\psi(a\mid b)\ge\frac12(a-b)^2$.  Since
$\sum_i(-y_i+\bar y_i)=0$ pointwise,
\[
 \mathcal H(y\mid\bar y)
 =\int_\Omega\sum_i\psi(y_i\mid\bar y_i)\,dx
 \ge\frac12\|y-\bar y\|_2^2.
\]
This gives the asserted $L^2$ decay for $y$ with the rate
$\kappa_{\rm LS}/2$.  For the remaining coordinates, let
\[
 K_y:=\operatorname{co}(y(\mathcal R_0)).
\]
Because $y(\mathcal R_0)\Subset\mathcal S^\circ$ and the simplex is convex,
$K_y\Subset\mathcal S^\circ$.  It contains every value $y(t,x)$ and also the
spatial mean $\overline y$.  The maps $y\mapsto z(y),U(y),w(y)$ therefore
have uniformly bounded derivatives on $K_y$.  Their Lipschitz bounds give the
remaining $L^2$ estimates without any unstated assumption that
$z(\overline y)$ itself belongs to the rectangular set $\mathcal R_0$.
\end{proof}

\begin{proposition}[Uniform positive-time smoothing of global solutions]
\label{prop:uniform-positive-time-smoothing}
Let $w$ be a global strong solution supplied by
\Cref{thm:global-strong-additive}.  {For every $\tau>0$ there exist
$\beta\in(0,\alpha)$ and $C_\tau<\infty$, depending only on
$\tau$, $\mathcal R_0$, $d$, $N$, $g$, $\Omega$, and $\alpha$, such that}
\begin{equation}\label{eq:uniform-global-C2beta-tail}
\begin{aligned}
 &\|z\|_{C_{t,x}^{1+\beta/2,2+\beta}
 ([\tau,\infty)\times\overline\Omega)}
 +\|y\|_{C_{t,x}^{1+\beta/2,2+\beta}
 ([\tau,\infty)\times\overline\Omega)}\\
 &\qquad
 +\|U\|_{C_{t,x}^{1+\beta/2,2+\beta}
 ([\tau,\infty)\times\overline\Omega)}
 +\|w\|_{C_{t,x}^{1+\beta/2,2+\beta}
 ([\tau,\infty)\times\overline\Omega)}
 \le C_\tau.
\end{aligned}
\end{equation}
Here $\mathcal R_0$ is the invariant initial root rectangle.
\end{proposition}

\begin{proof}
By \Cref{thm:global-strong-additive}, the whole trajectory stays in the fixed
compact root rectangle $\mathcal R_0$.  Apply \Cref{thm:root-DeGiorgi-holder} with
$T=\infty$ on $[\tau/4,\infty)$; this gives one exponent $\theta>0$ and a
uniform $C_{t,x}^{\theta/2,\theta}$ bound on $[\tau/2,\infty)$.  The constants
depend only on the fixed root box, the structural coefficients, the domain,
and the positive distance from the initial time.

Choose a fixed number
$\delta_\tau>0$ with $16\delta_\tau<\min\{\tau,1\}$ and define, for
$n=0,1,2,\ldots$,
\[
 {t_n^-}:=\tau-5\delta_\tau+6n\delta_\tau,
 \qquad {t_n^+}:={t_n^-}+12\delta_\tau,
 \qquad I_n:=({t_n^-},{t_n^+}).
\]
Then ${t_0^-}>\tau/2$, consecutive outer intervals overlap by
$6\delta_\tau$, and all $I_n$ lie in the strip on which the uniform De Giorgi
bound is available.  On every $I_n$ the hypotheses of
\Cref{lem:symmetrizable-schauder-bootstrap} hold with the same compact state
set, the same H\"older bound, and the same choice $\delta=\delta_\tau$.
Hence that lemma gives one uniform
$C_{t,x}^{(1+\beta_0)/2,1+\beta_0}$ bound on
$[{t_n^-}+2\delta_\tau,{t_n^+})$, where
$0<\beta_0<\min\{\theta,\alpha\}$.  This inner interval has length
$10\delta_\tau>4\delta_\tau$, so
\Cref{lem:moment-second-Schauder-bootstrap}, applied with the same
$\delta_\tau$, gives a uniform
$C_{t,x}^{1+\beta/2,2+\beta}$ estimate on
\[
 [{t_n^-}+5\delta_\tau,{t_n^+})
 =[\tau+6n\delta_\tau,\,\tau+(6n+7)\delta_\tau)
\]
for every $0<\beta<\min\{\beta_0,\alpha\}$.  These final intervals have
length $7\delta_\tau$ and their starting points are separated by only
$6\delta_\tau$; hence they overlap and cover $[\tau,\infty)$ exactly.  The
constants are independent of $n$ because all coefficient, root-gap, and
H\"older bounds come from the same compact root rectangle.  Finally, all coordinate maps are smooth with
uniformly bounded derivatives of every required order on that compact set,
so composition transfers the same estimate to $y,U,w$.  Patching neighboring
windows gives the global tail norm in
\eqref{eq:uniform-global-C2beta-tail}; space--time pairs separated by more than
one window are controlled by the uniform supremum bounds.
\end{proof}

\begin{lemma}[$L^2$--$C^1$ interpolation on a bounded domain]
\label{lem:L2-C1-GN}
Let $\Omega\subset\mathbb R^d$ be bounded and of class $C^2$.  There is
$C_\Omega<\infty$ such that every
{$v\in L^2(\Omega)\cap W^{2,\infty}(\Omega)$} satisfies
\begin{equation}\label{eq:GN-C1-from-L2-C2}
 \|{v}\|_{W^{1,\infty}(\Omega)}
 \le C_\Omega
 \|{v}\|_{L^2(\Omega)}^{\frac{2}{d+4}}
 \|{v}\|_{W^{2,\infty}(\Omega)}^{\frac{d+2}{d+4}}
 +C_\Omega\|{v}\|_{L^2(\Omega)}.
\end{equation}
\end{lemma}

\begin{proof}
{By the standard universal Sobolev extension theorem for bounded Lipschitz
(and hence $C^2$) domains, there is a single linear extension operator $E$ that is bounded on
both $L^2(\Omega)$ and $W^{2,\infty}(\Omega)$; see, for example,
\cite[Chapter~VI]{Stein1970Extension}.  Multiplying $Ev$ by a fixed smooth cutoff that equals one
near $\overline\Omega$ gives a compactly supported extension $\widetilde v$ with}
\[
 \|{\widetilde v}\|_{L^2(\mathbb R^d)}\le C_\Omega\|{v}\|_{L^2(\Omega)},
 \qquad
 \|{\widetilde v}\|_{W^{2,\infty}(\mathbb R^d)}
 \le C_\Omega\|{v}\|_{W^{2,\infty}(\Omega)}.
\]
{Thus it is enough to prove the estimate on $\mathbb R^d$ for this compactly
supported extension.}  Let $\eta$ be a smooth radial mollifier and
$\eta_r(x)=r^{-d}\eta(x/r)$.  Put
$L=\|{\widetilde v}\|_2$ and
$M=\|{\widetilde v}\|_{W^{2,\infty}}$.  For $0<r\le1$, Young's inequality and
the Lipschitz continuity of $\nabla{\widetilde v}$ give
\begin{equation}\label{eq:mollifier-gradient-interpolation}
 \|\nabla{\widetilde v}\|_\infty
 \le \|\nabla\eta_r*{\widetilde v}\|_\infty
     +\|\nabla{\widetilde v}-\eta_r*\nabla{\widetilde v}\|_\infty
 \le C\bigl(r^{-1-d/2}L+rM\bigr).
\end{equation}
If $0<L\le M$, choose
$r=(L/M)^{2/(d+4)}$.  The two terms in
\eqref{eq:mollifier-gradient-interpolation} then have the same size and give
\[
 \|\nabla{\widetilde v}\|_\infty
 \le C L^{2/(d+4)}M^{(d+2)/(d+4)}.
\]
If $L>M$, take $r=1$ and obtain instead
$\|\nabla{\widetilde v}\|_\infty\le C L$, which is the lower-order term in
\eqref{eq:GN-C1-from-L2-C2}.  The case $L=0$ is trivial.

For the zeroth-order term, radiality gives
$\int x\eta(x)\,dx=0$, and Taylor's formula yields
\[
 \|{\widetilde v}\|_\infty
 \le \|\eta_r*{\widetilde v}\|_\infty
      +\|{\widetilde v}-\eta_r*{\widetilde v}\|_\infty
 \le C\bigl(r^{-d/2}L+r^2M\bigr).
\]
With the same optimizing radius this is
$CL^{4/(d+4)}M^{d/(d+4)}$ when $L\le M$, which is no larger than the
right-hand leading term of \eqref{eq:GN-C1-from-L2-C2}; when $L>M$ it is
again bounded by $CL$.  Restriction to $\Omega$ proves the lemma.
\end{proof}

\begin{corollary}[Exponential convergence in $C^1$]
\label{cor:C1-exponential-convergence}
Under the assumptions of \Cref{thm:exponential-relaxation}, suppose in
addition that the solution is one of the global strong solutions of
\Cref{thm:global-strong-additive}.  Then there are constants
$C<\infty$ and $\kappa_1>0$ such that
\begin{equation}\label{eq:C1-exponential-general}
\begin{aligned}
 &\|y(t)-\overline y\|_{C^1(\overline\Omega)}
 +\|z(t)-z(\overline y)\|_{C^1(\overline\Omega)}\\
 &\qquad
 +\|U(t)-U(\overline y)\|_{C^1(\overline\Omega)}
 +\|w(t)-w(\overline y)\|_{C^1(\overline\Omega)}
 \le C e^{-\kappa_1 t},
 \qquad t\ge0.
\end{aligned}
\end{equation}
One may take $\kappa_1=\kappa_{\rm LS}/(d+4)$.
\end{corollary}

\begin{proof}
If the initial relative entropy is zero, then $y^0\equiv\overline y$ and the
claim is immediate.  Otherwise, \Cref{prop:uniform-positive-time-smoothing},
with for instance $\tau=1/2$, gives
\[
 \sup_{t\ge1}\|y(t)-\overline y\|_{W^{2,\infty}(\Omega)}\le M_2<\infty.
\]
Apply \Cref{lem:L2-C1-GN} componentwise to
{$v_t:=y(t)-\overline y$}.  By \eqref{eq:L2-exponential-general},
\[
 \|{v_t}\|_{L^2}\le C_{\mathcal R_0}
 e^{-\kappa_{\rm LS}t/2}\mathcal H(y^0\mid\overline y)^{1/2}.
\]
Using the uniform $W^{2,\infty}$ bound in
\eqref{eq:GN-C1-from-L2-C2} therefore yields, for $t\ge1$,
\[
 \|y(t)-\overline y\|_{C^1}
 \le C e^{-\kappa_{\rm LS}t/(d+4)}.
\]
The finite interval $0\le t\le1$ is absorbed by enlarging $C$, since the
strong solution is continuous into $C^1(\overline\Omega)$.  On the compact concentration range containing both $y(t)$ and $\overline y$,
the maps $y\mapsto z,U,w$ have uniformly bounded first derivatives.  Since
$\overline y$ is spatially constant, the mean-value theorem and the spatial
chain rule give
\[
 \|z(t)-z(\overline y)\|_{C^1}
 +\|U(t)-U(\overline y)\|_{C^1}
 +\|w(t)-w(\overline y)\|_{C^1}
 \le C\|y(t)-\overline y\|_{C^1}.
\]
Hence the same exponential rate holds for the remaining coordinate systems.  This proves
\eqref{eq:C1-exponential-general}.
\end{proof}

\begin{corollary}[{Exponential convergence for global strong solutions}]
\label{cor:global-strong-exponential}
{Every global strong solution furnished either by
\Cref{thm:global-strong-additive} from entropy-variable data or by
\Cref{cor:global-positive-y-data} from positive concentration data converges exponentially to the
homogeneous equilibrium determined by its conserved species masses.  More precisely,
\eqref{eq:relative-entropy-exponential}--\eqref{eq:L2-exponential-general} hold, and the stronger
$C^1$ estimate \eqref{eq:C1-exponential-general} holds as well.}
\end{corollary}

\begin{proof}
{For the entropy-variable formulation, the trace embedding
\eqref{eq:trace-holder-embedding} and continuity in the global strong class give
$w\in C([0,\infty);C^1(\overline\Omega))$.  The softmax and root maps then give
$y,z\in C([0,\infty);C^1(\overline\Omega))$, and
\Cref{rem:entropy-data-compact-range} places the initial concentration range compactly inside
$\mathcal S^\circ$.  Positive-time classical regularity follows from
\Cref{cor:positive-time-classicality-system}, so the hypotheses of
\Cref{thm:exponential-relaxation,cor:C1-exponential-convergence} are satisfied.  The concentration-data
formulation is equivalent by \Cref{cor:global-positive-y-data,rem:initial-trace-equivalence}; hence it
has the same conclusion.}
\end{proof}

\section{Conclusion}

{For the normalized Maxwell--Stefan system with
$f_{ij}=g_i+g_j$, $0<g_1<\cdots<g_N$, $d\ge2$, and $p>d+2$, we proved global
strong solvability on bounded $C^{2+\alpha}$ domains for arbitrary entropy
data in the natural trace space, equivalently for arbitrary uniformly positive
concentration data in the corresponding trace class with homogeneous Neumann
compatibility.  No smallness assumption is imposed on the initial oscillation.
The additive structure turns the positive constrained relaxation operator
into a shifted compression of the diagonal species operator; its scalar
resolvent yields the explicit constrained inverse and the global interlacing
root coordinates, while its Laurent data lead to the closed affine moment
system symmetrized by the Hessian of the mixing entropy.  The root equations admit forward-invariant
rectangles, which keep the concentration range in a compact subset of the open
simplex.}

Entropy-stabilized multi-EPD truncations yield De Giorgi H\"older
continuity of the roots.  Combined with entropy symmetrization, conormal
regularity, and short-interval maximal regularity, this gives the continuation
estimates required for global strong solvability.  The global solution has
positive-time $C^{2+\beta}$ regularity for some $0<\beta<\alpha$.  The
Fisher-information estimate and logarithmic Sobolev inequality yield
exponential convergence to the homogeneous equilibrium determined by the
conserved species masses in relative entropy, $L^2$, and
$C^1(\overline\Omega)$.

Extensions to repeated values among the $g_i$, to general non-additive coefficients for $N\ge4$, or to
strong solutions whose initial state lies on the boundary of the concentration
simplex require separate analysis and are not claimed here.

\section*{Acknowledgments}
It is with deep gratitude and fondness that I remember the many fruitful
discussions I had with the late Jan Pr\"uss on Maxwell--Stefan diffusion in
reactive and nonreactive multicomponent fluid systems.  The idea of additive
friction coefficients emerged from these discussions.\vskip3mm

\noindent
{ChatGPT (OpenAI) was used during manuscript preparation for editorial
assistance, consistency checks, and alternative formulations.  The author is solely responsible for
the mathematical content and final text.}

\section*{Statements and Declarations}
\noindent\textbf{Funding.}
The author declares that no funds, grants, or other support were received during the preparation of this manuscript.

\smallskip
\noindent\textbf{Competing interests.}
The author declares no competing interests.

\smallskip
\noindent\textbf{Data availability.}
Data sharing is not applicable to this article, as no datasets were generated
or analyzed during the study.

\end{document}